\documentclass[10pt, openany]{book}
\usepackage{amsmath, amsthm, amssymb, makeidx, fancyhdr}

\newtheorem{theorem}{Theorem}[chapter]
\newtheorem{proposition}[theorem]{Proposition}
\newtheorem{corollary}[theorem]{Corollary}
\newtheorem{lemma}[theorem]{Lemma}

\newtheorem{definition}[theorem]{Definition}
\newtheorem*{definition-notag}{Definition}

\theoremstyle{remark}
\newtheorem{remark}[theorem]{Remark}

\newtheorem{example}[theorem]{Example}

\DeclareMathOperator{\Var}{Var}
\DeclareMathOperator{\Cov}{Cov}

\DeclareMathOperator{\supp}{supp}

\DeclareMathOperator{\vol}{vol}
\DeclareMathOperator{\Median}{Median}
\DeclareMathOperator{\tr}{tr}
\DeclareMathOperator{\rank}{rank}

\def \< {\langle}
\def \> {\rangle}

\def \N {\mathbb{N}}
\def \R {\mathbb{R}}
\def \C {\mathbb{C}}

\def \E {\mathbb{E}}
\def \P {\mathbb{P}}

\def \EE {\mathcal{E}}
\def \NN {\mathcal{N}}
\def \TT {\mathcal{T}}
\def \a {\alpha}

\def \e {\varepsilon}
\def \d {\delta}
\def \l {\lambda}
\def \s {\sigma}
\def \S {\Sigma}

\def \psione {{\psi_1}}
\def \psitwo {{\psi_2}}
\def \smin {s_{\min}}
\def \smax {s_{\max}}

\makeindex

\begin{document}

\title{Introduction to the non-asymptotic analysis of random matrices}

\author{Roman Vershynin\footnote{Partially supported by NSF grants FRG DMS 0918623, DMS 1001829} \\
University of Michigan\\
\texttt{romanv@umich.edu}
}

\date{
  \small{{\bf Chapter~5} of: Compressed Sensing, Theory and Applications. Edited by Y.~Eldar and G.~Kutyniok. Cambridge University Press, 2012.
    pp.~210--268.}\\
  \qquad \vspace{4cm} \\ 
  August 11, 2010; final revision November 23, 2011
  }

\maketitle

\tableofcontents 

\clearpage

\noindent This is a tutorial on some basic non-asymptotic methods and concepts in random matrix theory. 
The reader will learn several tools for the analysis of the extreme singular values of random matrices 
with independent rows or columns. Many of these methods sprung off from the development of 
geometric functional analysis since the 1970's. They have applications in several fields, 
most notably in theoretical computer science, statistics and signal processing. 
A few basic applications are covered in this text, particularly for the problem of estimating covariance matrices 
in statistics and for validating probabilistic constructions of measurement matrices in compressed sensing. 
These notes are written particularly for graduate students and beginning researchers in different areas, 
including functional analysts, probabilists, theoretical statisticians, electrical engineers, 
and theoretical computer scientists.

\setcounter{chapter}{5}

\section{Introduction}					\label{s: introduction}

\paragraph{Asymptotic and non-asymptotic regimes}
Random matrix theory studies properties of $N \times n$ matrices $A$ chosen from some distribution
on the set of all matrices. As dimensions $N$ and $n$ grow to infinity, one observes that the spectrum of $A$ tends to 
stabilize. This is manifested in several {\em limit laws}, which may be regarded as random matrix versions of the 
central limit theorem. Among them is Wigner's semicircle law for the eigenvalues of symmetric Gaussian
matrices, the circular law for Gaussian matrices, the Marchenko-Pastur law for Wishart matrices $W = A^*A$ where
$A$ is a Gaussian matrix, the Bai-Yin and Tracy-Widom laws for the extreme eigenvalues of Wishart matrices $W$.
The books \cite{Mehta, AGZ, Deift-Gioev, Bai-Silverstein} offer thorough introduction to the classical 
problems of random matrix theory and its fascinating connections.  

The asymptotic regime where the dimensions $N,n \to \infty$ is well suited for the purposes of 
statistical physics, e.g. when random matrices serve as finite-dimensional models of 
infinite-dimensional operators. But in some other areas including statistics,
geometric functional analysis, and compressed sensing, the limiting regime may not be very useful \cite{RV ICM}. 
Suppose, for example, that we ask about the largest singular value $\smax(A)$
(i.e. the largest eigenvalue of $(A^*A)^{1/2}$); to be specific assume that $A$ is an $n \times n$ matrix
whose entries are independent standard normal random variables. 
The asymptotic random matrix theory answers this question as follows: 
the Bai-Yin law (see Theorem~\ref{Bai-Yin}) states that
$$
\smax(A) / 2\sqrt{n} \to 1	\quad \text{almost surely}
$$ 
as the dimension $n \to \infty$.
Moreover, the limiting distribution of $\smax(A)$ is known to be the Tracy-Widom law
(see \cite{Soshnikov, FeSo}).
In contrast to this, a non-asymptotic answer to the same question is the following: 
in {\em every} dimension $n$, one has 
$$
\smax(A) \le C\sqrt{n} \quad \text{with probability at least } 1 - e^{-n},
$$
 here $C$
is an absolute constant (see Theorems~\ref{Gaussian} and \ref{sub-gaussian rows}). 
The latter answer is less precise (because of an absolute constant $C$) but more quantitative
because for fixed dimensions $n$ it gives an exponential probability of success.\footnote{For 
this specific model (Gaussian matrices),Theorems~\ref{Gaussian} and \ref{Gaussian deviation}
even give a sharp absolute constant $C\approx 2$ here. But the 
result mentioned here is much more general as we will see later; it only requires
independence of rows or columns of $A$.}
This is the kind of answer we will seek in this text --
guarantees up to absolute constants in all dimensions, and with large probability.

\paragraph{Tall matrices are approximate isometries}
The following heuristic will be our guideline:
{\em tall random matrices should act as approximate isometries}.
So, an $N \times n$ random matrix $A$ with $N \gg n$ should act
almost like an isometric embedding of $\ell_2^n$ into $\ell_2^N$:
$$
(1-\d) K \|x\|_2 \le \|Ax\|_2 \le (1+\d) K \|x\|_2 \quad \text{for all } x \in \R^n
$$
where $K$ is an appropriate normalization factor and $\d \ll 1$. 
Equivalently, this says that all the singular values of $A$ are close to each other: 
$$
(1-\d)K \le \smin(A) \le \smax(A) \le (1+\d)K,
$$
where $\smin(A)$ and $\smax(A)$ denote the smallest and the largest singular values of $A$. 
Yet equivalently, this means that tall matrices are well conditioned: the {\em condition number}
\index{Condition number}
of $A$ is $\kappa(A) = \smax(A)/\smin(A) \le (1+\d)/(1-\d) \approx 1$. 

In the asymptotic regime and for random matrices with independent entries, our heuristic
is justified by Bai-Yin's law, which is Theorem~\ref{Bai-Yin} below. 
Loosely speaking, it states that as the dimensions $N,n$ increase to infinity while the aspect ratio $N/n$ is fixed,
we have
\begin{equation}							\label{Bai-Yin heuristic}
 \sqrt{N} - \sqrt{n} \approx \smin(A) \le \smax(A) \approx \sqrt{N} + \sqrt{n}.
\end{equation}
In these notes, we study $N \times n$ random matrices $A$ with independent rows or independent columns,
but not necessarily independent entries. 
We develop non-asymptotic versions of \eqref{Bai-Yin heuristic} for such matrices,
which should hold for all dimensions $N$ and $n$. The desired results should have the form 
\begin{equation}							\label{heuristic}
\sqrt{N} - C \sqrt{n} \le \smin(A) \le \smax(A) \le \sqrt{N} + C \sqrt{n}
\end{equation}
with large probability, e.g. $1-e^{-N}$, where $C$ is an absolute constant.\footnote{More accurately,  
we should expect $C=O(1)$ to depend on easily computable quantities of the distribution, 
such as its moments. This will be clear from the context.} 
For tall matrices, where $N \gg n$, both sides of this inequality
would be close to each other, which would guarantee that $A$ is an approximate isometry.

\paragraph{Models and methods}
We shall study quite general models of random matrices -- those with independent rows
or independent columns that are sampled from high-dimensional distributions. We will place either  
strong moment assumptions on the distribution (sub-gaussian growth of moments), 
or no moment assumptions at all (except finite variance). This leads us
to four types of main results: 

\begin{enumerate} \setlength{\itemsep}{-3pt}
  \item Matrices with independent sub-gaussian rows: Theorem~\ref{sub-gaussian rows}
  \item Matrices with independent heavy-tailed rows: Theorem~\ref{heavy-tailed rows} 
  \item Matrices with independent sub-gaussian columns: Theorem~\ref{sub-gaussian columns} 
  \item Matrices with independent heavy-tailed columns: Theorem~\ref{heavy-tailed columns} 
\end{enumerate}

These four models cover many natural classes of random matrices that occur in applications, 
including random matrices with independent entries (Gaussian and Bernoulli in particular)
and random sub-matrices of orthogonal matrices (random Fourier matrices in particular). 

The analysis of these four models is based on a variety of tools of probability theory 
and geometric functional analysis, most of which have not been covered in the texts on the
``classical'' random matrix theory. The reader will learn basics on sub-gaussian 
and sub-exponential random variables,
isotropic random vectors, large deviation inequalities for sums of independent random variables,
extensions of these inequalities to random matrices, and several basic methods of high dimensional probability 
such as symmetrization, decoupling, and covering ($\e$-net) arguments.

\paragraph{Applications}
In these notes we shall emphasize two applications, one in statistics and one in compressed sensing. 
Our analysis of random matrices with independent rows immediately applies to
a basic problem in statistics -- {\em estimating covariance matrices} of high-dimensional 
distributions. If a random matrix $A$ has i.i.d. rows $A_i$, then
$A^*A = \sum_i A_i \otimes A_i$ is the {\em sample covariance matrix}. If $A$ has
independent columns $A_j$, then $A^*A = (\< A_j, A_k\> )_{j,k}$ is the {\em Gram matrix}. 
Thus our analysis of the row-independent and column-independent models can be interpreted as a
study of sample covariance matrices and Gram matrices of high dimensional distributions.
We will see in Section~\ref{s: covariance} that for a general distribution in $\R^n$, its covariance matrix 
can be estimated from a sample of size $N = O(n \log n)$ drawn from the distribution. 
Moreover, for sub-gaussian distributions we have an even better bound $N = O(n)$.
For low-dimensional distributions, much fewer samples are needed -- if a distribution 
lies close to a subspace of dimension $r$ in $\R^n$, then a sample of size $N = O(r \log n)$ 
is sufficient for covariance estimation.

In compressed sensing, the best known measurement matrices are random. A sufficient
condition for a matrix to succeed for the purposes of compressed sensing
is given by the {\em restricted isometry property}. Loosely speaking, this property 
demands that all sub-matrices of given size be well-conditioned. This fits well in the circle of 
problems of the non-asymptotic random matrix theory. Indeed, we will see in Section~\ref{s: restricted isometries} 
that all basic models 
of random matrices are nice restricted isometries. These include Gaussian and Bernoulli matrices,
more generally all matrices with sub-gaussian independent entries, and even more generally
all matrices with sub-gaussian independent rows or columns. Also, the class of restricted
isometries includes random Fourier matrices, more generally random sub-matrices of 
bounded orthogonal matrices, and even more generally matrices whose rows are independent
samples from an isotropic distribution with uniformly bounded coordinates. 

\paragraph{Related sources}
This text is a tutorial rather than a survey, so we focus on explaining methods rather than results.
This forces us to make some concessions in our choice of the subjects. 
{\em Concentration of measure} and its applications to random matrix theory are only briefly mentioned. 
For an introduction into concentration of measure suitable for a beginner, 
see \cite{Ball} and \cite[Chapter~14]{Matousek}; 
for a thorough exposition see \cite{MS, Ledoux};
for connections with random matrices see \cite{DS, Ledoux extremal}. The monograph \cite{Ledoux-Talagrand}
also offers an introduction into concentration of measure and related probabilistic methods in 
analysis and geometry, some of which we shall use in these notes.

We completely avoid the important (but more difficult) model of {\em symmetric random matrices} 
with independent entries on and above the diagonal. 
Starting from the work of F\"uredi and Komlos \cite{FuKo},
the largest singular value (the spectral norm) of symmetric random matrices has been a subject 
of study in many works; see e.g. \cite{Meckes, Vu, PeSo} and the references therein. 

We also did not even attempt to discuss sharp 
small {\em deviation inequalities} (of Tracy-Widom type) for the extreme eigenvalues. 
Both these topics and much more are discussed in the surveys \cite{DS, Ledoux extremal, RV ICM},
which serve as bridges between asymptotic and non-asymptotic problems in random matrix theory.

Because of the absolute constant $C$ in \eqref{heuristic}, our analysis of
the smallest singular value (the {\em ``hard edge''}) will only be useful for sufficiently tall matrices, where $N \ge C^2 n$. 
For square and almost square matrices, the hard edge problem will be only briefly mentioned in Section~\ref{s: entries}.
The surveys \cite{Tao-Vu survey, RV ICM} discuss this problem at length, and they offer a glimpse
of connections to other problems of random matrix theory 
and additive combinatorics.

Many of the results and methods presented in these notes are known in one form or another.
Some of them are published while some others belong to the folklore of probability in Banach spaces, 
geometric functional analysis, and related areas. 
When available, historic references are given in Section~\ref{s: notes}.

\paragraph{Acknowledgements}
The author is grateful to the colleagues who made a number of improving suggestions for the 
earlier versions of the manuscript, in particular to Richard Chen, Subhroshekhar Ghosh, Alexander Litvak, Deanna Needell, Holger Rauhut, 
S V N Vishwanathan and the anonymous referees. Special thanks are due to Ulas Ayaz and Felix Krahmer
who thoroughly read the entire text, and whose numerous comments led to significant improvements of 
this tutorial.

\section{Preliminaries}			\label{s: preliminaries}

\subsection{Matrices and their singular values}

The main object of our study will be an $N \times n$ matrix $A$ with real or complex entries. 
We shall state all results in the real case; the reader will be able 
to adjust them to the complex case as well. Usually but not always one should think of tall matrices $A$, 
those for which $N \ge n > 1$. By passing to the adjoint matrix $A^*$, many results can be carried over 
to ``flat'' matrices, those for which $N \le n$. 

It is often convenient to study $A$ through the $n \times n$ symmetric positive-semidefinite matrix the matrix $A^*A$. 
The eigenvalues of $|A| := \sqrt{A^*A}$ are therefore non-negative real numbers. Arranged in a non-decreasing 
order, they are called the {\em singular values}\footnote{In the literature, singular values are also called {\em s-numbers}.}
\index{Singular values}
of $A$ and denoted $s_1(A) \ge \cdots \ge s_n(A) \ge 0$.
Many applications require estimates on the extreme singular values
$$
\smax(A) := s_1(A), \quad \smin(A) := s_n(A).
$$
The smallest singular value is only of interest for tall matrices, since 
for $N < n$ one automatically has $\smin(A) = 0$. 

Equivalently, $\smax(A)$ and $\smin(A)$ 
are respectively the smallest number $M$ and the largest number $m$ such that 
\begin{equation}										\label{mM}
m \|x\|_2 \le \|Ax\|_2 \le M \|x\|_2
\quad \text{for all } x \in \R^n.
\end{equation}
In order to interpret this definition geometrically, we look at $A$ as a linear operator
from $\R^n$ into $\R^N$.
The Euclidean distance between any two points in $\R^n$
can increase by at most the factor $\smax(A)$ and decrease
by at most the factor $\smax(A)$ under the action of $A$.
Therefore, the extreme singular values control the distortion of the Euclidean geometry 
under the action of $A$. If $\smax(A) \approx \smin(A) \approx 1$ then $A$ 
acts as an {\em approximate isometry},\index{Approximate isometries} or more accurately an approximate
isometric embedding of $\ell_2^n$ into $\ell_2^N$. 

The extreme singular values can also be described 
in terms of the {\em spectral norm of $A$},\index{Spectral norm} which is by definition
\begin{equation}							\label{spectral norm}
\|A\| = \|A\|_{\ell_2^n \to \ell_2^N} = \sup_{x \in \R^n \setminus \{0\}} \frac{\|Ax\|_2}{\|x\|_2}
= \sup_{x \in S^{n-1}} \|Ax\|_2.
\end{equation}
\eqref{mM} gives a link between the extreme singular values and 
the spectral norm:
$$
\smax(A) = \|A\|, \quad \smin(A) = 1/\|A^\dagger\|
$$
where $A^\dagger$ denotes the pseudoinverse of $A$; if $A$ is invertible then $A^\dagger = A^{-1}$.

\subsection{Nets}

Nets are convenient means to discretize compact sets. In our study we will mostly need 
to discretize the unit Euclidean sphere $S^{n-1}$ in the definition of the spectral norm \eqref{spectral norm}. 
Let us first recall a general definition of an $\e$-net. 

\begin{definition}[Nets, covering numbers]	\index{Net} \index{Covering numbers}
  Let $(X,d)$ be a metric space and let $\e>0$. 
  A subset $\NN_\e$ of $X$ is called an {\em $\e$-net} of $X$
  if every point $x \in X$ can be approximated to within $\e$ 
  by some point $y \in \NN_\e$, i.e. so that 
  $d(x,y) \le \e$.
  The minimal cardinality of an $\e$-net of $X$, if finite, is denoted $\NN(X,\e)$
  and is called the {\em covering number}\footnote{Equivalently, $\NN(X,\e)$ is the minimal 
    number of balls with radii $\e$ and with centers in $X$ needed to cover $X$.} 
  of $X$ (at scale $\e$).
\end{definition}

From a characterization of compactness we remember that $X$ is compact 
if and only if $\NN(X,\e) < \infty$ for each $\e > 0$. 
A quantitative estimate on $\NN(X, \e)$
would give us a {\em quantitative version of compactness} of $X$.\footnote{In 
  statistical learning theory and geometric functional analysis, 
  $\log \NN(X,\e)$ is called {\em the metric entropy of $X$}. 
  In some sense it measures the ``complexity'' of metric space $X$.}
Let us therefore take a simple example of a metric space, the unit Euclidean sphere $S^{n-1}$ 
equipped with the Euclidean metric\footnote{A similar result holds for the 
geodesic metric on the sphere, since for small $\e$ these two distances are equivalent.}
$d(x,y) = \|x-y\|_2$, and estimate its covering numbers. 

\begin{lemma}[Covering numbers of the sphere]                     				\label{net cardinality}
  The unit Euclidean sphere $S^{n-1}$ equipped with the Euclidean metric satisfies
  for every $\e>0$ that 
  $$
  \NN(S^{n-1},\e) \le \Big( 1 + \frac{2}{\e} \Big)^n.
  $$
\end{lemma}

\begin{proof}
This is a simple {\em volume argument}.
Let us fix $\e>0$ and choose $\NN_\e$ to be a maximal $\e$-separated subset of $S^{n-1}$.
In other words, $\NN_\e$ is such that 
$d(x,y) \ge \e$ for all $x, y \in \NN_\e$, $x \ne y$,
and no subset of $S^{n-1}$ containing $\NN_\e$ has this property.\footnote{One 
can in fact construct $\NN_\e$ inductively by first selecting an arbitrary 
point on the sphere, and at each next step selecting a point that is at distance at least $\e$
from those already selected. By compactness, this algorithm will terminate after finitely 
many steps and it will yield a set $\NN_\e$ as we required.}

The maximality property implies that $\NN_\e$ is an $\e$-net of $S^{n-1}$. Indeed, 
otherwise there would exist  $x \in S^{n-1}$ that is at least $\e$-far from all points in $\NN_\e$.
So $\NN_\e \cup \{x\}$ would still be an $\e$-separated set, contradicting the minimality property.

Moreover, the separation property implies via the triangle inequality that the balls of radii $\e/2$ 
centered at the points in $\NN_\e$ are disjoint. On the other hand, all such balls lie in
$(1+\e/2) B_2^n$ where $B_2^n$ denotes the unit Euclidean ball centered at the origin. 
Comparing the volume gives
$\vol \big( \frac{\e}{2} B_2^n \big) \cdot |\NN_\e|
\le \vol \big( (1 + \frac{\e}{2}) B_2^n \big)$.
Since $\vol \big( r B_2^n \big) = r^n \vol(B_2^n)$ for all $r \ge 0$, we conclude that 
$|\NN_\e| \le (1+\frac{\e}{2})^n / (\frac{\e}{2})^n = (1+\frac{2}{\e})^n$
as required.
\end{proof}

Nets allow us to reduce the complexity of computations with linear operators. 
One such example is the computation of the spectral norm. To evaluate the 
spectral norm by definition \eqref{spectral norm} one needs to take the supremum 
over the whole sphere $S^{n-1}$. However, one can essentially replace
the sphere by its $\e$-net:

\begin{lemma}[Computing the spectral norm on a net] \index{Spectral norm!computing on a net} \label{norm on net general}
  Let $A$ be an $N \times n$ matrix,
  and let $\NN_\e$ be an $\e$-net of  $S^{n-1}$
  for some $\e \in [0,1)$. Then 
  $$
  \max_{x \in \NN_\e} \|Ax\|_2 \le \|A\| \le (1 - \e)^{-1} \max_{x \in \NN_\e} \|Ax\|_2
  $$  
\end{lemma}

\begin{proof}
The lower bound in the conclusion follows from the definition. To prove the upper bound
let us fix $x \in S^{n-1}$ for which $\|A\| = \|Ax\|_2$, and choose $y \in \NN_\e$ 
which approximates $x$ as $\|x-y\|_2 \le \e$. By the triangle inequality we have
$\|Ax-Ay\|_2 \le \|A\| \|x-y\|_2 \le \e \|A\|$.
It follows that 
$$
\|Ay\|_2 \ge \|Ax\|_2 - \|Ax-Ay\|_2
\ge \|A\| - \e\|A\| = (1-\e)\|A\|.
$$
Taking maximum over all $y \in \NN_\e$ in this inequality, we complete the proof.
\end{proof}

A similar result holds for symmetric $n \times n$ matrices $A$,
whose spectral norm can be computed via the associated quadratic form: 
$\|A\| = \sup_{x \in S^{n-1}} |\< Ax,x\> |$.
Again, one can essentially replace the sphere by its $\e$-net:

\begin{lemma}[Computing the spectral norm on a net]  \index{Spectral norm!computing on a net}	\label{norm on net}
  Let $A$ be a symmetric $n \times n$ matrix,
  and let $\NN_\e$ be an $\e$-net of $S^{n-1}$
  for some $\e \in [0,1)$. Then 
  $$
  \|A\| = \sup_{x \in S^{n-1}} |\< Ax, x\> | 
  \le (1 - 2\e)^{-1} \sup_{x \in \NN_\e} |\< Ax, x\> |.
  $$  
\end{lemma}

\begin{proof}
Let us choose $x \in S^{n-1}$ for which $\|A\| = |\< Ax, x\> |$,
and choose $y \in \NN_\e$ which approximates $x$ as $\|x - y\|_2 \le \e$.
By the triangle inequality we have
\begin{align*}							
|\< Ax, x\> - \< Ay, y\> |
&= |\< Ax, x-y\> + \< A(x-y), y\> |\\
&\le \|A\| \|x\|_2 \|x-y\|_2 + \|A\| \|x-y\|_2 \|y\|_2
\le 2 \e \|A\|.
\end{align*}
It follows that 
$|\< Ay, y \> | \ge |\< Ax, x\> | - 2 \e \|A\| = (1-2\e) \|A\|$.
Taking the maximum over all $y \in \NN_\e$ in this inequality completes the proof.
\end{proof}

\subsection{Sub-gaussian random variables}	\index{Sub-gaussian!random variables}			\label{s: sub-gaussian}

In this section we introduce the class of sub-gaussian random variables,\footnote{It would be more
rigorous to say that we study {\em sub-gaussian probability distributions}. 
The same concerns some other properties of random variables and random vectors we study later in this text.
However, it is convenient for us to focus on random variables and vectors because we will form random matrices out of them.}
those whose distributions are dominated by the distribution of a centered gaussian
random variable. This is a convenient and quite wide class, which contains in particular
the standard normal and all bounded random variables. 

Let us briefly recall some of the well known properties of the {\em standard normal random variable} $X$.
The distribution of $X$ has density $\frac{1}{\sqrt{2 \pi}} e^{-x^2/2}$ and is denoted $N(0,1)$. 
Estimating the integral of this density between $t$ and $\infty$ one checks that 
the tail of a standard normal random variable $X$ decays super-exponentially:
\begin{equation}							\label{normal tail}
\P \{ |X| > t \} = \frac{2}{\sqrt{2 \pi}} \int_t^\infty e^{-x^2/2} \, dx
\le 2 e^{-t^2/2}, \quad t \ge 1,
\end{equation}
see e.g. \cite[Theorem 1.4]{Durrett} for a more precise two-sided inequality. 
The absolute moments of $X$ can be computed as
\begin{equation}							\label{normal moments}
(\E |X|^p)^{1/p} = \sqrt{2} \Big[ \frac{\Gamma((1+p)/2)}{\Gamma(1/2)} \Big]^{1/p}
= O(\sqrt{p}), \quad p \ge 1.
\end{equation}
The moment generating function of $X$ equals 
\begin{equation}							\label{normal mgf}
\E \exp(tX) = e^{t^2/2}, \quad t \in \R.
\end{equation}

Now let $X$ be a general random variable. We observe that these three properties are equivalent --
a super-exponential tail decay like in \eqref{normal tail}, the moment growth \eqref{normal moments},
and the growth of the moment generating function like in \eqref{normal mgf}. 
We will then focus on the class of random variables that satisfy these properties, which we shall call
sub-gaussian random variables. 

\begin{lemma}[Equivalence of sub-gaussian properties]			\label{sub-gaussian properties}
  Let $X$ be a random variable. Then the following properties are equivalent with parameters 
  $K_i > 0$ differing from each other by at most an absolute constant factor.\footnote{The precise meaning 
  of this equivalence is the following. There exists an absolute constant $C$ such that property $i$
  implies property $j$ with parameter $K_j \le C K_i$ for any two properties $i,j=1,2,3$.}
  
  \begin{enumerate}
    \item Tails:  
      $\P \{ |X| > t \} \le \exp(1-t^2/K_1^2)$ for all $t \ge 0$;
    \item Moments:
      $(\E |X|^p)^{1/p} \le K_2 \sqrt{p}$ for all $p \ge 1$;
    \item Super-exponential moment:
      $\E \exp(X^2/K_3^2) \le e$.
  \end{enumerate}
  Moreover, if $\E X = 0$ then properties 1--3 are also equivalent 
  to the following one: 
  \begin{enumerate} \setcounter{enumi}{3}
    \item Moment generating function:
      $\E \exp(tX) \le \exp(t^2 K_4^2)$ for all $t \in \R$.
  \end{enumerate}
  
\end{lemma}

\begin{proof}
{\bf 1. $\Rightarrow$ 2.}
Assume property 1 holds. 
By homogeneity, rescaling $X$ to $X/K_1$ we can assume that $K_1=1$. 
Recall that for every non-negative random variable $Z$, integration by parts 
yields the identity 
$\E Z = \int_0^\infty \P \{ Z \ge u \} \, du$.
We apply this for $Z = |X|^p$. 
After change of variables $u = t^p$, we obtain using property 1 that
$$
\E |X|^p 
= \int_0^\infty \P \{ |X| \ge t \} \, p t^{p-1} \, dt
\le \int_0^\infty e^{1-t^2} p t^{p-1} \, dt
= \big( \frac{ep}{2} \big) \Gamma \big( \frac{p}{2} \big) 
\le \big( \frac{ep}{2} \big) \big( \frac{p}{2} \big)^{p/2}.
$$
Taking the $p$-th root yields property 2 with a suitable absolute constant $K_2$.

{\bf 2. $\Rightarrow$ 3.}
Assume property 2 holds. 
As before, by homogeneity we may assume that $K_2 = 1$. Let $c>0$ be a sufficiently small absolute constant. 
Writing the Taylor series of the exponential function, we obtain
$$
\E \exp(c X^2) 
= 1 + \sum_{p=1}^\infty \frac{c^p \E(X^{2p})}{p!}
\le 1 + \sum_{p=1}^\infty \frac{c^p (2p)^p}{p!}	
\le 1 + \sum_{p=1}^\infty (2c/e)^p.
$$
The first inequality follows from property 2; in the second one we use $p! \ge (p/e)^p$.
For small $c$ this gives $\E \exp(c X^2) \le e$, which 
is property 3 with $K_3 = c^{-1/2}$.

{\bf 3. $\Rightarrow$ 1.}
Assume property 3 holds. As before we may assume that $K_3=1$. 
Exponentiating and using Markov's inequality\footnote{This simple argument is sometimes called 
exponential Markov's inequality.}
and then property 3, we have
$$
\P \{ |X| > t \} 
= \P \{ e^{X^2} \ge e^{t^2} \} 
\le e^{-t^2} \E e^{X^2}
\le e^{1-t^2}.
$$
This proves property 1 with $K_1=1$.

{\bf 2. $\Rightarrow$ 4.}  Let us now assume that $\E X = 0$ and property 2 holds; 
as usual we can assume that $K_2 = 1$. 
We will prove that property 4 holds with an appropriately large absolute constant $C = K_4$.
This will follow by estimating Taylor series for the exponential function
\begin{equation}										\label{mgf above}
\E \exp(tX) 
= 1 + t \E X + \sum_{p=2}^\infty \frac{t^p \E X^p}{p!} 
\le 1 +  \sum_{p=2}^\infty \frac{t^p p^{p/2}}{p!}	
\le 1 +  \sum_{p=2}^\infty \Big( \frac{e|t|}{\sqrt{p}} \Big)^p.
\end{equation}
The first inequality here follows from $\E X = 0$ and property 2; the second one holds since $p! \ge (p/e)^p$.
We compare this with Taylor's series for 
\begin{equation}										\label{exp t squared}
\exp(C^2 t^2) 
= 1 + \sum_{k=1}^\infty \frac{(C|t|)^{2k}}{k!}
\ge 1 + \sum_{k=1}^\infty \Big( \frac{C|t|}{\sqrt{k}} \Big)^{2k}
= 1 + \sum_{p \in 2\N} \Big( \frac{C|t|}{\sqrt{p/2}} \Big)^p.
\end{equation}
The first inequality here holds because $p! \le p^p$; the second one is obtained by substitution $p=2k$.
One can show that the series in \eqref{mgf above} is bounded by the series in \eqref{exp t squared} with large absolute constant $C$. 
We conclude that $\E \exp(tX) \le \exp(C^2 t^2)$, which proves property~4.

{\bf 4. $\Rightarrow$ 1.}
Assume property 4 holds; we can also assume that $K_4=1$. 
Let $\l > 0$ be a parameter to be chosen later. By exponential Markov inequality,
and using the bound on the moment generating function given in property 4, we obtain
$$
\P \{ X \ge t \} = \P \{ e^{\l X} \ge e^{\l t} \}
\le e^{-\l t} \E e^{\l X}
\le e^{-\l t + \l^2}.
$$
Optimizing in $\l$ and thus choosing $\l = t/2$ we conclude that 
$\P \{ X \ge t \} \le e^{-t^2/4}$.
Repeating this argument for $-X$, we also obtain $\P \{ X \le -t \} \le e^{-t^2/4}$.
Combining these two bounds we conclude that 
$\P \{ |X| \ge t \} \le 2 e^{-t^2/4} \le e^{1-t^2/4}$.
Thus property 1 holds with $K_1=2$.
The lemma is proved. 
\end{proof}

\begin{remark}
\begin{enumerate}
  \item The constants $1$ and $e$ in properties 1 and 3 respectively are chosen for convenience.
  Thus the value $1$ can be replaced by any positive number and the value $e$ can be replaced
  by any number greater than $1$.   
    
  \item The assumption $\E X = 0$ is only needed to prove the necessity of property 4;
  the sufficiency holds without this assumption. 
\end{enumerate}
\end{remark}

\begin{definition}[Sub-gaussian random variables]  \index{Sub-gaussian!random variables}
  A random variable $X$ that satisfies one of the equivalent properties 1 -- 3 in Lemma~\ref{sub-gaussian properties}
  is called a {\em sub-gaussian random variable}.
  The {\em sub-gaussian norm} \index{Sub-gaussian!norm} of $X$, denoted $\|X\|_\psitwo$, 
  is defined to be the smallest $K_2$ in property 2.
  In other words,\footnote{The sub-gaussian norm is also called $\psitwo$ norm in the literature.}
  $$
  \|X\|_\psitwo = \sup_{p \ge 1} p^{-1/2} (\E |X|^p)^{1/p}.
  $$  
\end{definition}

The class of sub-gaussian random variables on a given probability space 
is thus a normed space. By Lemma~\ref{sub-gaussian properties},
every sub-gaussian random variable $X$ satisfies:
\begin{gather}
\P \{ |X| > t \} \le \exp(1-ct^2/\|X\|_\psitwo^2) \quad \text{for all } t \ge 0; 	\label{sub-gaussian tail}\\
(\E |X|^p)^{1/p} \le \|X\|_\psitwo \sqrt{p} \quad \text{for all } p \ge 1; 		\label{sub-gaussian moments}\\
\E \exp(cX^2/\|X\|_\psitwo^2) \le e; \nonumber\\
\text{if $\E X =0$ then } \E \exp(tX) \le \exp(C t^2 \|X\|_\psitwo^2) \quad \text{for all } t \in \R, \label{sub-gaussian mgf}
\end{gather}
where $C, c > 0$ are absolute constants. Moreover, up to absolute constant factors, 
$\|X\|_\psitwo$ is the smallest possible number in each of these inequalities.

\begin{example}
Classical examples of sub-gaussian random variables are Gaussian, Bernoulli
and all bounded random variables.  
\begin{enumerate}
  \item {\bf (Gaussian):} A standard normal random variable $X$ is sub-gaussian with 
  $\|X\|_\psitwo \le C$ where $C$ is an absolute constant. This follows from \eqref{normal moments}. 
  More generally, if $X$ is 
  a centered normal random variable with variance $\s^2$, then $X$ is sub-gaussian 
  with $\|X\|_\psitwo \le C \s$.
  
  \item {\bf (Bernoulli):} \index{Bernoulli!random variables} Consider a random variable $X$ with distribution 
  $\P\{X=-1\} = \P\{X=1\} = 1/2$. We call $X$ a {\em symmetric Bernoulli random variable}.
  Since $|X|=1$, it follows that $X$ is a sub-gaussian random variable with $\|X\|_\psitwo = 1$. 
      
  \item {\bf (Bounded):} More generally, consider any bounded random variable $X$, 
  thus $|X| \le M$ almost surely for some $M$. Then $X$ is a sub-gaussian random variable with $\|X\|_\psitwo \le M$.
  We can write this more compactly as
  $\|X\|_\psitwo \le \|X\|_\infty$.
\end{enumerate}
\end{example}

A remarkable property of the normal distribution is {\em rotation invariance}. 
Given a finite number of independent centered normal random variables $X_i$,
their sum $\sum_i X_i$ is also a centered normal random variable, 
obviously with $\Var(\sum_i X_i) = \sum_i \Var(X_i)$.
Rotation invariance passes onto sub-gaussian random variables, although approximately:

\begin{lemma}[Rotation invariance]  \index{Rotation invariance}			\label{rotation invariance}
  Consider a finite number of independent centered sub-gaussian random variables $X_i$.
  Then $\sum_i X_i$ is also a centered sub-gaussian random variable. Moreover, 
  $$
  \big\| \sum_i X_i \big\|_\psitwo^2 
  \le C \sum_i \|X_i\|_\psitwo^2
  $$
  where $C$ is an absolute constant. 
\end{lemma} 

\begin{proof}
The argument is based on estimating the moment generating function.
Using independence and \eqref{sub-gaussian mgf} we have for every $t \in \R$: 
\begin{align*}
\E \exp \big( t \sum_i X_i \big)
&= \E \prod_i \exp(t X_i) 
= \prod_i \E \exp(t X_i)
\le \prod_i \exp(C t^2 \|X_i\|_\psitwo^2) \\
&= \exp(t^2 K^2) \quad \text{where } K^2 = C \sum_i \|X_i\|_\psitwo^2.
\end{align*}
Using the equivalence of properties 2 and 4 in Lemma~\ref{sub-gaussian properties}
we conclude that $\|\sum_i X_i\|_\psitwo \le C_1 K$ where $C_1$ is an absolute constant. 
The proof is complete.
\end{proof}

The rotation invariance immediately yields a {\em large deviation inequality} 
for sums of independent sub-gaussian random variables: 

\begin{proposition}[Hoeffding-type inequality]	\index{Hoeffding-type inequality}	\label{sub-gaussian large deviations}
  Let $X_1,\ldots,X_N$ be independent centered sub-gaussian random variables,
  and let $K = \max_i \|X_i\|_\psitwo$. 
  Then for every  $a = (a_1,\ldots,a_N) \in \R^N$ and every $t \ge 0$, we have
  $$
  \P \Big\{ \Big| \sum_{i=1}^N a_i X_i \Big| \ge t \Big\} 
  \le e \cdot \exp \Big( -\frac{ct^2}{K^2\|a\|_2^2} \Big)
  $$   
  where $c>0$ is an absolute constant. 
\end{proposition}

\begin{proof}
The rotation invariance (Lemma~\ref{rotation invariance}) implies the bound
$\|\sum_i a_i X_i\|_\psitwo^2 \le C \sum_i a_i^2 \|X_i\|_\psitwo^2 \le C K^2 \|a\|_2^2$.
Property \eqref{sub-gaussian tail} yields the required tail decay.
\end{proof}

\begin{remark}
  One can interpret these results (Lemma~\ref{rotation invariance} and Proposition~\ref{sub-gaussian large deviations})
  as one-sided {\em non-asymptotic manifestations of the central limit theorem}. For example, 
  consider the normalized sum of independent symmetric Bernoulli random variables
  $S_N = \frac{1}{\sqrt{N}} \sum_{i=1}^N \e_i$. Proposition~\ref{sub-gaussian large deviations}
  yields the tail bounds $\P \{ |S_N| > t \}  \le e \cdot e^{-ct^2}$ for any number of terms $N$. 
  Up to the absolute constants $e$ and $c$, these tails coincide with those of the standard normal random variable
  \eqref{normal tail}.
  
\end{remark}

Using moment growth \eqref{sub-gaussian moments}
instead of the tail decay \eqref{sub-gaussian tail}, we immediately 
obtain from Lemma~\ref{rotation invariance} a general form of the well known Khintchine inequality:

\begin{corollary}[Khintchine inequality]	\index{Khinchine inequality}				\label{Khintchine}		
  Let $X_i$ be a finite number of independent sub-gaussian random variables
  with zero mean, unit variance, and $\|X_i\|_{\psitwo} \le K$.
  Then, for every sequence of coefficients $a_i$ and every exponent $p \ge 2$ we have
  $$
  \big( \sum_i a_i^2 \big)^{1/2}
  \le \big( \E \big| \sum_i a_i X_i \big|^p \big)^{1/p}
  \le C K \sqrt{p} \, \big( \sum_i a_i^2 \big)^{1/2}
  $$   
  where $C$ is an absolute constant.
\end{corollary}

\begin{proof}
The lower bound follows by independence and H\"older's inequality: indeed, 
$\big( \E \big| \sum_i a_i X_i \big|^p \big)^{1/p} \ge \big( \E \big| \sum_i a_i X_i \big|^2 \big)^{1/2}
= \big( \sum_i a_i^2 \big)^{1/2}$. For the upper bound, we argue as in Proposition~\ref{sub-gaussian large deviations},
but use property \eqref{sub-gaussian moments}.
\end{proof}

\subsection{Sub-exponential random variables}	\index{Sub-exponential!random variables} \label{s: sub-exponential}

Although the class of sub-gaussian random variables is natural and quite wide, 
it leaves out some useful random variables which have tails heavier than gaussian.
One such example is a standard exponential random variable
-- a non-negative random variable with exponential tail decay 
\begin{equation}							\label{exponential}
\P\{X \ge t\} = e^{-t}, \quad t \ge 0.
\end{equation}
To cover such examples, we consider a class of {\em sub-exponential random variables},
those with at least an exponential tail decay. With appropriate 
modifications, the basic properties of sub-gaussian random variables hold
for sub-exponentials. 
In particular, a version of Lemma~\ref{sub-gaussian properties} holds with a similar proof 
for sub-exponential properties, except for property 4 of the moment generating function. 
Thus for a random variable $X$ the following properties are equivalent with parameters 
$K_i > 0$ differing from each other by at most an absolute constant factor:
\begin{gather}
\P \{ |X| > t \} \le \exp(1-t/K_1) \quad \text{for all } t \ge 0;  	\label{sub-exponential tail} \\
(\E |X|^p)^{1/p} \le K_2 p \quad \text{for all } p \ge 1; 		\label{sub-exponential moments} \\
\E \exp(X/K_3) \le e.									\label{sub-exponential integrability}
\end{gather} 

\begin{definition}[Sub-exponential random variables]  \index{Sub-exponential!random variables}
  A random variable $X$ that satisfies one of the equivalent properties 
  \eqref{sub-exponential tail} -- \eqref{sub-exponential integrability}
  is called a {\em sub-exponential random variable}.
  The {\em sub-exponential norm} \index{Sub-exponential!norm} of $X$, denoted $\|X\|_\psione$, 
  is defined to be the smallest parameter $K_2$.
  In other words,
  $$
  \|X\|_\psione = \sup_{p \ge 1} p^{-1} (\E |X|^p)^{1/p}.
  $$ 
\end{definition}

\begin{lemma}[Sub-exponential is sub-gaussian squared]				\label{sub-exponential squared}
  A random variable $X$ is sub-gaussian if and only if $X^2$ is sub-exponential.
  Moreover,
  $$
  \|X\|_\psitwo^2 \le \|X^2\|_\psione \le 2 \|X\|_\psitwo^2. 
  $$
\end{lemma}

\begin{proof}
This follows easily from the definition.
\end{proof}

The moment generating function of a sub-exponential random variable has a similar 
upper bound as in the sub-gaussian case (property 4 in Lemma~\ref{sub-gaussian properties}).
The only real difference is that the bound only holds in a neighborhood of zero rather than on the 
whole real line. This is inevitable, as the moment generating function
of an exponential random variable \eqref{exponential} does not exist for $t \ge 1$. 

\begin{lemma}[Mgf of sub-exponential random variables]		\label{sub-exponential mgf}
  Let $X$ be a centered sub-exponential random variable. Then, for $t$ such that
  $|t| \le c/\|X\|_\psione$, one has
  $$
  \E \exp(t X) \le \exp(C t^2 \|X\|_\psione^2)
  $$
  where $C, c > 0$ are absolute constants. 
\end{lemma}

\begin{proof}
The argument is similar to the sub-gaussian case. 
We can assume that $\|X\|_\psione=1$ by replacing $X$ with $X/\|X\|_\psione$
and $t$ with $t \|X\|_\psione$. Repeating the proof of the 
implication 2 $\Rightarrow$ 4 of Lemma~\ref{sub-gaussian properties}
and using $\E|X|^p \le p^p$ this time, we obtain that 
$\E \exp(tX) \le 1 + \sum_{p=2}^\infty (e|t|)^p$.
If $|t| \le 1/2e$ then the right hand side is bounded by $1 + 2e^2 t^2 \le \exp(2e^2 t^2)$.
This completes the proof. 
\end{proof}

Sub-exponential random variables satisfy a  {\em large deviation inequality} 
similar to the one for sub-gaussians (Proposition~\ref{sub-gaussian large deviations}).
The only significant difference is that {\em two tails} have to appear here -- 
a gaussian tail responsible for the central limit theorem, 
and an exponential tail coming from the tails of each term.

\begin{proposition}[Bernstein-type inequality]	\index{Bernstein-type inequality} \label{sub-exponential large deviations}
  Let~$X_1,\ldots,X_N$ be independent centered sub-exponential random variables,
  and $K = \max_i \|X_i\|_\psione$. 
  Then for every $a = (a_1,\ldots,a_N) \in \R^N$
  and every $t \ge 0$, we have
  $$
  \P \Big\{ \Big| \sum_{i=1}^N a_i X_i \Big| \ge t \Big\} 
  \le 2 \exp \Big[ -c \min \Big( \frac{t^2}{K^2\|a\|_2^2}, \; \frac{t}{K\|a\|_\infty} \Big) \Big]
  $$  
  where $c>0$ is an absolute constant. 
\end{proposition}

\begin{proof}
Without loss of generality, we assume that $K=1$ by replacing $X_i$ with $X_i/K$ 
and $t$ with $t/K$. We use the exponential Markov inequality for
the sum $S = \sum_i a_i X_i$ and with a parameter $\l>0$:
$$
\P \{ S \ge t \} 
= \P \{ e^{\l S} \ge e^{\l t} \}  
\le e^{-\l t} \E e^{\l S}
= e^{-\l t} \prod_i \E \exp (\l a_i X_i).
$$
If $|\l| \le c/\|a\|_\infty$ then $|\l a_i| \le c$ for all $i$, so Lemma~\ref{sub-exponential mgf} yields
$$
\P \{ S \ge t \}  
\le e^{-\l t} \; \prod_i \exp (C \l^2 a_i^2) 
=  \exp(-\l t + C \l^2 \|a\|_2^2).
$$
Choosing $\l = \min( t/2C\|a\|_2^2, \, c/\|a\|_\infty )$, we obtain that 
$$
\P \{ S \ge t \} \le \exp \Big[ - \min \Big( \frac{t^2}{4C\|a\|_2^2}, \; \frac{ct}{2\|a\|_\infty} \Big) \Big].
$$
Repeating this argument for $-X_i$ instead of $X_i$, we obtain the same bound for
$\P \{ -S \ge t \}$. A combination of these two bounds completes the proof.
\end{proof}

\begin{corollary}						\label{average sub-exponentials}
  Let $X_1,\ldots,X_N$ be independent centered sub-exponential random variables,
  and let $K = \max_i \|X_i\|_\psione$. 
  Then, for every $\e \ge 0$, we have
  $$
  \P \Big\{ \Big| \sum_{i=1}^N X_i \Big| \ge \e N \Big\} 
  \le 2 \exp \Big[ -c \min \Big( \frac{\e^2}{K^2}, \; \frac{\e}{K} \Big) N \Big]
  $$  
  where $c>0$ is an absolute constant.  
\end{corollary}

\begin{proof}
This follows from Proposition~\ref{sub-exponential large deviations} for $a_i=1$ and $t=\e N$.
\end{proof}

\begin{remark}[Centering]				\label{centering}
The definitions of sub-gaussian and sub-exponential random variables $X$ do not require
them to be centered. In any case, one can always center $X$ using the simple fact that
if $X$ is sub-gaussian (or sub-exponential), then so is $X - \E X$. Moreover,  
$$
\|X - \E X\|_\psitwo \le 2 \|X\|_\psitwo, \quad
\|X - \E X\|_\psione \le 2 \|X\|_\psione.
$$
This follows by triangle inequality
$\|X - \E X\|_\psitwo \le \|X\|_\psitwo + \|\E X\|_\psitwo$ along with 
$\|\E X\|_\psitwo = |\E X| \le \E|X| \le \|X\|_\psitwo$,
and similarly for the sub-exponential norm.
\end{remark}

\subsection{Isotropic random vectors} \index{Isotropic random vectors}			\label{s: isotropic}

Now we carry our work over to higher dimensions. We will thus 
be working with random vectors $X$ in $\R^n$, 
or equivalently probability distributions in $\R^n$.

While the concept of the mean $\mu = \E Z$ of a random variable $Z$ remains the same
in higher dimensions, the second moment $\E Z^2$ 
is replaced by the $n \times n$ {\em second moment matrix} \index{Second moment matrix}
of a random vector $X$, defined as 
$$
\Sigma = \Sigma(X) = \E X \otimes X = \E X X^T
$$
where $\otimes$ denotes the outer product of vectors in $\R^n$. 
Similarly, the concept of variance $\Var(Z) = \E(Z - \mu)^2 = \E Z^2 - \mu^2$ of a random variable
is replaced in higher dimensions with the {\em covariance matrix} \index{Covariance matrix} of a random vector 
$X$, defined as 
$$
\Cov(X) = \E (X-\mu) \otimes (X-\mu) = \E X \otimes X - \mu \otimes \mu
$$
where $\mu = \E X$. 
By translation, many questions can be reduced to the case of centered random vectors, 
for which $\mu = 0$ and $\Cov(X) = \Sigma(X)$. We will also need a higher-dimensional 
version of unit variance: 

\begin{definition}[Isotropic random vectors]  \index{Isotropic random vectors}
  A random vector $X$ in $\R^n$ is called {\em isotropic} if $\Sigma(X) = I$.
  Equivalently, $X$ is isotropic if
  \begin{equation}							\label{isotropy}
  \E \< X, x\> ^2 = \|x\|_2^2 	\quad \text{for all } x \in \R^n.
  \end{equation}
\end{definition}

Suppose $\Sigma(X)$ is an invertible matrix, which means that the distribution of $X$ is 
not essentially supported on any proper subspace of $\R^n$. 
Then $\Sigma(X)^{-1/2} X$ is an isotropic random vector in $\R^n$. 
Thus every non-degenerate random vector can be made isotropic by an appropriate
linear transformation.\footnote{This transformation (usually preceded by centering) 
is a higher-dimensional version of {\em standardizing} of random variables, which enforces zero mean and unit variance.}
This allows us to mostly focus on studying isotropic random vectors in the future.

\begin{lemma}					\label{norm isotropic}
  Let $X,Y$ be independent isotropic random vectors in $\R^n$. Then
  $\E \|X\|_2^2 = n$ and $\E \< X,Y\> ^2 = n$.
\end{lemma}

\begin{proof}
The first part follows from 
$\E \|X\|_2^2 = \E \tr(X \otimes X) = \tr(\E X \otimes X) = \tr(I) = n$. 
The second part follows by conditioning on $Y$, using isotropy of $X$ and using the first part for $Y$:
this way we obtain $\E \< X,Y\> ^2 = \E \|Y\|_2^2 = n$.
\end{proof}

\begin{example}								\label{random vectors}
\begin{enumerate}
  \item {\bf (Gaussian):} 
  The (standard) {\em Gaussian random vector} \index{Gaussian!random vectors} $X$ in $\R^n$ chosen according to the 
  standard normal distribution $N(0, I)$ is isotropic. The coordinates of $X$ 
  are independent standard normal random variables.
  
  \item {\bf (Bernoulli):} \index{Bernoulli!random vectors} A similar example of a discrete isotropic distribution is 
  given by a {\em Bernoulli random vector} $X$ in $\R^n$ whose
  coordinates are independent symmetric Bernoulli random variables.
  
  \item {\bf (Product distributions):} More generally, consider a random vector $X$
  in $\R^n$ whose coordinates are independent random variables with zero mean and unit variance. 
  Then clearly $X$ is an isotropic vector in $\R^n$. 
  
  \item {\bf (Coordinate):} \index{Coordinate random vectors}
  Consider a {\em coordinate random vector} $X$, which is  
  uniformly distributed in the set $\{ \sqrt{n} \, e_i \}_{i=1}^n$ 
  where $\{e_i \}_{i=1}^n$ is the canonical basis of $\R^n$. 
  Clearly $X$ is an isotropic random vector in $\R^n$.\footnote{The examples of Gaussian 
    and coordinate random vectors are somewhat opposite --
    one is very continuous and the other is very discrete. They may be used as test 
    cases in our study of random matrices.} 
  
  \item {\bf (Frame):} \index{Frames} This is a more general version of the coordinate random vector. 
  A {\em frame} is a set of vectors $\{u_i\}_{i=1}^M$ in $\R^n$ 
  which obeys an approximate Parseval's identity, i.e. there exist numbers $A,B>0$ called {\em frame bounds}
  such that  
  $$
  A\|x\|_2^2 \le \sum_{i=1}^M \< u_i, x \> ^2 \le B\|x\|_2^2		\quad \text{for all } x \in \R^n.
  $$
  If $A=B$ the set is called a {\em tight frame}.
  Thus, tight frames are generalizations of orthogonal bases without linear independence. 
  Given a tight frame $\{u_i\}_{i=1}^M$ with bounds $A=B=M$, 
  the random vector $X$ uniformly distributed in the set 
  $\{u_i \}_{i=1}^M$ is clearly isotropic in $\R^n$.\footnote{There is clearly a reverse implication, too, 
  which shows that the class of tight frames can be identified with the class of discrete isotropic random vectors.}
  
  \item{\bf (Spherical):} \index{Spherical random vector}
  Consider a random vector $X$ uniformly distributed on the unit Euclidean
  sphere in $\R^n$ with center at the origin and radius $\sqrt{n}$. Then $X$ is isotropic. 
  Indeed, by rotation invariance $\E \< X,x\> ^2$ is proportional to $\|x\|_2^2$; the correct normalization 
  $\sqrt{n}$ is derived from Lemma~\ref{norm isotropic}.
  
  \item {\bf (Uniform on a convex set):} In convex geometry, a convex set $K$ in $\R^n$
  is called isotropic if a random vector $X$ chosen uniformly from $K$ according 
  to the volume is isotropic. As we noted, every full dimensional convex set can be made into an isotropic one by 
  an affine transformation. Isotropic convex sets look ``well conditioned'', 
  which is advantageous in geometric algorithms (e.g. volume computations). 
\end{enumerate}
\end{example}

We generalize the concepts of sub-gaussian random variables to higher dimensions 
using one-dimensional marginals. 

\begin{definition}[Sub-gaussian random vectors] \index{Sub-gaussian!random vectors}			\label{d: sub-gaussian}
  We say that a random vector $X$ in $\R^n$ is {\em sub-gaussian}
  if the one-dimensional marginals $\< X, x\> $ are sub-gaussian random variables for all $x \in \R^n$.  
  The {\em sub-gaussian norm} \index{Sub-gaussian!norm} of $X$ is defined as 
  $$
  \|X\|_\psitwo = \sup_{x \in S^{n-1}} \|\< X, x\> \|_\psitwo.
  $$
\end{definition}

\begin{remark}[Properties of high-dimensional distributions]
  The definitions of isotropic and sub-gaussian distributions suggest that more generally,
  natural properties of high-dimensional distributions may be defined via one-dimensional marginals. 
  This is a natural way to generalize properties of random variables to random vectors. 
  For example, we shall call a random vector sub-exponential if all of its one-dimensional 
  marginals are sub-exponential random variables, etc.
\end{remark}

One simple way to create sub-gaussian distributions in $\R^n$ is by taking a product
of $n$ sub-gaussian distributions on the line: 

\begin{lemma}[Product of sub-gaussian distributions]		\label{sub-gaussian products}
  Let $X_1,\ldots,X_n$ be independent centered sub-gaussian random variables. 
  Then $X = (X_1,\ldots,X_n)$ is a centered sub-gaussian random vector in $\R^n$, and
  $$
  \|X\|_\psitwo \le C \max_{i \le n} \|X_i\|_\psitwo
  $$
  where $C$ is an absolute constant.
\end{lemma}

\begin{proof}
This is a direct consequence of the rotation invariance principle, Lemma~\ref{rotation invariance}.
Indeed, for every $x = (x_1,\ldots,x_n) \in S^{n-1}$ we have
$$
\|\< X, x\> \|_\psitwo 
= \Big\|  \sum_{i=1}^n x_i X_i \Big\|_\psitwo
\le C \sum_{i=1}^n x_i^2 \|X_i\|_\psitwo^2 
\le C \max_{i \le n} \|X_i\|_\psitwo
$$
where we used that $\sum_{i=1}^n x_i^2 = 1$. 
This completes the proof.
\end{proof}

\begin{example}				\label{random vectors sub-gaussian}
Let us analyze the basic examples of random vectors introduced earlier 
in Example~\ref{random vectors}.  
\begin{enumerate}
  \item {\bf (Gaussian, Bernoulli):} \index{Gaussian!random vectors} \index{Bernoulli!random vectors}
  Gaussian and Bernoulli random vectors are sub-gaussian; 
  their sub-gaussian norms are bounded by an absolute constant. 
  These are particular cases of Lemma~\ref{sub-gaussian products}. 
  
  \item {\bf (Spherical):} \index{Spherical random vector} A spherical random vector is also sub-gaussian; 
  its sub-gaussian norm is bounded by an absolute constant.
  Unfortunately, this does not follow from Lemma~\ref{sub-gaussian products} because the coordinates
  of the spherical vector are not independent. Instead, by rotation invariance, the claim 
  clearly follows from the following geometric fact.  
  For every $\e \ge 0$, the spherical cap $\{ x \in S^{n-1}:\; x_1 > \e\}$ makes up 
  at most $\exp(-\e^2 n/2)$ proportion of the total area on the sphere.\footnote{This
  fact about spherical caps may seem counter-intuitive. For example, for $\e = 0.1$ the
  cap looks similar to a hemisphere, but the proportion of its area goes to zero 
  very fast as dimension $n$ increases. 
  This is a starting point of the study of the {\em concentration of measure phenomenon},
  see \cite{Ledoux}.}
  This can be proved directly by integration, and also by elementary geometric considerations \cite[Lemma~2.2]{Ball}.
  
  \item {\bf (Coordinate):} \index{Coordinate random vectors} 
  Although the coordinate random vector $X$ is formally sub-gaussian 
  as its support is finite, its sub-gaussian norm is too big: $\|X\|_\psitwo = \sqrt{n} \gg 1$.
  So we would not think of $X$ as a sub-gaussian random vector. 
 
  \item {\bf (Uniform on a convex set):} For many isotropic convex sets $K$ (called $\psi_2$ bodies), 
  a random vector $X$ uniformly distributed in $K$ is sub-gaussian with $\|X\|_\psitwo = O(1)$.
  For example, the cube $[-1,1]^n$ is a $\psi_2$ 
  body by Lemma~\ref{sub-gaussian products}, while the appropriately 
  normalized cross-polytope $\{ x \in \R^n:\; \|x\|_1 \le M \}$ is not. 
  Nevertheless, Borell's lemma (which is a consequence of Brunn-Minkowski inequality) 
  implies a weaker property, that $X$ is always {\em sub-exponential},
  and $\|X\|_\psione = \sup_{x \in S^{n-1}} \|\< X, x\> \|_\psione$ is bounded by absolute constant. 
  See \cite[Section~2.2.b$_3$]{GiMi} for a proof and discussion of these ideas.     
\end{enumerate}
\end{example}

\subsection{Sums of independent random matrices}				\label{s: sums matrices}

In this section, we mention without proof some results of classical probability theory 
in which scalars can be replaced by matrices. 
Such results are useful in particular for problems on random matrices,
since we can view a random matrix as a generalization of a random variable. 
One such remarkable generalization is valid for Khintchine inequality,
Corollary~\ref{Khintchine}. The scalars $a_i$ can be replaced by matrices, and the absolute value 
by the {\em Schatten norm}. \index{Schatten norm}
Recall that for $1 \le p \le \infty$, the $p$-Schatten norm of an $n \times n$ matrix $A$ is defined as
the $\ell_p$ norm of the sequence of its singular values:  
$$
\|A\|_{C_p^n} = \| (s_i(A))_{i=1}^n\|_p = \big( \sum_{i=1}^n s_i(A)^p \big)^{1/p}.
$$
For $p=\infty$, the Schatten norm equals the spectral norm $\|A\| = \max_{i \le n} s_i(A)$. 
Using this one can quickly check that already for $p = \log n$ the Schatten and spectral 
norms are equivalent: $\|A\|_{C_p^n} \le \|A\| \le e \|A\|_{C_p^n}$.

\begin{theorem}[Non-commutative Khintchine inequality, see \cite{Pisier operator} Section 9.8]		
					\index{Khinchine inequality!non-commutative}	\label{non-commutative Khintchine}
  \hfill Let $A_1, \ldots, A_N$ be self-adjoint $n \times n$ matrices and 
  $\e_1, \ldots, \e_N$ be independent symmetric Bernoulli random variables. 
  Then, for every $2 \le p < \infty$, we have
  $$
  \Big\| \Big( \sum_{i=1}^N A_i^2 \Big)^{1/2} \Big\|_{C_p^n}
  \le \Big( \E \Big\| \sum_{i=1}^N \e_i A_i \Big\|_{C_p^n}^p \Big)^{1/p}
  \le C \sqrt{p} \, \Big\| \Big( \sum_{i=1}^N A_i^2 \Big)^{1/2} \Big\|_{C_p^n}
  $$
  where $C$ is an absolute constant. 
\end{theorem}

\begin{remark}
\begin{enumerate}
  \item The scalar case of this result, for $n=1$, recovers the classical Khintchine inequality, 
  Corollary~\ref{Khintchine}, for $X_i = \e_i$.

  \item By the equivalence of Schatten and spectral norms for $p=\log n$, 
  a version of non-commutative Khintchine inequality holds for the spectral norm:
  \begin{equation}										\label{Khintchine operator norm}
  \E \Big\| \sum_{i=1}^N \e_i A_i \Big\|
  \le C_1 \sqrt{\log n} \, \Big\| \Big( \sum_{i=1}^N A_i^2 \Big)^{1/2} \Big\|
  \end{equation}
  where $C_1$ is an absolute constant. The logarithmic factor is unfortunately essential; 
  it role will be clear when we discuss applications of this result to random matrices in the next sections. 
\end{enumerate}
\end{remark}

\begin{corollary}[Rudelson's inequality \cite{Rudelson isotropic}]	\index{Rudelson's inequality}	\label{Rudelson}
  Let $x_1, \ldots, x_N$ be vectors in $\R^n$ and
  $\e_1, \ldots, \e_N$ be independent symmetric Bernoulli random variables. 
  Then 
  $$
  \E \Big\| \sum_{i=1}^N \e_i x_i \otimes x_i \Big\|
  \le C \sqrt{\log \min(N,n)} \cdot \max_{i \le N} \|x_i\|_2 \cdot \Big\| \sum_{i=1}^N x_i \otimes x_i \Big\|^{1/2}
  $$ 
  where $C$ is an absolute constant.
\end{corollary}

\begin{proof}
One can assume that $n \le N$ by replacing $\R^n$ with the linear span of $\{x_1,\ldots,x_N\}$
if necessary. The claim then follows from \eqref{Khintchine operator norm}, since 
$$
\Big\| \Big( \sum_{i=1}^N (x_i \otimes x_i)^2 \Big)^{1/2} \Big\|
= \Big\| \sum_{i=1}^N \|x_i\|_2^2 \; x_i \otimes x_i \Big\|^{1/2}
\le \max_{i \le N} \|x_i\|_2 \Big\| \sum_{i=1}^N x_i \otimes x_i \Big\|^{1/2}. \qedhere
$$
\end{proof}

Ahlswede and Winter \cite{AW} pioneered a different approach to matrix-valued 
inequalities in probability theory, which was based on trace inequalities like Golden-Thompson 
inequality. A development of this idea leads to remarkably sharp results. We quote one such inequality 
from \cite{Tropp}:

\begin{theorem}[Non-commutative Bernstein-type inequality \cite{Tropp}]	
  \index{Bernstein-type inequality!non-commutative}			\label{matrix Bernstein}
  Consider a finite sequence $X_i$ of independent centered self-adjoint random $n \times n$ matrices. 
  Assume we have for some numbers $K$ and $\s$ that  
  $$
  \|X_i\| \le K \text{ almost surely}, \quad \big\| \sum_i \E X_i^2 \big\| \le \s^2.
  $$
  Then, for every $t \ge 0$ we have
  \begin{equation}							\label{eq matrix Bernstein}
  \P \Big\{ \big\| \sum_i X_i \big\| \ge t \Big\} \le 2 n \cdot \exp \Big( \frac{-t^2/2}{\s^2 + Kt/3} \Big).
  \end{equation}
\end{theorem}

\begin{remark}				\label{mixed tail}
  This is a direct matrix generalization of a classical Bernstein's inequality for bounded random variables. 
  To compare it with our version of Bernstein's inequality for sub-exponentials, 
  Proposition~\ref{sub-exponential large deviations},
  note that the probability bound in \eqref{eq matrix Bernstein} is equivalent to  
  $2n \cdot \exp \big[ -c \min \big( \frac{t^2}{\s^2}, \frac{t}{K} \big) \big]$ where $c>0$ is an absolute constant. 
  In both results we see a mixture of gaussian and exponential tails.
\end{remark}

\section{Random matrices with independent entries}				\label{s: entries}

We are ready to study the extreme singular values of random matrices. 
In this section, we consider the classical model of random matrices whose entries are independent and centered
random variables. Later we will study the more difficult models where 
only the rows or the columns are independent. 

The reader may keep in mind some classical examples of $N \times n$ random matrices with independent entries.
The most classical example is the {\em Gaussian random matrix} $A$ \index{Gaussian!random matrices}
whose entries are independent standard normal random variables. In this case, 
the $n \times n$ symmetric matrix $A^*A$ is called Wishart matrix; it is a higher-dimensional version of 
chi-square distributed random variables.

The simplest example of discrete random matrices is the {\em Bernoulli random matrix} $A$ 
\index{Bernoulli!random matrices} whose entries 
are independent symmetric Bernoulli random variables. In other words, Bernoulli random matrices are distributed
uniformly in the set of all $N \times n$ matrices with $\pm 1$ entries.

\subsection{Limit laws and Gaussian matrices}

Consider an $N \times n$ random matrix $A$ whose entries are independent centered identically distributed
random variables.  By now, the {\em limiting behavior} of the extreme singular values of $A$, 
as the dimensions $N, n \to \infty$, is well understood:

\begin{theorem}[Bai-Yin's law, see \cite{Bai-Yin}]      \index{Bai-Yin's law}   \label{Bai-Yin}
  Let $A = A_{N,n}$ be an $N \times n$ random matrix whose entries
  are independent copies of a random variable with zero mean, unit variance,
  and finite fourth moment. Suppose that the dimensions $N$ and $n$ grow to infinity
  while the aspect ratio $n/N$ converges to a constant in $[0,1]$.
  Then
  $$
  \smin(A) = \sqrt{N} - \sqrt{n} + o(\sqrt{n}), \quad
  \smax(A) = \sqrt{N} + \sqrt{n} + o(\sqrt{n}) \quad
  \text{almost surely}.
  $$
\end{theorem}

As we pointed out in the introduction, our program is to find non-asymptotic 
versions of Bai-Yin's law. There is precisely one model of random matrices, namely Gaussian, 
where an {\em exact} non-asymptotic result is known: 

\begin{theorem}[Gordon's theorem for Gaussian matrices] \index{Gordon's theorem} \label{Gaussian}
  Let $A$ be an $N \times n$ matrix whose entries
  are independent standard normal random variables. Then
  $$
  \sqrt{N} - \sqrt{n} \le \E \smin(A) \le \E \smax(A) \le \sqrt{N} + \sqrt{n}.
  $$
\end{theorem}

The proof of the upper bound, which we borrowed from \cite{DS}, is based 
on Slepian's comparison inequality for Gaussian processes.\footnote{Recall that a Gaussian process $(X_t)_{t \in T}$
is a collection of centered normal random variables $X_t$ on the same probability space, indexed by
points $t$ in an abstract set $T$.}

\begin{lemma}[Slepian's inequality, see \cite{Ledoux-Talagrand} Section 3.3]	\index{Slepian's inequality} \label{Slepian}
  Consider two Gaussian processes $(X_t)_{t \in T}$ and $(Y_t)_{t \in T}$
  whose increments satisfy the inequality
  $\E |X_s - X_t|^2 \le \E |Y_s - Y_t|^2$ for all $s,t \in T$. 
  Then
  $\E \sup_{t \in T} X_t \le \E \sup_{t \in T} Y_t$.
\end{lemma}

\begin{proof}[Proof of Theorem~\ref{Gaussian}]
We recognize
$\smax(A) = \max_{u \in S^{n-1}, \; v \in S^{N-1}} \< Au, v\> $
to be the supremum of the Gaussian process $X_{u,v} = \< Au, v\> $ indexed by the pairs
of vectors $(u,v) \in S^{n-1} \times S^{N-1}$. We shall compare this process to the
following one whose supremum is easier to estimate:
$Y_{u,v} = \< g, u\>  + \< h, v\> $ where $g \in \R^n$ and $h \in \R^N$
are independent standard Gaussian random vectors.
The rotation invariance of Gaussian measures makes it easy to compare
the increments of these processes. For every $(u,v), (u',v') \in S^{n-1} \times S^{N-1}$,
one can check that 
$$
\E |X_{u,v} - X_{u',v'}|^2
= \sum_{i=1}^n \sum_{j=1}^N |u_i v_j - u'_i v'_j|^2
\le \|u - u'\|_2^2 + \|v - v'\|_2^2
= \E |Y_{u,v} - Y_{u',v'}|^2.
$$
Therefore Lemma~\ref{Slepian} applies, and it yields
the required bound
$$
\E \smax(A) = \E \max_{(u,v)} X_{u,v}
\le \E \max_{(u,v)}Y_{u,v}
= \E \|g\|_2 + \E \|h\|_2
\le \sqrt{N} + \sqrt{n}.
$$
Similar ideas are used to estimate
$\E \smin(A) = \E \max_{v \in S^{N-1}} \min_{u \in S^{n-1}} \< Au, v\> $,
see \cite{DS}.
One uses in this case Gordon's generalization of Slepian's
inequality for minimax of
Gaussian processes \cite{Gordon 84, Gordon 85, Gordon 92}, see \cite[Section 3.3]{Ledoux-Talagrand}.
\end{proof}

While Theorem~\ref{Gaussian} is about the expectation of singular values, it also yields
a large deviation inequality for them. 
It can be deduced formally by using the {\em concentration of measure} 
in the Gauss space.

\begin{proposition}[Concentration in Gauss space, see \cite{Ledoux}]
  \index{Concentration of meaure}		\label{Gaussian concentration}
  Let $f$ be a real valued Lipschitz function on $\R^n$ with Lipschitz constant $K$, i.e. 
  $|f(x)-f(y)| \le K \|x-y\|_2$ for all $x,y \in \R^n$ (such functions are also called $K$-Lipschitz). 
  Let $X$ be the standard normal random 
  vector in $\R^n$. Then for every $t \ge 0$ one has
  $$
  \P \{ f(X) - \E f(X) > t \} \le \exp(-t^2/2K^2).
  $$
\end{proposition}

\begin{corollary}[Gaussian matrices, deviation; see \cite{DS}] 	\index{Gaussian!random matrices} \label{Gaussian deviation}
  Let $A$ be an $N \times n$ matrix whose entries
  are independent standard normal random variables. 
  Then for every $t \ge 0$, with probability at least $1 - 2 \exp(-t^2/2)$ one has
  $$
  \sqrt{N} - \sqrt{n} - t \le \smin(A) \le \smax(A) \le
  \sqrt{N} + \sqrt{n} + t. 
  $$
\end{corollary}

\begin{proof}
Note that $\smin(A)$, $\smax(A)$ are $1$-Lipschitz functions of matrices $A$ considered
as vectors in $\R^{Nn}$. The conclusion now follows from the estimates on the expectation
(Theorem~\ref{Gaussian}) and Gaussian concentration (Proposition~\ref{Gaussian concentration}). 
\end{proof}

Later in these notes, we find it more convenient to work with the $n \times n$ 
positive-definite symmetric matrix $A^*A$ rather than with the original $N \times n$ matrix $A$.
Observe that the normalized matrix $\bar{A} = \frac{1}{\sqrt{N}} A$ is an approximate isometry 
(which is our goal) if and only if $\bar{A}^*\bar{A}$ is an approximate identity:

\begin{lemma}[Approximate isometries]	\index{Approximate isometries} \label{approximate isometries}
  Consider a matrix $B$ that satisfies 
  \begin{equation}							\label{B*B}
  \|B^*B - I\| \le \max(\d,\d^2)
  \end{equation}
  for some $\d > 0$. Then 
  \begin{equation}							\label{smin smax B}
  1-\d \le \smin(B) \le \smax(B) \le 1+\d.
  \end{equation}
  Conversely, if $B$ satisfies \eqref{smin smax B} for some $\d > 0$ then
  $\|B^*B - I\| \le 3 \max(\d,\d^2)$.
\end{lemma}

\begin{proof}
Inequality \eqref{B*B} holds if and only if 
$\big| \|Bx\|_2^2 - 1 \big| \le \max(\d,\d^2)$ for all $x \in S^{n-1}$. 
Similarly, \eqref{smin smax B} holds if and only if 
$\big| \|Bx\|_2 - 1 \big| \le \d$ for all $x \in S^{n-1}$.
The conclusion then follows from the elementary inequality
$$
\max(|z-1|, |z-1|^2) \le |z^2-1| \le 3 \max(|z-1|, |z-1|^2) \quad \text{for all } z \ge 0.  \qedhere
$$ 
\end{proof}

Lemma~\ref{approximate isometries} reduces our task of proving 
inequalities \eqref{heuristic} to showing an equivalent (but often more convenient) 
bound 
$$
\big\| \frac{1}{N} A^*A-I \big\| \le \max(\d, \d^2)
\quad \text{where } \d = O(\sqrt{n/N}).
$$

\subsection{General random matrices with independent entries}

Now we pass to a more general model of random matrices whose entries
are independent centered random variables with some general distribution 
(not necessarily normal). The largest singular value (the spectral norm) 
can be estimated by Latala's theorem for general random matrices with non-identically
distributed entries:

\begin{theorem}[Latala's theorem \cite{Latala}]     \index{Latala's theorem}           \label{Latala}
Let $A$ be a random matrix whose entries $a_{ij}$ are independent centered
random variables with finite fourth moment. Then
$$
\E \smax(A) \le C \Big[ \max_i \big( \sum_j \E a_{ij}^2 \big)^{1/2}
  +  \max_j \big( \sum_i \E a_{ij}^2 \big)^{1/2}
  +  \big( \sum_{i,j} \E a_{ij}^4 \big)^{1/4}  \Big].
$$
\end{theorem}

If the variance and the fourth moments of the entries are uniformly bounded,
then Latala's result yields
$\smax(A) = O(\sqrt{N} + \sqrt{n})$. This is slightly weaker than our goal \eqref{heuristic},
which is $\smax(A) = \sqrt{N} + O(\sqrt{n})$ but still satisfactory for most applications.
Results of the latter type will appear later in the more general model of random matrices
with independent rows or columns.

Similarly, our goal \eqref{heuristic} for the smallest singular value is $\smin(A) \ge \sqrt{N} - O(\sqrt{n})$. 
Since the singular values are non-negative anyway, such inequality would only be useful 
for sufficiently tall matrices, $N \gg n$. For almost square and square matrices, estimating 
the smallest singular value (known also as the {\em hard edge} of spectrum) is considerably more difficult. 
The progress on estimating the hard edge is summarized in \cite{RV ICM}.
If $A$ has independent entries, then indeed $\smin(A) \ge c (\sqrt{N} - \sqrt{n})$,
and the following is an optimal probability bound:

\begin{theorem}[Independent entries, hard edge \cite{RV rectangular}]	\index{Hard edge of spectrum}	
  \label{RV rectangular}
  Let $A$ be an $n \times n$ random matrix whose entries are independent
  identically distributed subgaussian random variables
  with zero mean and unit variance.
  Then for $\e \ge 0$,
  $$
  \P \big( \smin(A) \le \e (\sqrt{N} - \sqrt{n-1}) \big) \le (C\e)^{N-n+1} +  c^N  
  $$
  where $C > 0$  and $c \in (0,1)$ depend only on the subgaussian
  norm of the entries.
\end{theorem}

This result gives an optimal bound for square matrices as well ($N=n$).

\section{Random matrices with independent rows}				\label{s: rows}

In this section, we focus on a more general model of random matrices, where we only 
assume independence of the rows rather than all entries. 
Such matrices are naturally
{\em generated by high-dimensional distributions}.
Indeed, given an arbitrary probability distribution in $\R^n$, one takes
a sample of $N$ independent points and arranges them as the rows of an $N \times n$ matrix $A$.
By studying spectral properties of $A$ one should be able to learn something useful about the 
underlying distribution. For example, as we will see in Section~\ref{s: covariance}, 
the extreme singular values of $A$ would tell us whether the covariance matrix 
of the distribution can be estimated from a sample of size $N$.

The picture will vary slightly depending on whether the rows of $A$ 
are sub-gaussian or have arbitrary distribution. For heavy-tailed distributions, an extra 
logarithmic factor has to appear in our desired inequality \eqref{heuristic}. 
The analysis of sub-gaussian and heavy-tailed matrices will be completely different.

There is an abundance of examples where the results of 
this section may be useful. They include all matrices with independent entries,
whether sub-gaussian such as Gaussian and Bernoulli, or completely general
distributions with mean zero and unit variance. In the latter case
one is able to surpass the fourth moment assumption which is 
necessary in Bai-Yin's law, Theorem~\ref{Bai-Yin}.

Other examples of interest come from non-product distributions, some of which we saw
in Example~\ref{random vectors}. Sampling from discrete objects (matrices and frames)
fits well in this framework, too. Given a deterministic matrix $B$, one puts a uniform distribution on 
the set of the rows of $B$ and creates 
a random matrix $A$ as before -- by sampling some $N$ random rows from $B$. 
Applications to sampling will be discussed in Section~\ref{s: sub-matrices}.

\subsection{Sub-gaussian rows}	  
\index{Sub-gaussian!random matrices with independent rows}		\label{s: sub-gaussian rows}

The following result goes in the direction of our goal \eqref{heuristic} for 
random matrices with independent sub-gaussian rows. 

\begin{theorem}[Sub-gaussian rows]		\label{sub-gaussian rows}
  Let $A$ be an $N \times n$ matrix whose rows $A_i$ are independent
  sub-gaussian isotropic random vectors in $\R^n$.
  Then for every $t \ge 0$, with probability at least $1 - 2\exp(-ct^2)$ one has
  \begin{equation}							\label{smin smax rectangular}
  \sqrt{N} - C \sqrt{n} - t \le \smin(A) \le \smax(A) \le \sqrt{N} + C \sqrt{n} + t.
  \end{equation}
  Here $C = C_K$, $c = c_K > 0$ depend only on the subgaussian norm 
  $K = \max_i \|A_i\|_\psitwo$ of the rows.
\end{theorem}

This result is a general version of Corollary~\ref{Gaussian deviation} (up to absolute constants); 
instead of independent Gaussian entries we allow independent sub-gaussian rows. 
This of course covers all matrices with independent sub-gaussian entries
such as Gaussian and Bernoulli. It also applies to some natural matrices whose entries
are not independent. One such example is a matrix whose rows are independent spherical 
random vectors (Example~\ref{random vectors sub-gaussian}).

\begin{proof}
The proof is a basic version of a {\em covering argument}, \index{Covering argument} and it has three steps. 
We need to control $\|Ax\|_2$ for all
vectors $x$ on the unit sphere $S^{n-1}$. To this end, we discretize the sphere using a net $\NN$
(the approximation step), establish a tight control of $\|Ax\|_2$ for every fixed vector $x \in \NN$ 
with high probability (the concentration step), and finish off by taking a union bound over all
$x$ in the net. The concentration step will be based on the deviation inequality 
for sub-exponential random variables, Corollary~\ref{average sub-exponentials}.

{\bf Step 1: Approximation.}
Recalling Lemma~\ref{approximate isometries} for the matrix $B=A/\sqrt{N}$ we see 
that the conclusion of the theorem is equivalent to 
\begin{equation}							\label{A*A rows}
\big\| \frac{1}{N}A^*A-I \big\| \le \max(\d, \d^2) =:\e
\quad \text{where} \quad
\d = C \sqrt{\frac{n}{N}} + \frac{t}{\sqrt{N}}.
\end{equation}
Using Lemma~\ref{norm on net}, we can evaluate the operator norm in \eqref{A*A rows}
on a $\frac{1}{4}$-net $\NN$ of the unit sphere $S^{n-1}$:
$$
\big\| \frac{1}{N}A^*A-I \big\|
\le 2 \max_{x \in \NN} \big| \big\langle (\frac{1}{N}A^*A-I)x, x \big\rangle \big|
= 2 \max_{x \in \NN} \big| \frac{1}{N} \|Ax\|_2^2 - 1 \big|.
$$
So to complete the proof it suffices to show that, with the required probability, 
$$
\max_{x \in \NN} \big| \frac{1}{N} \|Ax\|_2^2 - 1 \big| \le  \frac{\e}{2}.
$$
By Lemma~\ref{net cardinality}, we can choose the net $\NN$ so that it has cardinality
$|\NN| \le 9^n$.

{\bf Step 2: Concentration.} 
Let us fix any vector $x \in S^{n-1}$. We can express $\|Ax\|_2^2$ as a sum of independent 
random variables 
\begin{equation}							\label{Ax as sum}
\|Ax\|_2^2 = \sum_{i=1}^N \< A_i, x\> ^2 =: \sum_{i=1}^N Z_i^2
\end{equation}
where $A_i$ denote the rows of the matrix $A$.
By assumption, $Z_i = \< A_i, x\> $ are independent sub-gaussian random variables 
with $\E Z_i^2 = 1$ and $\|Z_i\|_\psitwo \le K$. 
Therefore, by Remark~\ref{centering} and Lemma~\ref{sub-exponential squared},
$Z_i^2 - 1$ are independent centered sub-exponential random variables with 
$\|Z_i^2-1\|_\psione \le 2\|Z_i^2\|_\psione \le 4 \|Z_i\|_\psitwo^2 \le 4 K^2$.

We can therefore use an exponential deviation inequality, Corollary~\ref{average sub-exponentials}, 
to control the sum \eqref{Ax as sum}. Since 
$K \ge \|Z_i\|_\psitwo \ge \frac{1}{\sqrt{2}} (\E|Z_i|^2)^{1/2} = \frac{1}{\sqrt{2}}$, this gives 
\begin{align*}
\P \Big\{ \big| \frac{1}{N} \|Ax\|_2^2 - 1 \big| \ge \frac{\e}{2} \Big\}
  &= \P \Big\{ \big| \frac{1}{N}\sum_{i=1}^N Z_i^2 - 1 \big| \ge \frac{\e}{2} \Big\} 
  \le 2 \exp \Big[ - \frac{c_1}{K^4} \min(\e^2, \e) N \Big] \\
  &= 2 \exp \Big[ - \frac{c_1}{K^4} \d^2 N \Big]
  \le 2 \exp \Big[ - \frac{c_1}{K^4} (C^2 n + t^2) \Big]
\end{align*}
where the last inequality follows by the definition of $\d$ 
and using the inequality $(a+b)^2 \ge a^2 + b^2$ for $a,b \ge 0$.

{\bf Step 3: Union bound.}
Taking the union bound over all vectors $x$ in the net $\NN$ of cardinality $|\NN| \le 9^n$,
we obtain 
$$
\P \Big\{ \max_{x \in \NN} \big| \frac{1}{N} \|Ax\|_2^2 - 1 \big| \ge \frac{\e}{2} \Big\}
\le 9^n \cdot 2 \exp \Big[ - \frac{c_1}{K^4} (C^2 n + t^2) \Big]
\le 2 \exp \Big( - \frac{c_1 t^2}{K^4} \Big)
$$
where the second inequality follows for $C = C_K$ sufficiently large, 
e.g. $C = K^2 \sqrt{\ln 9/c_1}$.
As we noted in Step~1, this completes the proof of the theorem.
\end{proof}

\begin{remark}[Non-isotropic distributions]				\label{r: non-isotropic}		
\begin{enumerate}		
\item A version of Theorem~\ref{sub-gaussian rows} holds for general,
  non-isotropic sub-gaussian distributions. 
  Assume that $A$ is an $N \times n$ matrix whose rows $A_i$ are independent
  sub-gaussian random vectors in $\R^n$ with second moment matrix $\Sigma$. 
  Then for every $t \ge 0$, the following inequality holds with probability at least $1 - 2\exp(-ct^2)$:
  \begin{equation}							\label{A*A rows non-isotropic}
  \big\| \frac{1}{N}A^*A-\Sigma \big\| \le \max(\d, \d^2)
  \quad \text{where} \quad
  \d = C \sqrt{\frac{n}{N}} + \frac{t}{\sqrt{N}}.
  \end{equation}
  Here as before $C = C_K$, $c = c_K > 0$ depend only on the subgaussian norm 
  $K = \max_i \|A_i\|_\psitwo$ of the rows. This result is a general version of \eqref{A*A rows}.
  It follows by a straighforward modification of the argument of Theorem~\ref{sub-gaussian rows}.
\item A more natural, multiplicative form of \eqref{A*A rows non-isotropic} is the following. 
  Assume that $\Sigma^{-1/2} A_i$ are isotropic sub-gaussian random vectors, and let $K$
  be the maximum of their sub-gaussian norms. Then 
  for every $t \ge 0$, the following inequality holds with probability at least $1 - 2\exp(-ct^2)$:
  \begin{equation}							\label{A*A rows non-isotropic multiplicative}
  \big\| \frac{1}{N}A^*A-\Sigma \big\| \le \max(\d, \d^2) \, \|\Sigma\|
  \quad \text{where} \quad
  \d = C \sqrt{\frac{n}{N}} + \frac{t}{\sqrt{N}}
  \end{equation}
  Here again $C = C_K$, $c = c_K > 0$. This result follows from Theorem~\ref{sub-gaussian rows}
  applied to the isotropic random vectors $\Sigma^{-1/2} A_i$.
\end{enumerate}
\end{remark}

\subsection{Heavy-tailed rows}
  \index{Heavy-tailed!random matrices with independent rows}			\label{s: heavy-tailed rows}

The class of sub-gaussian random variables in Theorem~\ref{sub-gaussian rows} may sometimes
be too restrictive in applications. For example, if the rows of $A$ 
are independent coordinate or frame random vectors 
(Examples~\ref{random vectors} and \ref{random vectors sub-gaussian}), 
they are poorly sub-gaussian and Theorem~\ref{sub-gaussian rows} is too weak.
In such cases, one would use the following result instead, which operates in remarkable generality.

\begin{theorem}[Heavy-tailed rows]	  \label{heavy-tailed rows}
  Let $A$ be an $N \times n$ matrix whose rows $A_i$ are independent
  isotropic random vectors in $\R^n$. Let $m$ be a number such that 
  $\|A_i\|_2 \le \sqrt{m}$ almost surely for all $i$. 
  Then for every $t \ge 0$, one has
  \begin{equation}							\label{eq heavy-tailed rows}
  \sqrt{N} - t \sqrt{m} \le \smin(A) \le \smax(A) \le \sqrt{N} + t \sqrt{m}
  \end{equation}
  with probability at least $1 - 2 n \cdot \exp(-ct^2)$, 
  where $c>0$ is an absolute constant.
\end{theorem}

Recall that $(\E \|A_i\|_2^2)^{1/2} = \sqrt{n}$ by Lemma~\ref{norm isotropic}. 
This indicates that one would typically use Theorem~\ref{heavy-tailed rows} with $m = O(n)$. 
In this case the result takes the form 
\begin{equation}							\label{heavy-tailed m=n}
\sqrt{N} - t \sqrt{n} \le \smin(A) \le \smax(A) \le \sqrt{N} + t \sqrt{n}
\end{equation}
with probability at least $1 - 2n \cdot \exp(-c't^2)$. This is a form of our desired inequality
\eqref{heuristic} for heavy-tailed matrices. We shall discuss this more after the proof.

\begin{proof}
We shall use the non-commutative Bernstein's inequality, Theorem~\ref{matrix Bernstein}.

{\bf Step 1: Reduction to a sum of independent random matrices.}
We first note that $m \ge n \ge 1$ since by Lemma~\ref{norm isotropic} we have $\E \|A_i\|_2^2 = n$.
Now we start an argument parallel to Step~1 of Theorem~\ref{sub-gaussian rows}. 
Recalling Lemma~\ref{approximate isometries} for the matrix $B=A/\sqrt{N}$ we see 
that the desired inequalities \eqref{eq heavy-tailed rows} are equivalent to 
\begin{equation}							\label{A*A heavy-tailed}
\big\| \frac{1}{N}A^*A-I \big\| \le \max(\d, \d^2) =:\e
\quad \text{where} \quad
\d = t \sqrt{\frac{m}{N}}.
\end{equation}
We express this random matrix as a sum of independent random matrices:
$$
\frac{1}{N} A^*A - I = \frac{1}{N} \sum_{i=1}^N A_i \otimes A_i - I  
= \sum_{i=1}^N X_i, 
\quad \text{where } X_i := \frac{1}{N} (A_i \otimes A_i - I);
$$
note that $X_i$ are independent centered $n \times n$ random matrices.

{\bf Step 2: Estimating the mean, range and variance.}
We are going to apply the non-commutative Bernstein inequality, Theorem~\ref{matrix Bernstein}, for the sum $\sum_i X_i$.
Since $A_i$ are isotropic random vectors, we have $\E A_i \otimes A_i = I$
which implies that $\E X_i = 0$ as required in the non-commutative Bernstein inequality.

We estimate the range of $X_i$ using that $\|A_i\|_2 \le \sqrt{m}$ and $m \ge 1$:
$$
\|X_i\| 
\le \frac{1}{N} ( \|A_i \otimes A_i\| + 1)
= \frac{1}{N} (\|A_i\|_2^2 + 1) 
\le \frac{1}{N} (m + 1)
\le \frac{2 m}{N}
=: K
$$
To estimate the total variance $\|\sum_i \E X_i^2\|$, we first compute
$$
X_i^2 = \frac{1}{N^2} \big[ (A_i \otimes A_i)^2 - 2(A_i \otimes A_i) + I \big],
$$
so using that the isotropy assumption $\E A_i \otimes A_i = I$ we obtain
\begin{equation}							\label{Xi squared}
\E X_i^2 = \frac{1}{N^2} \big[ \E (A_i \otimes A_i)^2 - I \big].
\end{equation}
Since $(A_i \otimes A_i)^2 = \|A_i\|_2^2 \, A_i \otimes A_i$ is a positive semi-definite matrix 
and $\|A_i\|_2^2 \le m$ by assumption, we have
$\big\| \E (A_i \otimes A_i)^2 \big\| \le m \cdot \| \E A_i \otimes A_i \| = m$.
Putting this into \eqref{Xi squared} we obtain
$$
\| \E X_i^2 \| \le \frac{1}{N^2} (m + 1) \le \frac{2 m}{N^2}
$$
where we again used that $m \ge 1$. 
This yields\footnote{Here the seemingly crude application of triangle inequality is actually not 
so loose. If the rows $A_i$ are identically distributed, then so are $X_i^2$, 
which makes the triangle inequality above into an equality.}
$$
\Big\| \sum_{i=1}^N \E X_i^2 \Big\| 
\le N \cdot \max_i \| \E X_i^2 \|
= \frac{2m}{N}
=: \s^2.
$$

{\bf Step 3: Application of the non-commutative Bernstein's inequality.}
\index{Bernstein-type inequality!non-commutative}
Applying Theorem~\ref{matrix Bernstein} (see Remark~\ref{mixed tail})
and recalling the definitions of $\e$ and $\d$ in \eqref{A*A heavy-tailed}, we 
we bound the probability in question as 
\begin{align*}
\P &\Big\{ \Big\| \frac{1}{N} A^*A - I \Big\| \ge \e \Big\}
= \P \Big\{ \Big\| \sum_{i=1}^N X_i \Big\| \ge \e \Big\} 
\le 2n \cdot \exp \Big[ -c \min \Big( \frac{\e^2}{\s^2}, \frac{\e}{K} \Big) \Big] \\
&\le 2n \cdot \exp \Big[ -c \min(\e^2,\e) \cdot \frac{N}{2m} \Big] 
= 2n \cdot \exp \Big( - \frac{c \d^2 N}{2m} \Big)
= 2n \cdot \exp(-ct^2/2).
\end{align*}
This completes the proof. 
\end{proof}

Theorem~\ref{heavy-tailed rows} for heavy-tailed rows is different from 
Theorem~\ref{sub-gaussian rows} for sub-gaussian rows in two ways: 
the boundedness assumption\footnote{Going a little 
  ahead, we would like to point out that the almost sure boundedness can be relaxed to 
  the bound in expectation $\E \max_i \|A_i\|_2^2 \le m$,
  see Theorem~\ref{heavy-tailed rows exp si}.}  
$\|A_i\|_2^2 \le m$ appears, and the probability bound is weaker. 
We will now comment on both differences. 

\begin{remark}[Boundedness assumption]				\label{r: boundedness}
  Observe that some boundendess assumption on the distribution
  is needed in Theorem~\ref{heavy-tailed rows}.
  Let us see this on the following example. Choose $\d \ge 0$ arbitrarily small, and 
  consider a random vector $X = \d^{-1/2} \xi Y$ in $\R^n$
  where $\xi$ is a $\{0,1\}$-valued random variable with $\E \xi = \d$ (a ``selector'')
  and $Y$ is an independent isotropic random vector in $\R^n$ with an arbitrary distribution.
  Then $X$ is also an isotropic random vector. 
  Consider an $N \times n$ random matrix $A$ whose rows $A_i$ are independent copies of $X$.  
  However, if $\d \ge 0$ is suitably small then $A = 0$ with high probability,
  hence no nontrivial lower bound on $\smin(A)$ is possible. 
\end{remark}

Inequality \eqref{heavy-tailed m=n} fits our goal \eqref{heuristic}, but not quite. The reason is that the probability 
bound is only non-trivial if $t \ge C \sqrt{\log n}$. Therefore, in reality Theorem~\ref{heavy-tailed rows} 
asserts that 
\begin{equation}							\label{goal log}
\sqrt{N} - C\sqrt{n \log n} \le \smin(A) \le \smax(A) \le \sqrt{N} + C\sqrt{n \log n}
\end{equation}
with probability, say $0.9$. This achieves our goal \eqref{heuristic} up to a logarithmic factor. 

\begin{remark}[Logarithmic factor] 
  The logarithmic factor can not be removed from \eqref{goal log} for some heavy-tailed distributions.
  Consider for instance the coordinate distribution introduced in Example~\ref{random vectors}. 
  In order that $\smin(A) > 0$ there must be no zero columns in $A$. Equivalently, each coordinate vector
  $e_1, \ldots,e_n$ \index{Coordinate random vectors} must be picked at least once in $N$ independent trials 
  (each row of $A$ picks an independent coordinate vector). 
  Recalling the classical coupon collector's problem, one must make at least $N \ge C n \log n$ trials to make this occur
  with high probability. Thus the logarithm is necessary in the left hand side of \eqref{goal log}.\footnote{This argument 
  moreover shows the optimality of the probability bound in Theorem~\ref{heavy-tailed rows}.
  For example, for $t = \sqrt{N}/2\sqrt{n}$ the conclusion \eqref{heavy-tailed m=n} implies 
  that $A$ is well conditioned (i.e. $\sqrt{N}/2 \le \smin(A) \le \smax(A) \le 2 \sqrt{N}$) 
  with probability $1 - n \cdot \exp(-cN/n)$. 
  On the other hand, by the coupon collector's problem we estimate the probability that $\smin(A) > 0$ as
  $1 - n \cdot (1- \frac{1}{n})^N \approx 1 - n \cdot \exp(-N/n)$.}
\end{remark}

A version of Theorem~\ref{heavy-tailed rows} holds for general, non-isotropic distributions.
It is convenient to state it in terms of the equivalent estimate \eqref{A*A heavy-tailed}:

\begin{theorem}[Heavy-tailed rows, non-isotropic]					\label{heavy-tailed rows non-isotropic}
  Let $A$ be an $N \times n$ matrix whose rows $A_i$ are independent
  random vectors in $\R^n$ with the common second moment matrix $\Sigma = \E A_i \otimes A_i$. 
  Let $m$ be a number such that $\|A_i\|_2 \le \sqrt{m}$ almost surely for all $i$. 
  Then for every $t \ge 0$, the following inequality holds with probability at least $1 - n \cdot \exp(-ct^2)$:
  \begin{equation}							\label{A*A heavy-tailed rows non-isotropic}
  \big\| \frac{1}{N}A^*A-\Sigma \big\| \le \max(\|\Sigma\|^{1/2}\d, \d^2)
  \quad \text{where} \quad
  \d = t \sqrt{\frac{m}{N}}.
  \end{equation}
  Here $c>0$ is an absolute constant.
  In particular, this inequality yields
  \begin{equation}							\label{A heavy-tailed rows non-isotropic}
  \|A\| \le \|\Sigma\|^{1/2} \sqrt{N} + t \sqrt{m}.
  \end{equation}
\end{theorem}

\begin{proof}
We note that $m \ge \|\Sigma\|$ because 
$\|\Sigma\| = \|\E A_i \otimes A_i\| \le \E \|A_i \otimes A_i\| = \E \|A_i\|_2^2 \le m$. 
Then \eqref{A*A heavy-tailed rows non-isotropic} follows by a straightforward modification of the argument
of Theorem~\ref{heavy-tailed rows}. Furthermore, if \eqref{A*A heavy-tailed rows non-isotropic} holds then
by triangle inequality
\begin{align*}
\frac{1}{N} \|A\|^2 
&= \big\| \frac{1}{N} A^*A \big\|
\le \|\Sigma\| + \big\| \frac{1}{N}A^*A-\Sigma \big\| \\
&\le \|\Sigma\| + \|\Sigma\|^{1/2}\d + \d^2 
\le (\|\Sigma\|^{1/2} + \d)^2. 
\end{align*}
Taking square roots and multiplying both sides by $\sqrt{N}$, we obtain \eqref{A heavy-tailed rows non-isotropic}.
\end{proof}

\bigskip

The {\em almost sure} boundedness requirement in Theorem~\ref{heavy-tailed rows} may sometimes be too
restrictive in applications, and it can be relaxed to a bound {\em in expectation}:

\begin{theorem}[Heavy-tailed rows; expected singular values]					\label{heavy-tailed rows exp si}
  Let $A$ be an $N \times n$ matrix whose rows $A_i$ are independent
  isotropic random vectors in $\R^n$. Let 
  $m := \E \max_{i \le N} \|A_i\|_2^2$. Then
  $$
  \E \max_{j \le n} |s_j(A) - \sqrt{N}|
  \le C \sqrt{m \log \min(N,n)}
  $$
  where $C$ is an absolute constant.
\end{theorem}

The proof of this result is similar to that of Theorem~\ref{heavy-tailed rows}, except that this time
we will use Rudelson's Corollary~\ref{Rudelson} instead of matrix Bernstein's inequality.
To this end, we need a link to symmetric Bernoulli random variables. This is provided by 
a general {\em symmetrization argument}: 

\begin{lemma}[Symmetrization] \index{Symmetrization}				\label{symmetrization}
  Let $(X_i)$ be a finite sequence of independent random vectors valued in some Banach space,
  and $(\e_i)$ be independent symmetric Bernoulli random variables. 
  Then 
  \begin{equation}							\label{eq symmetrization}
  \E \Big\| \sum_i (X_i - \E X_i) \Big\|
  \le 2 \E \Big\| \sum_i \e_i X_i \Big\|.
  \end{equation}
\end{lemma}

\begin{proof}
We define random variables $\tilde{X}_i = X_i - X_i'$ 
where $(X_i')$ is an independent copy of the sequence
$(X_i)$. 
Then $\tilde{X}_i$ are independent symmetric random variables, i.e. the sequence 
$(\tilde{X}_i)$ is distributed
identically with $(-\tilde{X}_i)$ and thus also with $(\e_i \tilde{X}_i)$.
Replacing $\E X_i$ by $\E X_i'$ in \eqref{eq symmetrization} and using 
Jensen's inequality, symmetry, and triangle inequality, we obtain the required inequality
\begin{align*}
\E \Big\| \sum_i (X_i - \E X_i) \Big\|
  &\le \E \Big\| \sum_i \tilde{X}_i \Big\| 
  = \E \Big\| \sum_i \e_i \tilde{X}_i \Big\| \\
  &\le \E \Big\| \sum_i \e_i X_i \Big\| + \E \Big\| \sum_i \e_i X_i' \Big\| 
  = 2 \E \Big\| \sum_i \e_i X_i \Big\|. 	\qedhere
\end{align*}
\end{proof}

We will also need a probabilistic version of Lemma~\ref{approximate isometries} on approximate isometries.
The proof of that lemma was based on the elementary inequality
$|z^2-1| \ge \max( |z-1|, |z-1|^2 )$ for $z \ge 0$. Here is a probabilistic version:

\begin{lemma}				\label{deviation from 1}
  Let $Z$ be a non-negative random variable. 
  Then $\E|Z^2-1| \ge \max( \E|Z-1|, (\E|Z-1|)^2 )$.
\end{lemma}

\begin{proof}
Since $|Z-1| \le |Z^2-1|$ pointwise, we have $\E |Z-1| \le \E |Z^2-1|$.
Next, since $|Z-1|^2 \le |Z^2-1|$ pointwise,  
taking square roots and expectations we obtain
$\E|Z-1| \le \E|Z^2-1|^{1/2} \le (\E|Z^2-1|)^{1/2}$, where the last bound follows by Jensen's inequality.
Squaring both sides 
completes the proof.
\end{proof}

\begin{proof}[Proof of Theorem~\ref{heavy-tailed rows exp si}]
{\bf Step 1: Application of Rudelson's inequality.} \index{Rudelson's inequality}
As in the proof of Theorem~\ref{heavy-tailed rows}, we are going to control
$$
E := \E \big\| \frac{1}{N} A^*A - I \big\| 
= \E \Big\| \frac{1}{N} \sum_{i=1}^N A_i \otimes A_i - I \Big\|
\le  \frac{2}{N} \, \E \Big\| \sum_{i=1}^N \e_i A_i \otimes A_i \Big\|
$$
where we used Symmetrization Lemma~\ref{symmetrization}
with independent symmetric Bernoulli random variables $\e_i$ 
(which are independent of $A$ as well). The expectation in the right hand side is taken both with respect 
to the random matrix $A$ and the signs $(\e_i)$. Taking first the expectation with respect to $(\e_i)$
(conditionally on $A$) and afterwards the expectation with respect to $A$,
we obtain by Rudelson's inequality (Corollary~\ref{Rudelson}) that
$$
E \le \frac{C \sqrt{l}}{N} \, 
  \E \Big( \max_{i \le N} \|A_i\|_2 \cdot \Big\| \sum_{i=1}^N A_i \otimes A_i \Big\|^{1/2} \Big)
$$
where $l = \log \min(N,n)$.
We now apply the Cauchy-Schwarz inequality. Since by the triangle inequality 
$\E \big\| \frac{1}{N} \sum_{i=1}^N A_i \otimes A_i \big\| = \E \big\| \frac{1}{N} A^*A \big\| \le E + 1$,
it follows that
$$
E \le C \sqrt{\frac{ml}{N}} (E+1)^{1/2}.
$$
This inequality is easy to solve in $E$. Indeed, 
considering the cases $E \le 1$ and $E > 1$ separately, we conclude that
$$
E = \E \big\| \frac{1}{N} A^*A - I \big\| \le \max(\d, \d^2)
\quad \text{where } \d := C \sqrt{\frac{2ml}{N}}.
$$

{\bf Step 2: Diagonalization.}
Diagonalizing the matrix $A^*A$ one checks that 
$$
\big\| \frac{1}{N} A^*A - I \big\|
= \max_{j \le n} \big| \frac{s_j(A)^2}{N} - 1 \big|
= \max \Big( \big| \frac{\smin(A)^2}{N} - 1 \big|,
  \big| \frac{\smax(A)^2}{N} - 1 \big| \Big).
$$
It follows that
$$
\max \Big( \E \big| \frac{\smin(A)^2}{N} - 1 \big|,
  \E \big| \frac{\smax(A)^2}{N} - 1 \big| \Big)
\le \max(\d,\d^2).
$$
(we replaced the expectation of maximum by the maximum of expectations). 
Using Lemma~\ref{deviation from 1} separately for the two terms on the left hand side, we obtain
$$
\max \Big( \E \big| \frac{\smin(A)}{\sqrt{N}} - 1 \big|,
  \E \big| \frac{\smax(A)}{\sqrt{N}} - 1 \big| \Big)
\le \d.
$$
Therefore 
\begin{align*}
\E \max_{j \le n} \big| \frac{s_j(A)}{\sqrt{N}}-1 \big|
&= \E \max \Big( \big| \frac{\smin(A)}{\sqrt{N}} - 1 \big|,
  \Big| \frac{\smax(A)}{\sqrt{N}} - 1 \Big| \Big) \\
&\le \E \Big( \big| \frac{\smin(A)}{\sqrt{N}} - 1 \big|
  + \big| \frac{\smax(A)}{\sqrt{N}} - 1 \big| \Big)
\le 2\d.
\end{align*}
Multiplying both sides by $\sqrt{N}$ completes the proof.
\end{proof}

In a way similar to Theorem~\ref{heavy-tailed rows non-isotropic} we note that a version 
of Theorem~\ref{heavy-tailed rows exp si} holds for general, non-isotropic distributions.

\begin{theorem}[Heavy-tailed rows, non-isotropic, expectation]			\label{heavy-tailed rows exp si non-isotropic}
  Let $A$ be an $N \times n$ matrix whose rows $A_i$ are independent
  random vectors in $\R^n$ with the common second moment matrix $\Sigma = \E A_i \otimes A_i$. 
  Let $m := \E \max_{i \le N} \|A_i\|_2^2$.  
  Then
  $$
  \E \big\| \frac{1}{N}A^*A-\Sigma \big\| \le \max(\|\Sigma\|^{1/2}\d, \d^2)
  \quad \text{where} \quad
  \d = C \sqrt{\frac{m \log \min(N,n)}{N}}.
  $$
  Here $C$ is an absolute constant.
  In particular, this inequality yields
  $$  
  \big( \E \|A\|^2 \big)^{1/2} \le \|\Sigma\|^{1/2} \sqrt{N} + C \sqrt{m \log \min(N,n)}.
  $$
\end{theorem}

\begin{proof}
The first part follows by a simple modification of the proof of Theorem~\ref{heavy-tailed rows exp si}.
The second part follows from the first like in Theorem~\ref{heavy-tailed rows non-isotropic}.
\end{proof}

\begin{remark}[Non-identical second moments]						\label{r: different second moments}
  The assumption that the rows $A_i$ have a common
  second moment matrix $\S$ is not essential in Theorems~\ref{heavy-tailed rows non-isotropic}
  and \ref{heavy-tailed rows exp si non-isotropic}. 
  The reader will be able to formulate more general versions of these results. 
  For example, if $A_i$ have arbitrary second moment matrices $\Sigma_i = \E A_i \otimes A_i$ 
  then the conclusion of Theorem~\ref{heavy-tailed rows exp si non-isotropic} 
  holds with $\S = \frac{1}{N} \sum_{i=1}^N \Sigma_i$. 
\end{remark}

\subsection{Applications to estimating covariance matrices}			
\index{Covariance matrix!estimation}\label{s: covariance}

One immediate application of our analysis of random matrices is in statistics, 
for the fundamental problem of {\em estimating covariance matrices}. 
Let $X$ be a random vector in $\R^n$; for simplicity 
we assume that $X$ is centered,\footnote{More generally, in this section we estimate 
  the {\em second moment matrix} $\E X \otimes X$ of an arbitrary random vector $X$ 
  (not necessarily centered).}
$\E X = 0$. Recall that the covariance matrix of 
$X$ is the $n \times n$ matrix $\Sigma = \E X \otimes X$, see Section~\ref{s: isotropic}.

The simplest way to estimate $\Sigma$ is to take some $N$ independent 
samples $X_i$ from the distribution and form the {\em sample covariance matrix} \index{Sample covariance matrix}
$\Sigma_N = \frac{1}{N} \sum_{i=1}^N X_i \otimes X_i$.
By the law of large numbers, $\Sigma_N \to \Sigma$ almost surely as $N \to \infty$.
So, taking sufficiently many samples we are guaranteed to estimate the covariance matrix 
as well as we want. This, however, does not address the quantitative aspect: 
what is the minimal {\em sample size} $N$ that guarantees approximation with a given accuracy? 

The relation of this question to random matrix theory becomes clear when we 
arrange the samples $X_i =: A_i$ as rows of the $N \times n$ random matrix $A$.
Then the sample covariance matrix is expressed as $\Sigma_N  = \frac{1}{N}A^*A$. 
Note that $A$ is a matrix with independent rows but usually not independent entries (unless 
we sample from a product distribution). We worked out the analysis of such matrices
in Section~\ref{s: rows}, separately for sub-gaussian and general distributions.
As an immediate consequence of Theorem~\ref{sub-gaussian rows}, we obtain:

\begin{corollary}[Covariance estimation for sub-gaussian distributions]	\label{covariance sub-gaussian} \hfill
  Consider a sub-gaussian distribution in $\R^n$ with covariance matrix $\Sigma$,
  and let $\e \in (0,1)$, $t \ge 1$. Then with probability at least $1 - 2 \exp(- t^2 n)$ one has
  $$
  \text{If } N \ge C(t/\e)^2 n \quad \text{then } \|\Sigma_N - \Sigma\| \le \e.
  $$
  Here $C = C_K$ depends only on the sub-gaussian norm $K = \|X\|_\psitwo$ of a random vector
  taken from this distribution. 
\end{corollary}

\begin{proof}
It follows from \eqref{A*A rows non-isotropic} that for every $s \ge 0$, 
with probability at least $1 - 2\exp(-cs^2)$ we have 
$\|\Sigma_N-\Sigma\| \le \max(\d, \d^2)$ where
$\d = C \sqrt{n/N} + s/\sqrt{N}$.
The conclusion follows for $s = C' t \sqrt{n}$ where
$C' = C'_K$ is sufficiently large. 
\end{proof}

Summarizing, Corollary~\ref{covariance sub-gaussian} shows that the sample size
$$
N = O(n)
$$ suffices to approximate the covariance 
matrix of a sub-gaussian distribution in $\R^n$ by the sample covariance matrix. 

\begin{remark}[Multiplicative estimates, Gaussian distributions]
  A weak point of Corollary~\ref{covariance sub-gaussian} 
  is that the sub-gaussian norm $K$ may in turn depend on $\|\Sigma\|$.
    
  To overcome this drawback, instead of using \eqref{A*A rows non-isotropic}
  in the proof of this result one can use the multiplicative version \eqref{A*A rows non-isotropic multiplicative}.
  The reader is encouraged to state a general result that follows from this argument.
  We just give one special example for arbitrary {\em centered Gaussian distributions} in $\R^n$.
  For every $\e \in (0,1)$, $t \ge 1$, the following holds with probability at least $1 - 2 \exp(- t^2 n)$:
  $$
  \text{If } N \ge C(t/\e)^2 n \quad \text{then } \|\Sigma_N - \Sigma\| \le \e \|\Sigma\|.
  $$
  Here $C$ is an absolute constant. 
\end{remark}

Finally, Theorem~\ref{heavy-tailed rows non-isotropic} yields a similar estimation result for arbitrary distributions,
possibly heavy-tailed:

\begin{corollary}[Covariance estimation for arbitrary distributions]	\label{covariance heavy-tailed} \hfill
  Consider a distribution in $\R^n$ with covariance matrix $\Sigma$ and supported in 
  some centered Euclidean ball whose radius we denote $\sqrt{m}$.
  Let $\e \in (0,1)$ and $t \ge 1$. 
  Then the following holds with probability at least $1 - n^{-t^2}$:
  $$
  \text{If } N \ge C(t/\e)^2 \|\Sigma\|^{-1} m \log n 
  \quad \text{then } \|\Sigma_N - \Sigma\| \le \e \|\Sigma\|.
  $$
  Here $C$ is an absolute constant.
\end{corollary}

\begin{proof}
It follows from Theorem~\ref{heavy-tailed rows non-isotropic} that for every $s \ge 0$, 
with probability at least $1 - n \cdot \exp(-cs^2)$ we have
$\|\Sigma_N-\Sigma\| \le \max(\|\Sigma\|^{1/2}\d, \d^2)$
where $\d = s \sqrt{m/N}$.
Therefore, if $N \ge (s/\e)^2 \|\Sigma\|^{-1} m$ then $\|\Sigma_N - \Sigma\| \le \e \|\Sigma\|$.
The conclusion follows with $s = C' t \sqrt{\log n}$ where $C'$ is a sufficiently large 
absolute constant.
\end{proof}

Corollary~\ref{covariance heavy-tailed} is typically used with $m = O(\|\Sigma\| n)$. 
Indeed, if $X$ is a random vector chosen from the distribution in question, 
then its expected norm is easy to estimate: 
$\E \|X\|_2^2  = \tr(\Sigma) \le n \|\Sigma\|$. 
So, by Markov's inequality, most of the distribution is supported
in a centered ball of radius $\sqrt{m}$ where $m = O(n \|\Sigma\|)$. 
If all distribution is supported there, i.e. if $\|X\| = O(\sqrt{n \|\Sigma\|})$ almost surely, 
then the conclusion of Corollary~\ref{covariance heavy-tailed} holds 
with sample size $N \ge C(t/\e)^2 n \log n$.

\begin{remark}[Low-rank estimation]
  In certain applications, the distribution in $\R^n$ lies close to a low dimensional subspace.
  In this case, a smaller sample suffices for covariance estimation. The intrinsic dimension 
  of the distribution can be measured with the {\em effective rank} \index{Effective rank}
  of the matrix $\Sigma$, defined as
  $$
  r(\Sigma) = \frac{\tr(\Sigma)}{\|\Sigma\|}. 
  $$
  One always has $r(\Sigma) \le \rank(\Sigma) \le n$, and this bound is sharp. 
  For example, if $X$ is an isotropic random vector in $\R^n$ then $\S = I$ and $r(\Sigma) = n$. 
  A more interesting example is where $X$ takes values in some $r$-dimensional subspace $E$,
  and the restriction of the distribution of $X$ onto $E$ is isotropic. The latter means that 
  $\Sigma = P_E$, where $P_E$ denotes the orthogonal projection in $\R^n$ onto $E$.
  Therefore in this case $r(\Sigma) = r$. The effective rank is a stable quantity compared with the 
  usual rank. For distributions that are approximately low-dimenional, the effective rank is 
  still small.
    
  The effective rank $r = r(\Sigma)$ always controls the typical norm of $X$, 
  as $\E \|X\|_2^2  = \tr(\Sigma) = r \|\Sigma\|$.
  It follows by Markov's inequality that most of the distribution is supported in a ball of radius $\sqrt{m}$ where
  $m = O(r \|\Sigma\|)$. Assume that all of the distribution is supported there, i.e. if $\|X\| = O(\sqrt{r \|\Sigma\|})$ 
  almost surely. Then the conclusion of Corollary~\ref{covariance heavy-tailed} holds 
  with sample size $N \ge C(t/\e)^2 r \log n$.
\end{remark}

We can summarize this discussion in the following way: the sample size
$$
N = O(n \log n)
$$ 
suffices to approximate the covariance matrix of a general distribution in $\R^n$ by the sample covariance matrix. 
Furthermore, for distributions that are approximately low-dimensional, a smaller sample size is sufficient. 
Namely, if the effective rank of $\Sigma$ equals $r$ then a sufficient sample size is 
$$
N = O(r \log n).
$$

\begin{remark}[Boundedness assumption]				\label{r: covariance boundedness}
  Without the boundedness assumption on the distribution, Corollary~\ref{covariance heavy-tailed}
  may fail. The reasoning is the same as in Remark~\ref{r: boundedness}:
  for an isotropic distribution which is highly concentrated at the origin, 
  the sample covariance matrix will likely equal $0$. 
  
  Still, one can weaken the boundedness assumption using Theorem~\ref{heavy-tailed rows exp si non-isotropic}
  instead of Theorem~\ref{heavy-tailed rows non-isotropic} in the proof of Corollary~\ref{covariance heavy-tailed}. 
  The weaker requirement is that $\E \max_{i \le N} \|X_i\|_2^2 \le m$ where $X_i$ denote
  the sample points. In this case, the covariance estimation will be guaranteed in expectation rather than
  with high probability; we leave the details for the interested reader.
  
  A different way to enforce the boundedness assumption is 
  to reject any sample points $X_i$ that fall outside the centered ball of radius $\sqrt{m}$. 
  This is equivalent to sampling from the conditional distribution inside the ball. 
  The conditional distribution satisfies the boundedness requirement,
  so the results discussed above provide a good covariance estimation for it. In many cases, this estimate
  works even for the original distribution -- namely, if only a small part of the 
  distribution lies outside the ball of radius $\sqrt{m}$. We leave the details for the interested reader;
  see e.g. \cite{V marginals}.
\end{remark}

\subsection{Applications to random sub-matrices and sub-frames}	
\index{Sampling from matrices and frames}	\label{s: sub-matrices}

The absence of any moment hypotheses on the distribution in Section~\ref{s: heavy-tailed rows} 
(except finite variance) makes these results especially relevant for discrete distributions.
One such situation arises when one wishes to sample entries or rows 
from a given matrix $B$, thereby creating a {\em random sub-matrix} $A$. 
It is a big program to understand what we can learn about 
$B$ by seeing $A$, see \cite{GMDL, DKM, RV sampling}.
In other words, we ask -- what properties of $B$ pass onto $A$?
Here we shall only scratch the surface of this problem: 
we notice that random sub-matrices of certain size preserve the property of being an {\em approximate isometry}.

\begin{corollary}[Random sub-matrices] \index{Sub-matrices}				\label{random sub-matrices}
  Consider an $M \times n$ matrix $B$ such that\footnote{The first hypothesis says 
  $B^*B = MI$. Equivalently, $\bar{B} := \frac{1}{\sqrt{M}}B$
  is an isometry, i.e. $\|\bar{B}x\|_2 = \|x\|_2$ for all $x$. Equivalently, the columns of $\bar{B}$ are orthonormal.}
  $\smin(B) = \smax(B) = \sqrt{M}$. 
  Let $m$ be such that all rows $B_i$ of $B$ satisfy $\|B_i\|_2 \le \sqrt{m}$.
  Let $A$ be an $N \times n$ matrix obtained by sampling $N$ random rows from $B$ 
  uniformly and independently. 
  Then for every $t \ge 0$, with probability at least $1 - 2n \cdot \exp(-ct^2)$ one has 
  $$
  \sqrt{N} - t \sqrt{m} \le \smin(A) \le \smax(A) \le \sqrt{N} + t \sqrt{m}.
  $$
  Here $c>0$ is an absolute constant.
\end{corollary}

\begin{proof}
By assumption, $I = \frac{1}{M} B^*B = \frac{1}{M} \sum_{i=1}^M B_i \otimes B_i$. 
Therefore, the uniform distribution on the set of the rows $\{B_1,\ldots,B_M\}$
is an isotropic distribution in $\R^n$. The conclusion then follows from Theorem~\ref{heavy-tailed rows}.  
\end{proof}

Note that the conclusion of Corollary~\ref{random sub-matrices}
does not depend on the dimension $M$ of the ambient matrix $B$. 
This happens because this result is a specific version of sampling from a discrete
isotropic distribution (uniform on the rows of $B$), where size $M$
of the support of the distribution is irrelevant. 

The hypothesis of Corollary~\ref{random sub-matrices} implies\footnote{To recall why this is true, 
take trace of both sides in the identity $I = \frac{1}{M} \sum_{i=1}^M B_i \otimes B_i$.}
that $\frac{1}{M} \sum_{i=1}^M \|B_i\|_2^2 = n$. 
Hence by Markov's inequality, most of the rows $B_i$ satisfy $\|B_i\|_2 = O(\sqrt{n})$. 
This indicates that Corollary~\ref{random sub-matrices} would be often used with $m = O(n)$.
Also, to ensure a positive probability of success, the useful magnitude of $t$ would be 
$t \sim \sqrt{\log n}$. With this in mind, the extremal singular values of $A$ will be close
to each other (and to $\sqrt{N}$) if $N \gg t^2 m \sim n \log n$.

Summarizing, Corollary~\ref{random sub-matrices} states that  
a random $O(n \log n) \times n$ sub-matrix of an $M \times n$ isometry
is an approximate isometry.\footnote{For the purposes of compressed sensing, 
we shall study the more difficult {\em uniform} 
problem for random sub-matrices in Section~\ref{s: restricted isometries}.
There $B$ itself will be chosen as a column sub-matrix of a given $M \times M$ matrix (such as DFT),
and one will need to control all such $B$ simultaneously, see Example~\ref{random measurements}.}

\medskip

Another application of random matrices with heavy-tailed isotropic rows is 
for {\em sampling from frames}. Recall that frames are generalizations of bases without 
linear independence, see Example~\ref{random vectors}. 
Consider a tight frame $\{u_i\}_{i=1}^M$ in $\R^n$, and for the sake of convenient 
normalization, assume that it has bounds $A=B=M$. 
We are interested in whether a small random subset of $\{u_i\}_{i=1}^M$ is still a nice frame in $\R^n$.
Such question arises naturally because frames are used in signal processing to
create {\em redundant representations} of signals. Indeed, every signal $x \in \R^n$ admits frame expansion 
$x = \frac{1}{M} \sum_{i=1}^M \< u_i, x\> u_i$. Redundancy makes 
frame representations more robust to errors and losses than basis representations. 
Indeed, we will show that if one loses all except $N = O(n \log n)$ random coefficients 
$\< u_i, x\> $ one is still able to reconstruct $x$ from the received coefficients $\< u_{i_k}, x\> $
as $x \approx \frac{1}{N} \sum_{k=1}^N \< u_{i_k}, x\> u_{i_k}$.
This boils down to showing that a random subset of size $N = O(n \log n)$ of a tight frame in $\R^n$
is an approximate tight frame. 

\begin{corollary}[Random sub-frames, see \cite{V frames}]	\index{Frames}			\label{random sub-frames}
  Consider a tight frame $\{u_i\}_{i=1}^M$ in $\R^n$ with frame bounds $A=B=M$. 
  Let number $m$ be such that all frame elements satisfy $\|u_i\|_2 \le \sqrt{m}$.
  Let $\{v_i\}_{i=1}^N$ be a set of vectors obtained by sampling $N$ random elements
  from the frame $\{u_i\}_{i=1}^M$ uniformly and independently.
  Let $\e \in (0,1)$ and $t \ge 1$. 
  Then the following holds with probability at least $1 - 2n^{-t^2}$:
  $$
  \text{If } N \ge C(t/\e)^2 m \log n 
  \quad \text{then $\{v_i\}_{i=1}^N$ is a frame in $\R^n$} 
  $$
  with bounds $A = (1-\e)N$, $B = (1+\e)N$. 
  Here $C$ is an absolute constant. 
  
  In particular, if this event holds, then every $x \in \R^n$ admits an approximate
  representation using only the sampled frame elements: 
  $$
  \Big\| \frac{1}{N} \sum_{i=1}^N \< v_i, x\> v_i - x \Big\| \le \e \|x\|.
  $$
\end{corollary}

\begin{proof}
The assumption implies that $I = \frac{1}{M} \sum_{i=1}^M u_i \otimes u_i$. 
Therefore, the uniform distribution on the set $\{u_i\}_{i=1}^M$
is an isotropic distribution in $\R^n$. 
Applying Corollary~\ref{covariance heavy-tailed} with $\Sigma = I$ and 
$\Sigma_N = \frac{1}{N} \sum_{i=1}^N v_i \otimes v_i$ we conclude that 
$\|\Sigma_N - I\| \le \e$ with the required probability. This clearly completes the proof. 
\end{proof}

As before, we note that $\frac{1}{M} \sum_{i=1}^M \|u_i\|_2^2 = n$, so 
Corollary~\ref{random sub-frames} would be often used with $m = O(n)$. 
This shows, liberally speaking, that a random subset of a frame in $\R^n$
of size $N = O(n \log n)$ is again a frame. 

\begin{remark}[Non-uniform sampling]
  The boundedness assumption $\|u_i\|_2 \le \sqrt{m}$, 
  although needed in Corollary~\ref{random sub-frames}, 
  can be removed by non-uniform sampling. 
  To this end, one would sample from the set of normalized vectors $\bar{u}_i := \sqrt{n} \frac{u_i}{\|u_i\|_2}$
  with probabilities proportional to $\|u_i\|_2^2$. 
  This defines an isotropic distribution in $\R^n$, and clearly $\|\bar{u}_i\|_2 = \sqrt{n}$.
  Therefore, by Theorem~\ref{random sub-frames}, a random sample of $N = O(n \log n)$ vectors
  obtained this way forms an almost tight frame in $\R^n$. 
  This result does not require any bound on $\|u_i\|_2$.
\end{remark}

\section{Random matrices with independent columns}				\label{s: columns}

In this section we study the extreme singular values of $N \times n$ random matrices $A$ 
with independent columns $A_j$. We are guided by our ideal bounds \eqref{heuristic} as before. 
The same phenomenon occurs in the column independent model as in the row independent
model -- sufficiently tall random matrices $A$ are approximate isometries.
As before, being tall will mean $N \gg n$ for sub-gaussian distributions 
and $N \gg n \log n$ for arbitrary distributions. 

The problem 
is equivalent to studying {\em Gram matrices} $G = A^*A = (\< A_j, A_k\> )_{j,k=1}^n$ \index{Gram matrix}
of independent isotropic random vectors $A_1,\ldots, A_n$ in $\R^N$.
Our results can be interpreted using Lemma~\ref{approximate isometries} as showing that 
the normalized Gram matrix $\frac{1}{N} G$ is an {\em approximate identity}
for $N, n$ as above. 

Let us first try to prove this with a heuristic argument.
By Lemma~\ref{norm isotropic} we know that the diagonal entries of $\frac{1}{N} G$
have mean $\frac{1}{N} \E \|A_j\|_2^2 = 1$ and off-diagonal ones have zero mean and 
standard deviation $\frac{1}{N} (\E \< A_j, A_k\> ^2)^{1/2} = \frac{1}{\sqrt{N}}$.
If, hypothetically, the off-diagonal entries were independent, then we could use the 
results of matrices with independent entries (or even rows) developed in Section~\ref{s: rows}. 
The off-diagonal part of $\frac{1}{N}G$ would have norm $O(\sqrt{\frac{n}{N}})$
while the diagonal part would approximately  equal $I$. Hence we would have
\begin{equation}										\label{Gram ideal}
\big\| \frac{1}{N} G - I \big\| = O \Big( \sqrt{\frac{n}{N}} \Big),
\end{equation}
i.e. $\frac{1}{N} G$ is an approximate identity 
for $N \gg n$. Equivalently, by Lemma~\ref{approximate isometries}, \eqref{Gram ideal}
would yield the ideal bounds \eqref{heuristic} on the extreme singular values of $A$.

Unfortunately, the entries of the Gram matrix $G$ are obviously not independent. 
To overcome this obstacle we shall use the {\em decoupling} technique 
of probability theory \cite{dG}.
We observe that there is still enough independence encoded in $G$. Consider a 
principal sub-matrix $(A_S)^*(A_T)$ of $G = A^*A$ with disjoint index sets $S$ and $T$.
If we condition on $(A_k)_{k \in T}$ then this sub-matrix has independent rows.
Using an elementary decoupling technique, we will indeed seek to replace the full Gram matrix 
$G$ by one such decoupled $S \times T$ matrix with independent rows, 
and finish off by applying results of Section~\ref{s: rows}.

\medskip

By transposition one can try to reduce our problem to studying the $n \times N$ 
matrix $A^*$. It has independent rows and the same singular values as $A$, 
so one can apply results of Section~\ref{s: rows}.
The conclusion would be that, with high probability,
$$
\sqrt{n} - C \sqrt{N} \le \smin(A) \le \smax(A) \le \sqrt{n} + C \sqrt{N}.
$$
Such estimate is only good for {\em flat} matrices ($N \le n$).
For {\em tall} matrices ($N \ge n$) the lower bound would be trivial because of the (possibly large)
constant $C$. 
So, from now on we can focus on tall matrices ($N \ge n$) with independent columns.

\subsection{Sub-gaussian columns}								
 \index{Sub-gaussian!random matrices with independent columns}  \label{s: sub-gaussian columns}

Here we prove a version of Theorem~\ref{sub-gaussian rows} for matrices with independent columns. 

\begin{theorem}[Sub-gaussian columns]				\label{sub-gaussian columns}
  Let $A$ be an $N \times n$ matrix ($N \ge n$) whose columns $A_i$ are independent
  sub-gaussian isotropic random vectors in $\R^N$ with $\|A_j\|_2 = \sqrt{N}$ a. s. 
  Then for every $t \ge 0$, the inequality holds
  \begin{equation}										\label{smin smax columns}
  \sqrt{N} - C \sqrt{n} - t \le \smin(A) \le \smax(A) \le \sqrt{N} + C \sqrt{n} + t
  \end{equation}
  with probability at least $1 - 2\exp(-ct^2)$, 
  where $C = C'_K$, $c = c'_K > 0$ depend only on the subgaussian norm 
  $K = \max_j \|A_j\|_\psitwo$ of the columns. 
\end{theorem}

The only significant difference between Theorem~\ref{sub-gaussian rows} for independent rows
and Theorem~\ref{sub-gaussian columns} for independent columns is 
that the latter requires {\em normalization of columns}, $\|A_j\|_2 = \sqrt{N}$ almost surely.
Recall that by isotropy of $A_j$ (see Lemma~\ref{norm isotropic}) 
one always has $(\E\|A_j\|_2^2)^{1/2} = \sqrt{N}$,
but the normalization is a bit stronger requirement.
We will discuss this more after the proof of Theorem~\ref{sub-gaussian columns}. 

\begin{remark}[Gram matrices are an approximate identity]
  By Lemma~\ref{approximate isometries}, the conclusion of Theorem~\ref{sub-gaussian columns}
  is equivalent to 
  $$
  \big\| \frac{1}{N} A^*A - I \| \le C \sqrt{\frac{n}{N}} + \frac{t}{\sqrt{N}}
  $$
  with the same probability $1 - 2\exp(-ct^2)$. This establishes our ideal inequality \eqref{Gram ideal}.
  In words, the normalized Gram matrix of $n$ independent sub-gaussian isotropic 
  random vectors in $\R^N$ is an approximate identity whenever $N \gg n$. 
\end{remark}

The proof of Theorem~\ref{sub-gaussian columns} is based on the decoupling technique \cite{dG}.
What we will need here is an elementary decoupling lemma for double arrays.
Its statement involves the notion of a {\em random subset} \index{Random subset} of a given finite set.
To be specific, we define a random set $T$ of $[n]$ with a given average size $m \in [0,n]$
as follows. Consider independent $\{0,1\}$ valued random variables $\d_1,\ldots,\d_n$
with $\E\d_i = m/n$; these are sometimes called {\em independent selectors}. \index{Selectors}
Then we define the random subset
$T = \{ i \in [n]:\; \d_i=1 \}$. Its average size equals $\E|T| = \E \sum_{i=1}^n \d_i = m$.

\begin{lemma}[Decoupling]		\index{Decoupling}					\label{decoupling}
  Consider a double array of real numbers $(a_{ij})_{i,j=1}^n$
  such that $a_{ii} = 0$ for all $i$. Then
  $$
  \sum_{i,j \in [n]} a_{ij} = 4 \E \sum_{i \in T,\, j \in T^c} a_{ij}
  $$
  where $T$ is a random subset of $[n]$ with average size $n/2$.
  In particular,
  $$
  4 \min_{T \subseteq [n]} \sum_{i \in T,\, j \in T^c} a_{ij}
  \le  \sum_{i,j \in [n]} a_{ij}
  \le 4 \max_{T \subseteq [n]} \sum_{i \in T,\, j \in T^c} a_{ij}
  $$
  where the minimum and maximum are over all subsets $T$ of $[n]$.
\end{lemma}

\begin{proof}
Expressing the random subset as $T = \{ i \in [n]:\; \d_i=1 \}$
where $\d_i$ are independent selectors with $\E\d_i=1/2$, we see that
$$
\E \sum_{i \in T,\, j \in T^c} a_{ij}
= \E \sum_{i,j \in [n]} \d_i (1-\d_j) a_{ij}
= \frac{1}{4} \sum_{i,j \in [n]} a_{ij},
$$
where we used that $\E \d_i (1-\d_j) = \frac{1}{4}$ for $i \ne j$ and the assumption $a_{ii}=0$.
This proves the first part of the lemma. The second part follows trivially by estimating 
expectation by maximum and minimum.
\end{proof}

\begin{proof}[Proof of Theorem~\ref{sub-gaussian columns}]
{\bf Step 1: Reductions.}
Without loss of generality we can assume that the columns $A_i$ have zero mean. 
Indeed, multiplying each column $A_i$ by $\pm 1$ arbitrarily 
preserves the extreme singular values of $A$, the isotropy of $A_i$ and
the sub-gaussian norms of $A_i$. Therefore, by multiplying
$A_i$ by independent symmetric Bernoulli random variables we achieve that $A_i$
have zero mean. 

For $t = O(\sqrt{N})$ the conclusion of Theorem~\ref{sub-gaussian columns}
follows from Theorem~\ref{sub-gaussian rows} by transposition. 
Indeed, the $n \times N$ random matrix $A^*$ has independent rows, so for $t \ge 0$ we have
\begin{equation}							\label{smax smax star}
\smax(A) = \smax(A^*) \le \sqrt{n} + C_K \sqrt{N} + t
\end{equation}
with probability at least $1 - 2 \exp(-c_K t^2)$. 
Here $c_K > 0$ and we can obviously assume that $C_K \ge 1$. 
For $t \ge C_K\sqrt{N}$ it follows that 
$\smax(A) \le \sqrt{N} + \sqrt{n} + 2t$, which
yields the conclusion of Theorem~\ref{sub-gaussian columns} 
(the left hand side of \eqref{smin smax columns} being trivial). 
So, it suffices to prove the conclusion for
$t \le C_K \sqrt{N}$.
Let us fix such $t$.

It would be useful to have some a priori control of $\smax(A) = \|A\|$. 
We thus consider the desired event 
$$
\EE := \big\{ \smax(A) \le 3C_K \sqrt{N} \big\}.
$$
Since $3C_K \sqrt{N} \ge \sqrt{n} + C_K \sqrt{N} + t$, by \eqref{smax smax star} 
we see that $\EE$ is likely to occur:
\begin{equation}										\label{EEc}
\P(\EE^c) \le 2 \exp(-c_K t^2).
\end{equation}

{\bf Step 2: Approximation.}
This step is parallel to Step~1 in the proof of Theorem~\ref{sub-gaussian rows},
except now we shall choose $\e := \d$. 
This way we reduce our task to the following. 
Let $\NN$ be a $\frac{1}{4}$-net of the unit sphere $S^{n-1}$ such that 
$|\NN| \le 9^n$.
It suffices to show that with probability at least $1 - 2\exp(-c'_K t^2)$ one has
$$
\max_{x \in \NN} \Big| \frac{1}{N} \|Ax\|_2^2 - 1 \Big| \le  \frac{\d}{2},
\quad \text{where } \d = C \sqrt{\frac{n}{N}} + \frac{t}{\sqrt{N}}.
$$
By \eqref{EEc}, it is enough to show that the probability
\begin{equation}							\label{desired p}
p:= \P \Big\{ \max_{x \in \NN} \Big| \frac{1}{N} \|Ax\|_2^2 - 1 \Big| > \frac{\d}{2} 
  \text{ and } \EE \Big\} 
\end{equation}
satisfies $p \le 2\exp(-c_K'' t^2)$,
where $c_K''>0$ may depend only on $K$.

{\bf Step 3: Decoupling.}
As in the proof of Theorem~\ref{sub-gaussian rows}, we will obtain the required bound for a fixed 
$x \in \NN$ with high probability, and then take a union bound over $x$. 
So let us fix any $x = (x_1,\ldots,x_n) \in S^{n-1}$.  
We expand
\begin{equation}										\label{norm expansion}
\|Ax\|_2^2 
= \Big\| \sum_{j=1}^n x_j A_j \Big\|_2^2
= \sum_{j=1}^n x_j^2 \|A_j\|_2^2 + \sum_{j,k \in [n], \, j \ne k} x_j x_k \< A_j, A_k \> .
\end{equation}
Since $\|A_j\|_2^2 = N$ by assumption and $\|x\|_2 =1$, the first sum equals $N$. 
Therefore, subtracting $N$ from both sides and dividing by $N$, we obtain the bound 
$$
\Big| \frac{1}{N} \|Ax\|_2^2 - 1 \Big|
\le \Big| \frac{1}{N} \sum_{j,k \in [n], \, j \ne k} x_j x_k \< A_j, A_k \> \Big| .
$$
The sum in the right hand side is $\< G_0 x, x\> $ 
where $G_0$ is the off-diagonal part of the Gram matrix $G = A^*A$.
As we indicated in the beginning of Section~\ref{s: columns}, we are going to replace
$G_0$ by its decoupled version whose rows and columns are indexed by disjoint sets. 
This is achieved by Decoupling Lemma~\ref{decoupling}: we obtain 
$$
\Big| \frac{1}{N} \|Ax\|_2^2 - 1 \Big|
\le \frac{4}{N} \max_{T \subseteq [n]} |R_T(x)|, 
\quad \text{where }
R_T(x) = \sum_{j \in T, \, k \in T^c} x_j x_k \< A_j, A_k \> .
$$
We substitute this into \eqref{desired p} and 
take union bound over all choices of $x \in \NN$ and $T \subseteq [n]$. 
As we know, $|\NN| \le 9^n$, and there are $2^n$ subsets $T$ in $[n]$. This gives
\begin{align}										\label{p}
p 
&\le \P \Big\{ \max_{x \in \NN, \, T \subseteq [n]} |R_T(x)| > \frac{\d N}{8} \text{ and } \EE \Big\} \notag\\
&\le 9^n \cdot 2^n \cdot \max_{x \in \NN, \, T \subseteq [n]}
  \P \Big\{ |R_T(x)| > \frac{\d N}{8} \text{ and } \EE \Big\}.
\end{align}

{\bf Step 4: Conditioning and concentration.}
To estimate the probability in \eqref{p}, we fix a vector $x \in \NN$ 
and a subset $T \subseteq [n]$ and we condition on a realization of 
random vectors $(A_k)_{k \in T^c}$. We express
\begin{equation}										\label{RJ equivalent}
R_T(x) = \sum_{j \in T} x_j \langle A_j, z \rangle 
\quad \text{where } z = \sum_{k \in T^c} x_k A_k.
\end{equation}
Under our conditioning $z$ is a fixed vector, so $R_T(x)$ is a sum of independent random variables.
Moreover, if event $\EE$ holds then $z$ is nicely bounded:
\begin{equation}										\label{norm z}
\|z\|_2 \le \|A\| \|x\|_2 \le 3 C_K \sqrt{N}.
\end{equation}
If in turn \eqref{norm z} holds then the terms
$\< A_j, z\> $ in \eqref{RJ equivalent} are independent centered sub-gaussian 
random variables with $\|\< A_j, z\> \|_\psitwo \le 3 K C_K \sqrt{N}$.
By Lemma~\ref{rotation invariance}, their linear combination $R_T(x)$
is also a sub-gaussian random variable with
\begin{equation}							\label{RT sub-gaussian}
\|R_T(x)\|_\psitwo \le C_1 \Big( \sum_{j \in T} x_j^2 \|\< A_j, z\> \|_\psitwo^2 \Big)^{1/2} 
\le \widehat{C}_K \sqrt{N}
\end{equation}
where $\widehat{C}_K$ depends only on $K$.

We can summarize these observations as follows. Denoting the conditional probability 
by $\P_T = \P \{ \; \cdot \; | (A_k)_{k \in T^c} \}$
and the expectation with respect to $(A_k)_{k \in T^c}$ by $\E_{T^c}$,
we obtain by \eqref{norm z} and \eqref{RT sub-gaussian} that
\begin{align*}
\P &\Big\{ |R_T(x)| > \frac{\d N}{8} \text{ and } \EE \Big\}
 \le \E_{T^c}\P_T \Big\{ |R_T(x)| > \frac{\d N}{8} \text{ and } \|z\|_2 \le 3 C_K \sqrt{N} \Big\} \\
 &\le 2 \exp \Big[ -c_1 \Big( \frac{\d N/8}{\widehat{C}_K \sqrt{N}} \Big)^2 \Big]
   = 2 \exp \Big( -\frac{c_2 \d^2 N}{\widehat{C}_K^2} \Big) 
 \le 2 \exp \Big( -\frac{c_2 C^2 n}{\widehat{C}_K^2} - \frac{c_2 t^2}{\widehat{C}_K^2} \Big).
\end{align*}
The second inequality follows because $R_T(x)$ is a sub-gaussian random variable \eqref{RT sub-gaussian}
whose tail decay is given by \eqref{sub-gaussian tail}. Here $c_1,c_2>0$ are absolute constants. 
The last inequality follows from the definition of $\d$.
Substituting this into \eqref{p} and choosing $C$ sufficiently large (so that $\ln 36 \le c_2 C^2/\widehat{C}_K^2$),
we conclude that 
$$
p \le 2 \exp \big( - c_2 t^2/\widehat{C}_K^2 \big).
$$
This proves an estimate that we desired in Step 2. The proof is complete.
\end{proof}

\begin{remark}[Normalization assumption]
Some a priori control of the norms of the columns $\|A_j\|_2$ is necessary 
for estimating the extreme singular values, since
$$
\smin(A) \le \min_{i \le n} \|A_j\|_2 \le \max_{i \le n} \|A_j\|_2 \le \smax(A).
$$
With this in mind, it is easy to construct an example showing that a normalization assumption $\|A_i\|_2 = \sqrt{N}$
is essential in Theorem~\ref{sub-gaussian columns}; it can not even be replaced by a boundedness
assumption $\|A_i\|_2 = O(\sqrt{N})$.
 
Indeed, consider a random vector $X = \sqrt{2} \xi Y$ in $\R^N$
where $\xi$ is a $\{0,1\}$-valued random variable with $\E \xi = 1/2$ (a ``selector'') 
and $X$ is an independent spherical random 
vector in $\R^n$ (see Example~\ref{random vectors sub-gaussian}). 
Let $A$ be a random matrix whose columns $A_j$ are independent copies of $X$. 
Then $A_j$ are independent centered sub-gaussian isotropic random vectors in $\R^n$ with $\|A_j\|_\psitwo = O(1)$
and $\|A_j\|_2 \le \sqrt{2N}$ a.s. 
So all assumptions of Theorem~\ref{sub-gaussian columns} except normalization are satisfied.  
On the other hand $\P\{X=0\}=1/2$, so matrix $A$ has a zero column with overwhelming probability $1 - 2^{-n}$.
This implies that $\smin(A)=0$ with this probability, so the lower estimate
in \eqref{smin smax columns} is false for all nontrivial $N,n,t$.

\end{remark}

\subsection{Heavy-tailed columns}
\index{Heavy-tailed!random matrices with independent columns}

Here we prove a version of Theorem~\ref{heavy-tailed rows exp si} for independent heavy-tailed
columns.

We thus consider $N \times n$ random matrices $A$ with independent columns $A_j$. 
In addition to the normalization assumption $\|A_j\|_2 = \sqrt{N}$ already present in
Theorem~\ref{sub-gaussian columns} for sub-gaussian columns, our new result must also 
require an a priori control of the off-diagonal part of the Gram matrix 
$G = A^*A = (\< A_j, A_k \> )_{j,k=1}^n$.

\begin{theorem}[Heavy-tailed columns]					\label{heavy-tailed columns}
  Let $A$ be an $N \times n$ matrix ($N \ge n$) whose columns $A_j$ are independent
  isotropic random vectors in $\R^N$ with $\|A_j\|_2 = \sqrt{N}$ a. s. 
  Consider the incoherence parameter
  $$ 
  m := \frac{1}{N} \E \max_{j \le n} \sum_{k \in [n], \, k \ne j} \< A_j, A_k \> ^2.
  $$
  Then $\E \big\| \frac{1}{N} A^*A - I \big\| \le C_0 \sqrt{ \frac{m \log n}{N}}$. In particular,
  \begin{equation}							\label{eq heavy-tailed columns}
  \E \max_{j \le n} |s_j(A) - \sqrt{N}|
  \le C \sqrt{m \log n}.
  \end{equation}
\end{theorem}

Let us briefly clarify the role of the incoherence parameter $m$, which 
controls the lengths of the rows of the off-diagonal part of $G$. 
After the proof we will see that a control of $m$ is essential in Theorem~\ref{heavy-tailed rows}. 
But for now, let us get a feel of the typical size of $m$. 
We have $\E \< A_j, A_k\> ^2 = N$ by Lemma~\ref{norm isotropic}, so for every row $j$ we see that 
$\frac{1}{N} \sum_{k \in [n], \, k \ne j} \< A_j, A_k \> ^2 = n-1$.
This indicates that Theorem~\ref{heavy-tailed columns} would be often used
with $m = O(n)$. 

In this case, Theorem~\ref{heavy-tailed rows} establishes our ideal inequality \eqref{Gram ideal}
up to a logarithmic factor. 
In words, the normalized Gram matrix \index{Gram matrix} of $n$ independent isotropic 
random vectors in $\R^N$ is an approximate identity whenever $N \gg n \log n$. 

\medskip

Our proof of Theorem~\ref{heavy-tailed columns} will be based on decoupling, symmetrization 
and an application of Theorem~\ref{heavy-tailed rows exp si non-isotropic} for a decoupled
Gram matrix with independent rows. 
The decoupling is done similarly to Theorem~\ref{sub-gaussian columns}. 
However, this time we will benefit from formalizing the decoupling inequality for Gram matrices:

\begin{lemma}[Matrix decoupling]	\index{Decoupling}					\label{matrix decoupling} 
  Let $B$ be a $N \times n$ random matrix whose columns $B_j$ satisfy $\|B_j\|_2 = 1$. 
  Then
  $$
  \E \|B^*B-I\| \le 4 \max_{T \subseteq [n]} \E \|(B_T)^* B_{T^c} \|.
  $$
\end{lemma}

\begin{proof}
We first note that $\|B^*B-I\| = \sup_{x \in S^{n-1}} \big| \|Bx\|_2^2 - 1 \big|$.
We fix $x = (x_1,\ldots,x_n) \in S^{n-1}$ and, expanding as in \eqref{norm expansion}, observe that
$$
\|Bx\|_2^2 = \sum_{j=1}^n x_j^2 \|B_j\|_2^2 + \sum_{j,k \in [n], \, j \ne k} x_j x_k \< B_j, B_k \> .
$$
The first sum equals $1$ since $\|B_j\|_2 = \|x\|_2 = 1$. 
So by Decoupling Lemma~\ref{decoupling}, a random subset $T$ of $[n]$ with average 
cardinality $n/2$ satisfies
$$
\|Bx\|_2^2 - 1 = 4 \E_T \sum_{j \in T, k \in T^c} x_j x_k \< B_j, B_k \> .
$$
Let us denote by $\E_T$ and $\E_B$ the expectations with respect to the random
set $T$ and the random matrix $B$ respectively. 
Using Jensen's inequality we obtain
\begin{align*}
\E_B \|B^*B-I\| 
&= \E_B \sup_{x \in S^{n-1}} \big| \|Bx\|_2^2 - 1 \big| \\
&\le 4 \E_B \E_T \sup_{x \in S^{n-1}} \Big| \sum_{j \in T, k \in T^c} x_j x_k \< B_j, B_k \> \Big| 
= 4 \E_T \E_B \|(B_T)^* B_{T^c} \|.
\end{align*}
The conclusion follows by replacing the expectation by the maximum over $T$. 
\end{proof}

\begin{proof}[Proof of Theorem~\ref{heavy-tailed columns}]
{\bf Step 1: Reductions and decoupling.}
It would be useful to have an a priori bound on $\smax(A) = \|A\|$.
We can obtain this by transposing $A$ and applying one of the results of Section~\ref{s: rows}. Indeed, the random 
$n \times N$ matrix $A^*$ has independent rows $A_i^*$ which by our assumption
are normalized as $\|A_i^*\|_2 = \|A_i\|_2 = \sqrt{N}$.
Applying Theorem~\ref{heavy-tailed rows exp si} with the roles of $n$ and $N$
switched, we obtain by the triangle inequality that 
\begin{equation}							\label{norm A}
\E \|A\| = \E \|A^*\| = \E \smax(A^*)
\le \sqrt{n} + C \sqrt{N \log n}
\le C \sqrt{N \log n}.
\end{equation}

Observe that $n \le m$ since by Lemma~\ref{norm isotropic} we have 
$\frac{1}{N} \E \< A_j, A_k \> ^2 =1$ for $j \ne k$.

We use Matrix Decoupling Lemma~\ref{matrix decoupling} for $B = \frac{1}{\sqrt{N}} A$ 
and obtain 
\begin{equation}								\label{E via Sigma}
E \le \frac{4}{N} \max_{T \subseteq [n]} \E \|(A_T)^* A_{T^c} \|
= \frac{4}{N} \max_{T \subseteq [n]} \E \|\Gamma\|
\end{equation}
where $\Gamma = \Gamma(T)$ denotes the decoupled Gram matrix 
$$
\Gamma = (A_T)^* A_{T^c} 
= \big( \< A_j, A_k\> \big)_{j \in T, k \in T^c}.
$$
Let us fix $T$; our problem then reduces to bounding the expected norm of $\Gamma$.

{\bf Step 2: The rows of the decoupled Gram matrix.}
For a subset $S \subseteq [n]$, we denote by $\E_{A_S}$ 
the conditional expectation given $A_{S^c}$, 
i.e. with respect to $A_S = (A_j)_{j \in S}$.
Hence $\E = \E_{A_{T^c}} \E_{A_T}$.

Let us condition on $A_{T^c}$.
Treating  $(A_k)_{k \in T^c}$ as fixed vectors we see that, 
conditionally, the random matrix $\Gamma$ has independent 
rows 
$$
\Gamma_j = \big( \< A_j, A_k\> \big)_{k \in T^c}, \quad j \in T.
$$
So we are going to use Theorem~\ref{heavy-tailed rows exp si non-isotropic} to 
bound the norm of $\Gamma$. To do this we need estimates on (a) the norms 
and (b) the second moment matrices of the rows $\Gamma_j$.

(a) Since for $j \in T$, $\Gamma_j$ is a random vector valued in $\R^{T^c}$, we estimate
its second moment matrix by choosing $x \in \R^{T^c}$ and evaluating the scalar second moment
\begin{align*}
\E_{A_T} \< \Gamma_j, x\> ^2
&= \E_{A_T} \Big( \sum_{k \in T^c} \< A_j, A_k\> x_k \Big)^2
= \E_{A_T} \Big\langle A_j, \sum_{k \in T^c} x_k A_k \Big\rangle ^2 \\
&= \Big\| \sum_{k \in T^c} x_k A_k \Big\|^2
= \|A_{T^c}x\|_2^2
\le \|A_{T^c}\|_2^2 \|x\|_2^2.
\end{align*}
In the third equality we used isotropy of $A_j$. 
Taking maximum over all $j \in T$ and $x \in \R^{T^c}$, 
we see that the second moment matrix 
$\Sigma(\Gamma_j) = \E_{A_T} \Gamma_j \otimes \Gamma_j$ satisfies 
\begin{equation}										\label{Sigma Gj}
  \max_{j \in T} \|\Sigma(\Gamma_j)\| \le \|A_{T^c}\|^2.
\end{equation}

(b) To evaluate the norms of $\Gamma_j$, $j \in T$, note that 
$\|\Gamma_j\|_2^2 = \sum_{k \in T^c} \< A_j, A_k\> ^2$. 
This is easy to bound, because the assumption says that the random variable 
$$
M := \frac{1}{N} \max_{j \in [n]} \sum_{k \in [n], \, k \ne j} \< A_j, A_k \> ^2
\quad \text{satisfies } \E M = m.
$$
This produces the bound $\E \max_{j \in T} \|\Gamma_j\|_2^2 \le N \cdot \E M = Nm$. But at this moment 
we need to work conditionally on $A_{T^c}$, so for now we will be satisfied with 
\begin{equation}										\label{rows of G}
\E_{A_T} \max_{j \in T} \|\Gamma_j\|_2^2 \le N \cdot \E_{A_T} M. 
\end{equation}

{\bf Step 3: The norm of the decoupled Gram matrix.}
We bound the norm of the random $T \times T^c$ Gram matrix $\Gamma$ 
with (conditionally) independent rows using Theorem~\ref{heavy-tailed rows exp si non-isotropic}
and Remark~\ref{r: different second moments}. 
Since by \eqref{Sigma Gj} we have
$\big\| \frac{1}{|T|} \sum_{j \in T} \Sigma(\Gamma_j) \big\| 
\le \frac{1}{|T|} \sum_{j \in T} \|\Sigma(\Gamma_j)\|
\le \|A_{T^c}\|^2$,  
we obtain using \eqref{rows of G} that
\begin{align}								\label{EAT Sigma}
\E_{A_T} \|\Gamma\| 
&\le (\E_{A_T} \|\Gamma\|^2)^{1/2} 
\le \|A_{T^c}\| \sqrt{|T|} + C \sqrt{N \cdot \E_{A_T} (M) \log |T^c|} \nonumber\\
&\le \|A_{T^c}\| \sqrt{n} + C \sqrt{N \cdot \E_{A_T} (M) \log n}.
\end{align}
Let us take expectation of both sides with respect to $A_{T^c}$.
The left side becomes the quantity we seek to bound, $\E \|\Gamma\|$.
The right side will contain the term which we can estimate by \eqref{norm A}:
$$
\E_{A_{T^c}} \|A_{T^c}\| = \E \|A_{T^c}\| \le \E \|A\| \le C \sqrt{N \log n}.
$$
The other term that will appear in the expectation of \eqref{EAT Sigma}
is 
$$
\E_{A_{T^c}} \sqrt{\E_{A_T} (M)}
\le \sqrt{\E_{A_{T^c}} \E_{A_T} (M)}
\le \sqrt{\E M}
= \sqrt{m}.
$$
So, taking the expectation in \eqref{EAT Sigma} and using these bounds, we obtain
$$
\E \|\Gamma\| 
= \E_{A_{T^c}} \E_{A_T} \|\Gamma\|
\le C \sqrt{N \log n} \sqrt{n} + C \sqrt{N m \log n}
\le 2C \sqrt{N m \log n}
$$ 
where we used that $n \le m$.
Finally, using this estimate in \eqref{E via Sigma} we conclude 
$$
E \le 8C \sqrt{\frac{m \log n}{N}}.
$$
This establishes the first part of Theorem~\ref{heavy-tailed columns}. 
The second part follow by the diagonalization argument
as in Step 2 of the proof of Theorem~\ref{heavy-tailed rows exp si}.
\end{proof}

\begin{remark}[Incoherence]
  A priori control on the {\em incoherence} \index{Incoherence} is essential in Theorem~\ref{heavy-tailed columns}.
  Consider for instance an $N \times n$ random matrix $A$ whose columns are independent coordinate random 
  vectors in $\R^N$. Clearly $\smax(A) \ge \max_j \|A_i\|_2 = \sqrt{N}$.
  On the other hand, if the matrix is not too tall, $n \gg \sqrt{N}$, then $A$ has two identical columns 
  with high probability, which yields $\smin(A)=0$. 
\end{remark}

\section{Restricted isometries}		\index{Restricted isometries}	\label{s: restricted isometries}

In this section we consider an application of the non-asymptotic random matrix theory
in compressed sensing. 
For a thorough introduction to compressed sensing, see the introductory chapter of
this book and \cite{FR, CS website}. 

In this area, $m \times n$ matrices $A$ 
are considered as measurement devices, taking as input a signal $x \in \R^n$ and returning 
its measurement $y = Ax \in \R^m$. One would like to take measurements economically, 
thus keeping $m$ as small as possible, and still to be able to recover the signal $x$ from its 
measurement $y$. 

The interesting regime for compressed sensing is where we take 
very few measurements, $m \ll n$. Such matrices $A$ are not one-to-one, 
so recovery of $x$ from $y$ is not possible for all signals $x$. 
But in practical applications, the amount of ``information'' contained in the signal is often small. 
Mathematically this is expressed as {\em sparsity} of $x$.
In the simplest case, one assumes that $x$ has few non-zero coordinates, say 
$|\supp(x)| \le k \ll n$. In this case, using any non-degenerate matrix $A$ one can check 
that $x$ can be recovered whenever $m > 2k$ using the optimization problem 
$\min \{ |\supp(x)|: \; Ax=y \}$.

This optimization problem is highly non-convex and generally NP-complete. 
So instead one considers a convex relaxation of this problem, $\min \{ \|x\|_1: \; Ax=y \}$.
A basic result in compressed sensing, due to Cand\`es and Tao \cite{Candes-Tao, Candes}, 
is that for sparse signals $|\supp(x)| \le k$,
the convex problem recovers the signal $x$ from its measurement $y$ exactly, provided 
that the measurement matrix $A$ is quantitatively non-degenerate. Precisely, the non-degeneracy  
of $A$ means that it satisfies the following {\em restricted isometry property} with $\d_{2k}(A) \le 0.1$.

\begin{definition-notag}[Restricted isometries] 
  An $m \times n$ matrix $A$ satisfies the {\em restricted isometry property} of order $k \ge 1$ if
  there exists $\d_k \ge 0$ such that the inequality
  \begin{equation}							\label{eq RIP}
  (1-\d_k) \|x\|_2^2 \le \|Ax\|_2^2 \le (1+\d_k) \|x\|_2^2
  \end{equation}
  holds for all $x \in \R^n$ with $|\supp(x)| \le k$.
  The smallest number $\d_k = \d_k(A)$ is called the {\em restricted isometry constant} of $A$. 
\end{definition-notag}

In words, $A$ has a restricted isometry property if $A$ acts as an approximate isometry 
on all sparse vectors. Clearly,
\begin{equation}							\label{equiv RIP}
\d_k(A) = \max_{|T| \le k} \|A_T^* A_T - I_{\R^T}\|
= \max_{|T| = \lfloor k \rfloor} \|A_T^* A_T - I_{\R^T}\|
\end{equation}
where the maximum is over all subsets $T \subseteq [n]$ with $|T| \le k$ or $|T| = \lfloor k \rfloor$.

The concept of restricted isometry can also be expressed via extreme singular values, 
which brings us to the topic we studied in the previous sections. $A$ is a restricted isometry 
if and only if all $m \times k$ sub-matrices $A_T$ of $A$
(obtained by selecting arbitrary $k$ columns from $A$) are approximate isometries. 
Indeed, for every $\d \ge 0$, Lemma~\ref{approximate isometries} shows that 
the following two inequalities are equivalent up to an absolute constant: 
\begin{gather}							
\d_k(A) \le \max(\d,\d^2);						\label{dk dd}		\\					
1-\d \le \smin(A_T) \le \smax(A_T) \le 1+\d		\label{s restricted}
\quad \text{for all } |T| \le k.
\end{gather}
More precisely, \eqref{dk dd} implies \eqref{s restricted} and
\eqref{s restricted} implies $\d_k(A) \le 3\max(\d,\d^2)$.

\medskip

Our goal is thus to find matrices that are good restricted isometries. What good means is clear
from the goals of compressed sensing described above. First, we need to keep 
the restricted isometry constant $\d_k(A)$ below some small absolute constant, say $0.1$. 
Most importantly, we would like the number of measurements $m$ to be small, ideally 
proportional to the sparsity $k \ll n$. 

This is where non-asymptotic random matrix theory enters. 
We shall indeed show that, with high probability, 
$m \times n$ random matrices $A$ are good restricted isometries of order $k$ 
with $m = O^*(k)$. Here the $O^*$ notation hides some logarithmic factors of $n$. 
Specifically, in Theorem~\ref{sub-gaussian RIP} we will show that
$$
m = O(k \log(n/k))
$$
for sub-gaussian random matrices $A$ (with independent rows or columns). 
This is due to the strong concentration properties of such matrices. 
A general observation of this kind is Proposition~\ref{concentration RIP}. 
It says that if for a given $x$, a random matrix $A$ (taken from any distribution) satisfies inequality \eqref{eq RIP}
with high probability, then $A$ is a good restricted isometry.

In Theorem~\ref{heavy-tailed RIP} we will extend these results to random matrices without concentration properties. 
Using a uniform extension of Rudelson's inequality, Corollary~\ref{Rudelson}, we shall show that
\begin{equation}										\label{m heavy-tailed}
m = O(k \log^4 n)
\end{equation}
for heavy-tailed random matrices $A$ (with independent rows). This includes the important
example of random Fourier matrices.

\subsection{Sub-gaussian restricted isometries} \index{Sub-gaussian!restricted isometries}

In this section we show that $m \times n$ sub-gaussian random matrices $A$ are good restricted isometries. 
We have in mind either of the following two models, which we analyzed in Sections~\ref{s: sub-gaussian rows}
and \ref{s: sub-gaussian columns} respectively:
\begin{description}
  \item[Row-independent model:] the rows of $A$ are independent 
    sub-gaussian isotropic random vectors in $\R^n$;
  \item[Column-independent model:] the columns $A_i$ of $A$ are independent 
    sub-gaussian isotropic random vectors in $\R^m$ with $\|A_i\|_2 = \sqrt{m}$ a.s.
\end{description}
Recall that these models cover many natural examples, including Gaussian and Bernoulli matrices 
(whose entries are independent standard normal or symmetric Bernoulli random variables), 
general sub-gaussian random matrices (whose entries are independent sub-gaussian random variables 
with mean zero and unit variance),
``column spherical'' matrices whose columns are independent vectors uniformly distributed on the centered
Euclidean sphere in $\R^m$ with radius $\sqrt{m}$, 
``row spherical'' matrices whose rows are independent vectors uniformly distributed on the centered
Euclidean sphere in $\R^d$ with radius $\sqrt{d}$, etc.

\begin{theorem}[Sub-gaussian restricted isometries]					\label{sub-gaussian RIP}
  Let $A$ be an $m \times n$ sub-gaussian random matrix with independent rows or columns, 
  which follows either of the two models above.   
  Then the normalized matrix $\bar{A} = \frac{1}{\sqrt{m}} A$ satisfies the following for every 
  sparsity level $1 \le k \le n$ and every number $\d \in (0,1)$:
  $$
  \text{if } m \ge C \d^{-2} k \log (en/k)
  \quad \text{then } \d_k(\bar{A}) \le \d
  $$
  with probability at least $1 - 2\exp(-c \d^2 m)$. 
  Here $C = C_K$, $c = c_K > 0$ depend only on the subgaussian norm 
  $K = \max_i \|A_i\|_\psitwo$ of the rows or columns of $A$. 
\end{theorem}

\begin{proof}
Let us check that the conclusion follows from Theorem~\ref{sub-gaussian rows} 
for the row-independent model, and from Theorem~\ref{sub-gaussian columns} 
for the column-independent model.
We shall control the restricted isometry constant using its equivalent description \eqref{equiv RIP}.
We can clearly assume that $k$ is a positive integer.

Let us fix a subset $T \subseteq [n]$, $|T| = k$ 
and consider the $m \times k$ random matrix $A_T$. If $A$ folows
the row-independent model, then the rows of $A_T$ are orthogonal projections of the rows of
$A$ onto $\R^T$, so they are still independent sub-gaussian
isotropic random vectors in $\R^T$. If alternatively, $A$ follows the column-independent 
model, then trivially the columns of $A_T$ satisfy the same assumptions as the columns of $A$.
In either case, Theorem~\ref{sub-gaussian rows} or Theorem~\ref{sub-gaussian columns} applies
to $A_T$. Hence for every $s \ge 0$, with probability at least $1-2\exp(-cs^2)$
one has
\begin{equation}							\label{AT smin smax}
\sqrt{m} - C_0\sqrt{k} - s \le \smin(A_T) \le \smax(A_T) \le \sqrt{m} + C_0\sqrt{k} + s.
\end{equation}
Using Lemma~\ref{approximate isometries} for
$\bar{A}_T = \frac{1}{\sqrt{m}}{A_T}$, we see that \eqref{AT smin smax} implies that
$$
\| \bar{A}_T^* \bar{A}_T - I_{\R^T} \| \le 3 \max(\d_0,\d_0^2)
\quad \text{where } \d_0 = C_0 \sqrt{\frac{k}{m}} + \frac{s}{\sqrt{m}}.
$$

Now we take a union bound over all subsets $T \subset [n]$, $|T| = k$. 
Since there are $\binom{n}{k} \le (en/k)^k$ ways to choose $T$, we conclude that 
$$
\max_{|T| = k} \| \bar{A}_T^* \bar{A}_T - I_{\R^T} \| \le 3 \max(\d_0,\d_0^2)
$$
with probability at least 
$1 - \binom{n}{k} \cdot 2\exp(-cs^2) 
\ge 1 - 2 \exp \big( k \log (en/k) - cs^2)$.
Then, once we choose $\e > 0$ arbitrarily and let 
$s = C_1 \sqrt{k \log(en/k)} + \e \sqrt{m}$, 
we conclude with probability at least $1 - 2\exp(-c \e^2 m)$ that 
$$
\d_k(\bar{A}) \le 3 \max(\d_0,\d_0^2) \quad \text{where } \d_0 = C_0 \sqrt{\frac{k}{m}} + C_1 \sqrt{\frac{k \log(en/k)}{m}} + \e.
$$

Finally, we apply this statement for $\e := \d/6$. By choosing constant $C$ in the statement of the theorem sufficiently large, 
we make $m$ large enough so that $\d_0 \le \d/3$, which yields $3 \max(\d_0,\d_0^2) \le \d$. The proof is complete.
\end{proof}

The main reason Theorem~\ref{sub-gaussian RIP} holds is that the random matrix $A$ 
has a strong concentration property, i.e. that $\|\bar{A}x\|_2 \approx \|x\|_2$ with high 
probability for every fixed sparse vector $x$. This concentration property alone implies
the restricted isometry property, regardless of the specific random matrix model:

\begin{proposition}[Concentration implies restricted isometry, see \cite{BDDW}]		\label{concentration RIP}
  Let $A$ be an $m \times n$ random matrix, and let $k \ge 1$, $\d \ge 0$, $\e > 0$.
  Assume that for every fixed $x \in \R^n$, $|\supp(x)| \le k$, the inequality
  $$
  (1-\d) \|x\|_2^2 \le \|Ax\|_2^2 \le (1+\d) \|x\|_2^2
  $$
  holds with probability at least $1 - \exp(-\e m)$.
  Then we have the following: 
  $$
  \text{if } m \ge C \e^{-1} k \log (en/k)
  \quad \text{then } \d_k(\bar{A}) \le 2\d
  $$
  with probability at least $1 - \exp(-\e m / 2)$. Here $C$ is an absolute constant.
\end{proposition}

In words, the restricted isometry property can be checked on each individual 
vector $x$ with high probability. 

\begin{proof}
We shall use the expression \eqref{equiv RIP} to estimate the restricted isometry constant.
We can clearly assume that $k$ is an integer, and focus on the sets 
$T \subseteq [n]$, $|T| = k$. 
By Lemma~\ref{net cardinality}, we can find a net $\NN_T$ of the unit sphere $S^{n-1} \cap \R^T$
with cardinality $|\NN_T| \le 9^k$. 
By Lemma~\ref{norm on net}, we estimate the operator norm as 
$$
\big\| A_T^* A_T - I_{\R^T} \big\|
\le 2 \max_{x \in \NN_T} \big| \big\langle (A_T^* A_T - I_{\R^T})x, x \big\rangle \big|
= 2 \max_{x \in \NN_T} \big| \|Ax\|_2^2 - 1 \big|.
$$
Taking maximum over all subsets $T \subseteq [n]$, $|T| = k$, we conclude that 
$$
\d_k(A) \le 2 \max_{|T| = k} \max_{x \in \NN_T} \big| \|Ax\|_2^2 - 1 \big|.
$$ 
On the other hand, by assumption we have for every $x \in \NN_T$ that 
$$
\P \big\{ \big| \|Ax\|_2^2 - 1 \big| > \d \big\} 
\le \exp(-\e m).
$$
Therefore, taking a union bound over $\binom{n}{k} \le (en/k)^k$ choices of the set $T$ 
and over $9^k$ elements $x \in \NN_T$, we obtain that
\begin{align*}
\P \{ \d_k(A) > 2\d \} 
&\le \binom{n}{k} 9^k \exp(-\e m)
\le \exp \big( k \ln(en/k) + k \ln 9 - \e m \big) \\
&\le \exp(- \e m / 2) 
\end{align*}
where the last line follows by the assumption on $m$. 
The proof is complete.
\end{proof}

\subsection{Heavy-tailed restricted isometries}	 \index{Heavy-tailed!restricted isometries}		\label{s: heavy-tailed RIP}

In this section we show that $m \times n$ random matrices $A$ with independent 
heavy-tailed rows (and uniformly bounded coefficients) are good restricted isometries.
This result will be established in Theorem~\ref{heavy-tailed RIP}. 
As before, we will prove this by controlling the extreme singular values of 
all $m \times k$ sub-matrices $A_T$. For each individual subset $T$, this can be achieved 
using Theorem~\ref{heavy-tailed rows}:
one has 
\begin{equation}										\label{AT individual control}
\sqrt{m} - t \sqrt{k} \le \smin(A_T) \le \smax(A_T) \le \sqrt{m} + t \sqrt{k}
\end{equation}
with probability at least $1 - 2 k \cdot \exp(-ct^2)$.
Although this optimal probability estimate has optimal order, it is too weak to allow for
a union bound over all $\binom{n}{k} = (O(1) n/k)^k$ choices of the subset $T$. 
Indeed, in order that $1 - \binom{n}{k} 2 k \cdot \exp(-ct^2) > 0$ 
one would need to take $t > \sqrt{k \log(n/k)}$. So in order to achieve a nontrivial 
lower bound in \eqref{AT individual control}, one would be forced to take $m \ge k^2$. 
This is too many measurements; recall that our hope is $m = O^*(k)$.

This observation suggests that instead of controlling each sub-matrix $A_T$ separately, 
we should learn how to control all $A_T$ at once. 
This is indeed possible with the following uniform version 
of Theorem~\ref{heavy-tailed rows exp si}:

\begin{theorem}[Heavy-tailed rows; uniform]					\label{heavy-tailed rows uniform}
  Let $A=(a_{ij})$ be an $N \times d$ matrix ($1 < N \le d$) whose rows $A_i$ are independent
  isotropic random vectors in $\R^d$. Let $K$ be a number such that
  all entries $|a_{ij}| \le K$ almost surely.
  Then for every $1 < n \le d$, we have
  $$
  \E \max_{|T| \le n} \max_{j \le |T|} |s_j(A_T) - \sqrt{N}|
  \le C l \sqrt{n}
  $$
  where $l = \log(n) \sqrt{\log d} \sqrt{\log N}$
  and where $C=C_K$ may depend on $K$ only.
  The maximum is, as usual, over all subsets $T \subseteq [d]$, $|T| \le n$.
\end{theorem}

The non-uniform prototype of this result, 
Theorem~\ref{heavy-tailed rows exp si}, 
was based on Rudelson's inequality, Corollary~\ref{Rudelson}.
In a very similar way, Theorem~\ref{heavy-tailed rows uniform}
is based on the following uniform version of Rudelon's inequality.

\begin{proposition}[Uniform Rudelson's inequality \cite{RV Fourier}]	 \index{Rudelson's inequality}	\label{RV Fourier}
  Let $x_1, \ldots, x_N$ be vectors in $\R^d$, $1 < N \le d$, 
  and let $K$ be a number such that
  all $\|x_i\|_\infty \le K$. 
  Let $\e_1, \ldots, \e_N$ be independent symmetric Bernoulli random variables. 
  Then for every $1 < n \le d$ one has
  $$
  \E \max_{|T| \le n} \Big\| \sum_{i=1}^N \e_i (x_i)_T \otimes (x_i)_T \Big\|
  \le C l \sqrt{n} \cdot \max_{|T| \le n} \Big\| \sum_{i=1}^N (x_i)_T \otimes (x_i)_T \Big\|^{1/2}
  $$ 
  where $l = \log(n) \sqrt{\log d} \sqrt{\log N}$
  and where $C = C_K$ may depend on $K$ only.
\end{proposition}

The non-uniform Rudelson's inequality (Corollary~\ref{Rudelson}) was a consequence of 
a non-commutative Khintchine inequality.
Unfortunately, there does not seem to exist a way to deduce Proposition~\ref{RV Fourier}
from any known result. 
Instead, this proposition is proved using Dudley's integral inequality for Gaussian processes and estimates
of covering numbers going back to Carl, see \cite{RV Fourier}. 
It is known however that such usage of Dudley's inequality is not optimal 
(see e.g. \cite{Ta book}). 
As a result, the logarithmic factors in Proposition~\ref{RV Fourier} are probably not optimal.
 
In contrast to these difficulties with Rudelson's inequality, proving uniform versions of 
the other two ingredients of Theorem~\ref{heavy-tailed rows exp si} --
the deviation Lemma~\ref{deviation from 1} and Symmetrization Lemma~\ref{symmetrization} --
is straightforward.

\begin{lemma}				\label{deviation from 1 uniform}
  Let $(Z_t)_{t \in \TT}$ be a stochastic process\footnote{A stochastic process $(Z_t)$ 
    is simply a collection of random variables 
    on a common probability space indexed by elements $t$ of some abstract set $\TT$.
    In our particular application, $\TT$ will consist of all subsets 
    $T \subseteq [d]$, $|T| \le n$.} 
  such that all $Z_t \ge 0$. Then 
  $
  \E \sup_{t \in \TT} |Z_t^2-1| \ge \max( \E \sup_{t \in \TT} |Z_t-1|, (\E \sup_{t \in \TT} |Z_t-1|)^2 ).
  $
\end{lemma}

\begin{proof}
The argument is entirely parallel to that of Lemma~\ref{deviation from 1}. 
\end{proof}

\begin{lemma}[Symmetrization for stochastic processes]	 \index{Symmetrization}		\label{symmetrization uniform}
  Let $X_{it}$, $1 \le i \le N$, $t \in \TT$, be random vectors valued in some Banach space $B$,
  where $\TT$ is a finite index set. Assume that the random vectors $X_i = (X_{ti})_{t \in \TT}$ 
  (valued in the product space $B^\TT$) are independent.
  Let $\e_1,\ldots, \e_N$ be independent symmetric Bernoulli random variables. 
  Then 
  $$ 
  \E \sup_{t \in \TT} \Big\| \sum_{i=1}^N (X_{it} - \E X_{it}) \Big\|
  \le 2 \E \sup_{t \in \TT} \Big\| \sum_{i=1}^N \e_i X_{it} \Big\|.
  $$
\end{lemma}

\begin{proof}
The conclusion follows from Lemma~\ref{symmetrization} applied to random vectors
$X_i$ valued in the product Banach space $B^\TT$ equipped with the norm
$||| (Z_t)_{t \in \TT} ||| = \sup_{t \in \TT} \|Z_t\|$.
The reader should also be able to prove the result directly, following the proof 
of Lemma~\ref{symmetrization}.
\end{proof}

\begin{proof}[Proof of Theorem~\ref{heavy-tailed rows uniform}]
Since the random vectors $A_i$ are isotropic in $\R^d$,
for every fixed subset $T \subseteq [d]$ the random vectors $(A_i)_T$ are
also isotropic in $\R^T$, so $\E (A_i)_T \otimes (A_i)_T = I_{\R^T}$.
As in the proof of Theorem~\ref{heavy-tailed rows exp si}, we are going to control
\begin{align*}
E 
:= \E \max_{|T| \le n} \big\| \frac{1}{N} A_T^* A_T - I_{\R^T} \big\| 
&= \E \max_{|T| \le n} \Big\| \frac{1}{N} \sum_{i=1}^N (A_i)_T \otimes (A_i)_T - I_{\R^T} \Big\| \\
&\le  \frac{2}{N} \, \E \max_{|T| \le n} \Big\| \sum_{i=1}^N \e_i (A_i)_T \otimes (A_i)_T \Big\|
\end{align*}
where we used Symmetrization Lemma~\ref{symmetrization uniform}
with independent symmetric Bernoulli random variables $\e_1,\ldots, \e_N$.
The expectation in the right hand side is taken both with respect 
to the random matrix $A$ and the signs $(\e_i)$. 
First taking the expectation with respect to $(\e_i)$
(conditionally on $A$) and afterwards the expectation with respect to $A$,
we obtain by Proposition~\ref{RV Fourier} that
$$
E \le \frac{C_K l \sqrt{n}}{N} \, 
  \E \max_{|T| \le n} \Big\| \sum_{i=1}^N (A_i)_T \otimes (A_i)_T \Big\|^{1/2}
=  \frac{C_K l \sqrt{n}}{\sqrt{N}} \, 
  \E \max_{|T| \le n} \big\| \frac{1}{N} A_T^* A_T \big\|^{1/2}
$$
By the triangle inequality,
$\E \max_{|T| \le n} \big\| \frac{1}{N} A_T^* A_T \big\| \le E + 1$.
Hence we obtain
$$
E \le C_K l \sqrt{\frac{n}{N}} (E+1)^{1/2}
$$
by H\"older's inequality.
Solving this inequality in $E$ we conclude that 
\begin{equation}										\label{A*A heavy-tailed exp}
E = \E \max_{|T| \le n} \big\| \frac{1}{N} A_T^* A_T - I_{\R^T} \big\|
\le \max(\d, \d^2)
\quad \text{where } \d = C_K l \sqrt{\frac{2n}{N}}.
\end{equation}
The proof is completed by a diagonalization argument similar to Step 2 in the 
proof of Theorem~\ref{heavy-tailed rows exp si}. One uses there a uniform version of 
deviation inequality given in Lemma~\ref{deviation from 1 uniform} for stochastic 
processes indexed by the sets $|T| \le n$. We leave the details to the reader. 
\end{proof}

\begin{theorem}[Heavy-tailed restricted isometries]					\label{heavy-tailed RIP}
  Let $A=(a_{ij})$ be an $m \times n$ matrix whose rows $A_i$ are independent
  isotropic random vectors in $\R^n$. Let $K$ be a number such that
  all entries $|a_{ij}| \le K$ almost surely.
  Then the normalized matrix $\bar{A} = \frac{1}{\sqrt{m}} A$ satisfies the following for $m \le n$, for every 
  sparsity level $1 < k \le n$ and every number $\d \in (0,1)$:
  \begin{equation}							\label{eq heavy-tailed RIP}
  \text{if } m \ge C \d^{-2} k \log n \log^2(k) \log(\d^{-2} k \log n \log^2 k)
  \quad \text{then } \E \d_k(\bar{A}) \le \d.
  \end{equation}
  Here $C = C_K > 0$ may depend only on $K$.  
\end{theorem}

\begin{proof}
The result follows from Theorem~\ref{heavy-tailed rows uniform}, more precisely
from its equivalent statement \eqref{A*A heavy-tailed exp}. In our notation, it says that 
$$
\E \d_k(\bar{A}) \le \max(\d,\d^2)
\quad \text{where } 
\d = C_K l \sqrt{\frac{k}{m}} = C_K \sqrt{\frac{k \log m}{m}} \log(k) \sqrt{\log n}.
$$
The conclusion of the theorem easily follows. 
\end{proof}

In the interesting sparsity range $k \ge \log n$ and $k \ge \d^{-2}$, the condition in 
Theorem~\ref{heavy-tailed RIP} clearly reduces to 
$$
m \ge C \d^{-2} k \log (n) \log^{3} k.
$$

\begin{remark}[Boundedness requirement]
  The {\em boundedness assumption} on the entries of $A$ is essential in
  Theorem~\ref{heavy-tailed RIP}. Indeed, if the rows of $A$ are independent
  coordinate vectors in $\R^n$, then $A$ necessarily has a zero column (in fact $n-m$ of them).
  This clearly contradicts the restricted isometry property.
\end{remark}

\begin{example}								\label{random measurements}
\begin{enumerate}
  \item {\bf (Random Fourier measurements):} \index{Fourier measurements}
    An important example for Theorem~\ref{heavy-tailed rows} 
    is where $A$ realizes random Fourier measurements. 
    Consider the $n \times n$ Discrete Fourier Transform (DFT) matrix $W$ with entries 
    $$
    W_{\omega,t} = \exp \Big( -\frac{2 \pi i \omega t}{n} \Big), 
    \quad \omega, t \in \{0,\ldots,n-1\}.
    $$
    Consider a random vector $X$ in $\C^n$ which picks a random row of $W$ (with uniform distribution).
    It follows from Parseval's inequality that $X$ is isotropic.\footnote{For convenience we have developed the theory over $\R$,
    while this example is over $\C$. As we noted earlier, all our definitions and results can be carried
    over to the complex numbers. So in this example we use the obvious complex versions of the notion of isotropy and 
    of Theorem~\ref{heavy-tailed RIP}.}
    Therefore the $m \times n$ random matrix $A$ whose rows are independent copies of $X$ 
    satisfies the assumptions of Theorem~\ref{heavy-tailed rows} with $K=1$.
    Algebraically, we can view $A$ as a {\em random row sub-matrix of the DFT matrix}.
    
    In compressed sensing, such matrix $A$ has a remarkable meaning -- it realizes 
    $m$ {\em random Fourier measurements} of a signal $x \in \R^n$. Indeed, $y=Ax$ is the DFT
    of $x$ evaluated at $m$ random points; in words, $y$ consists of $m$ random frequencies of $x$. 
    Recall that in compressed sensing, we would like to guarantee that with high probability 
    every sparse signal $x \in \R^n$ (say, $|\supp(x)| \le k$) 
    can be effectively recovered from its $m$ random frequencies $y=Ax$.
    Theorem~\ref{heavy-tailed RIP} together with Cand\`es-Tao's result (recalled in the beginning of 
    Section~\ref{s: restricted isometries}) imply that an exact recovery is given by the convex 
    optimization problem $\min\{ \|x\|_1 : Ax=y\}$ provided that we observe {\em slightly more frequencies 
    than the sparsity of a signal}: $m \gtrsim \ge C \d^{-2} k \log (n) \log^{3} k$.
    
  \item {\bf (Random sub-matrices of orthogonal matrices):} \index{Sub-matrices}
    In a similar way, Theorem~\ref{heavy-tailed RIP} applies to a random row sub-matrix $A$ 
    of an {\em arbitrary bounded orthogonal matrix} $W$. Precisely, $A$ may consist of  
    $m$ randomly chosen rows, uniformly 
    and without replacement,\footnote{Since in the interesting regime 
    very few rows are selected, $m \ll n$, sampling with or without replacement are formally equivalent. 
    For example, see \cite{RV Fourier} which deals with the model of sampling without replacement.}
    from an arbitrary $n \times n$ matrix $W = (w_{ij})$ such that $W^*W = n I$ 
    and with uniformly bounded coefficients, $\max_{ij} |w_{ij}| = O(1)$.
    The examples of such $W$ include the class of {\em Hadamard matrices} \index{Hadamard matrices}
    -- orthogonal matrices in which all entries equal $\pm 1$. 
\end{enumerate}
\end{example}

\section{Notes}					\label{s: notes}

\paragraph{For Section~\ref{s: introduction}}

We work with two kinds of moment assumptions for random matrices: sub-gaussian and heavy-tailed. 
These are the two extremes. By the central limit theorem, the sub-gaussian tail decay 
is the strongest condition one can demand from an isotropic distribution. In contrast, our heavy-tailed model
is completely general -- no moment assumptions (except the variance) are required.
It would be interesting to analyze random matrices with independent rows or columns in the 
intermediate regime, {\em between sub-gaussian and heavy-tailed} moment assumptions.
We hope that for distributions with an appropriate finite moment (say, $(2+\e)$th or $4$th), 
the results should be the same as for sub-gaussian distributions, i.e. no $\log n$ factors should occur. 
In particular, tall random matrices ($N \gg n)$ should still be approximate isometries. 
This indeed holds for sub-exponential distributions \cite{ALPT}; 
see \cite{V covariance} for an attempt to go down to finite moment assumptions.

\paragraph{For Section~\ref{s: preliminaries}}

The material presented here is well known. 
The volume argument presented in Lemma~\ref{net cardinality} is quite flexible. 
It easily generalizes to covering numbers of more general metric spaces, including convex bodies
in Banach spaces. See \cite[Lemma 4.16]{Pisier volume} and other parts of \cite{Pisier volume} for various 
methods to control covering numbers.

\paragraph{For Section~\ref{s: sub-gaussian}}

The concept of sub-gaussian random variables is due to Kahane~\cite{Kahane}.
His definition was based on the moment generating function 
(Property 4 in Lemma~\ref{sub-gaussian properties}),
which automatically required sub-gaussian random variables to be centered. 
We found it more convenient to use the equivalent Property 3 instead. 
The characterization of sub-gaussian random variables in terms of tail decay and moment growth
in Lemma~\ref{sub-gaussian properties} also goes back to \cite{Kahane}.

The rotation invariance of sub-gaussian random variables (Lemma~\ref{rotation invariance}) 
is an old observation \cite{BuKo}. Its consequence, Proposition~\ref{sub-gaussian large deviations},
is a general form of {\em Hoeffding's inequality}, which is usually stated for bounded random variables.
For more on large deviation inequalities, see also notes for Section~\ref{s: sub-exponential}.

Khintchine inequality is usually stated for the particular case of symmetric Bernoulli random variables. 
It can be extended for $0<p<2$ using a simple extrapolation argument based on H\"older's inequality,
see \cite[Lemma~4.1]{Ledoux-Talagrand}.

\paragraph{For Section~\ref{s: sub-exponential}}

Sub-gaussian and sub-exponential random variables can be studied together in a general framework. 
For a given exponent $0 < \alpha < \infty$, 
one defines general $\psi_\alpha$ random variables, 
those with moment growth $(\E |X|^p)^{1/p} = O(p^{1/\a})$. 
Sub-gaussian random variables correspond to $\a = 2$ and sub-exponentials 
to $\a = 1$. The reader is encouraged to extend 
the results of Sections~\ref{s: sub-gaussian} and \ref{s: sub-exponential} to this general class. 

Proposition~\ref{sub-exponential large deviations} is a form of {\em Bernstein's inequality},
which is usually stated for bounded random variables in the literature. 
These forms of Hoeffding's and Bernstein's inequalities 
(Propositions~\ref{sub-gaussian large deviations} and \ref{sub-exponential large deviations})
are partial cases of a large deviation inequality for general $\psi_\a$ norms, 
which can be found in \cite[Corollary~2.10]{Talagrand canonical} with a similar proof.
For a thorough introduction to large deviation inequalities for sums of independent random variables (and more),
see the books \cite{Petrov, Ledoux-Talagrand, Dembo-Zeitouni} and the tutorial \cite{BBL}.

\paragraph{For Section~\ref{s: isotropic}}

Sub-gaussian distributions in $\R^n$ are well studied in geometric functional analysis;
see \cite{MePaTJ reconstruction} for a link with compressed sensing. 
General $\psi_\alpha$ distributions in $\R^n$ are discussed e.g. in \cite{GiMi concentration}.

Isotropic distributions on convex bodies, and more generally isotropic log-concave distributions, 
are central to asymptotic convex geometry (see \cite{Giannopoulos isotropic, Paouris})
and computational geometry \cite{Vempala}. 
A completely different way in which isotropic distributions appear in convex geometry is from 
{\em John's decompositions} for contact points of convex bodies, see \cite{Ball, Rudelson contact, Vershynin John}.
Such distributions are finitely supported and therefore are usually heavy-tailed.

For an introduction to the concept of {\em frames} (Example~\ref{random vectors}), 
see \cite{KC, Christensen}.

\paragraph{For Section~\ref{s: sums matrices}}

The non-commutative Khintchine inequality, Theorem~\ref{non-commutative Khintchine},
was first proved by Lust-Piquard \cite{Lust-Piquard} with an unspecified constant $B_p$ 
in place of $C \sqrt{p}$.
The optimal value of $B_p$ was computed by Buchholz \cite{Buc01, Buc05}; 
see \cite[Section~6.5]{Rauhut structured} for an thorough introduction to Buchholz's argument.
For the complementary range $1 \le p \le 2$, a corresponding version of non-commutative Khintchine 
inequality was obtained by Lust-Piquard and Pisier \cite{Lu-Pi}. 
By a duality argument implicitly contained in \cite{Lu-Pi} and independently observed by Marius Junge, 
this latter inequality also implies the optimal order $B_p = O(\sqrt{p})$, 
see \cite{Rudelson isotropic} and \cite[Section~9.8]{Pisier operator}.

Rudelson's Corollary~\ref{Rudelson} was initially proved using a majorizing measure
technique; our proof follows Pisier's argument from \cite{Rudelson isotropic}
based on the non-commutative Khintchine inequality.

\paragraph{For Section~\ref{s: entries}}

The ``Bai-Yin law'' (Theorem~\ref{Bai-Yin}) was established for $\smax(A)$ by Geman \cite{Geman}
and Yin, Bai and Krishnaiah \cite{YBK}. 
The part for $\smin(A)$ is due to Silverstein \cite{Silverstein} for Gaussian random matrices.
Bai and Yin \cite{Bai-Yin} gave a unified treatment of both extreme singular values for general distributions.
The fourth moment assumption in Bai-Yin's law is known to be necessary \cite{Bai-Silverstein-Yin}.

Theorem~\ref{Gaussian} and its argument is due to Gordon \cite{Gordon 84, Gordon 85, Gordon 92}.
Our exposition of this result and of Corollary~\ref{Gaussian deviation} follows \cite{DS}.

Proposition~\ref{Gaussian concentration} is just a tip of an iceberg called {\em concentration of measure 
phenomenon}. \index{Concentration of measure}
We do not discuss it here because there are many excellent sources, some of which
were mentioned in Section~\ref{s: introduction}. Instead we give just one example related 
to Corollary~\ref{Gaussian deviation}.  
For a general random matrix $A$ with independent centered entries bounded by $1$,
one can use Talagrand's concentration inequality for convex
Lipschitz functions on the cube \cite{Tal1, Tal2}. 
Since $\smax(A)= \|A\|$ is a convex function of
$A$, Talagrand's concentration inequality implies
$\P \big\{ |\smax(A) - \Median(\smax(A))| \ge t \big\} \le 2 e^{-ct^2}$.
Although the precise value of the median may be unknown,
integration of this inequality shows
that $|\E \smax(A)-\Median(\smax(A))| \le C$.

For the recent developments related to the {\em hard edge} problem 
for almost square and square matrices (including Theorem~\ref{RV rectangular})
see the survey \cite{RV ICM}.

\paragraph{For Section~\ref{s: rows}}

Theorem~\ref{sub-gaussian rows} on random matrices with sub-gaussian rows,
as well as its proof by a covering argument, is a folklore in geometric functional analysis. 
The use of covering arguments in a similar context goes back to Milman's proof of Dvoretzky's theorem 
\cite{Milman Dvoretzky}; see e.g. \cite{Ball} and \cite[Chapter 4]{Pisier volume} for an introduction. 
In the more narrow context of extreme singular values of random matrices, 
this type of argument appears recently e.g. in \cite{ALPT}.

The breakthrough work on heavy-tailed isotropic distributions is due to Rudelson \cite{Rudelson isotropic}.
He used Corollary~\ref{Rudelson} in the way we described in the proof of 
Theorem~\ref{heavy-tailed rows exp si} to show that $\frac{1}{N} A^*A$ is an approximate isometry.
Probably Theorem~\ref{heavy-tailed rows} can also be deduced by a modification of this argument; 
however it is simpler to use the non-commutative Bernstein's inequality.

The symmetrization technique is well known. 
For a slightly more general two-sided inequality than Lemma~\ref{symmetrization},
see \cite[Lemma~6.3]{Ledoux-Talagrand}.

The problem of estimating covariance matrices described in Section~\ref{s: covariance} is 
a basic problem in statistics, see e.g. \cite{Johnstone}.
However, most work in the statistical literature is focused on the normal distribution 
or general product distributions (up to linear transformations), 
which corresponds to studying random matrices with independent entries. 
For non-product distributions, an interesting example
is for uniform distributions on convex sets \cite{KLS}. 
As we mentioned in Example~\ref{random vectors sub-gaussian}, such 
distributions are sub-exponential but not necessarily sub-gaussian, 
so Corollary~\ref{covariance sub-gaussian} does not apply. 
Still, the sample size $N = O(n)$ suffices to estimate the covariance matrix in this case 
\cite{ALPT}. It is conjectured that the same should hold for general distributions with 
finite (e. g. $4$th) moment assumption  \cite{V covariance}.

Corollary~\ref{random sub-matrices} on random sub-matrices is a variant of 
the Rudelson's result from \cite{Rudelson sub-matrices}. The study of random sub-matrices
was continued in \cite{RV sampling}. Random sub-frames were studied in \cite{V frames}
where a variant of Corollary~\ref{random sub-frames} was proved.

\paragraph{For Section~\ref{s: columns}}

Theorem~\ref{sub-gaussian columns} for sub-gaussian columns seems to be new. 
However, historically the efforts of geometric functional analysts were immediately focused on the more difficult case 
of sub-exponential tail decay (given by uniform distributions on convex bodies).
An indication to prove results like Theorem~\ref{sub-gaussian columns} 
by decoupling and covering is present in \cite{Bourgain}
and is followed in \cite{GiMi concentration, ALPT}.

The normalization condition $\|A_j\|_2 = \sqrt{N}$ in Theorem~\ref{sub-gaussian columns} 
can not be dropped but can be relaxed. Namely, consider the random variable
$\d := \max_{i \le n} \big| \frac{\|A_j\|_2^2}{N} - 1 \big|$.
Then the conclusion of Theorem~\ref{sub-gaussian columns} holds with \eqref{smin smax columns} replaced by 
$$
(1-\d)\sqrt{N} - C \sqrt{n} - t \le \smin(A) \le \smax(A) \le (1+\d)\sqrt{N} + C \sqrt{n} + t.
$$

Theorem~\ref{heavy-tailed columns} for heavy-tailed columns also seems to be new. 
The incoherence parameter $m$ 
is meant to prevent collisions of the columns of $A$ in a quantitative way.
It is not clear whether the {\em logarithmic factor} is needed in the conclusion
of Theorem~\ref{heavy-tailed columns}, or whether the incoherence parameter alone takes care of the
logarithmic factors whenever they appear. The same question can be raised for all other 
results for heavy-tailed matrices in Section~\ref{s: heavy-tailed rows} and their applications --
can we replace the logarithmic factors by more sensitive quantities (e.g. the logarithm 
of the incoherence parameter)?

\paragraph{For Section~\ref{s: restricted isometries}}

For a mathematical introduction to compressed sensing, see the introductory chapter of this book
and \cite{FR, CS website}. 

A version of Theorem~\ref{sub-gaussian RIP} was proved in \cite{MePaTJ} for the row-independent model; 
an extension from sub-gaussian to sub-exponential distributions is given in \cite{ALPT RIP}. 
A general framework of stochastic processes with sub-exponential tails is discussed 
in \cite{Mendelson}. 
For the column-independent model, Theorem~\ref{sub-gaussian RIP} seems to be new.

Proposition~\ref{concentration RIP} that formalizes a simple approach to restricted 
isometry property based on concentration is taken from \cite{BDDW}.
Like Theorem~\ref{sub-gaussian RIP}, it can also be used to show that 
Gaussian and Bernoulli random matrices are restricted isometries. 
Indeed, it is not difficult to check that these matrices satisfy a concentration inequality 
as required in Proposition~\ref{concentration RIP} \cite{Achlioptas}.

Section~\ref{s: heavy-tailed RIP} on heavy-tailed restricted isometries is 
an exposition of the results from \cite{RV Fourier}. 
Using concentration of measure techniques, one can prove a version of 
Theorem~\ref{heavy-tailed RIP} with high probability $1 - n^{-c \log^3 k}$ 
rather than in expectation \cite{Rauhut structured}.
Earlier, Candes and Tao \cite{Candes-Tao Fourier} proved a similar result
for random Fourier matrices, although with a slightly higher exponent
in the logarithm for the number of measurements in \eqref{m heavy-tailed},
$m = O(k \log^6 n)$.
The survey \cite{Rauhut structured} offers a thorough exposition of the material
presented in Section~\ref{s: heavy-tailed RIP} and more.

\printindex


\begin{thebibliography}{}

\bibitem{Achlioptas} Achlioptas, D. (2003).
  \newblock Database-friendly random projections: Johnson-Lindenstrauss with binary coins, 
  in: Special issue on PODS 2001 (Santa Barbara, CA).
  {\em J. Comput. System Sci.}, {\bf 66}, 671--687. 

\bibitem{ALPT} Adamczak, R., Litvak, A., Pajor, A., Tomczak-Jaegermann, N. (2010).
  \newblock Quantitative estimates of the convergence of the empirical covariance matrix 
  in log-concave ensembles,
  {\em J. Amer. Math. Soc.}, {\bf 23}, 535--561.
  
\bibitem{ALPT RIP} Adamczak, R., Litvak, A., Pajor, A., Tomczak-Jaegermann, N. (2010).
  \newblock Restricted isometry property of matrices with independent columns 
    and neighborly polytopes by random sampling, 
  {\em Constructive Approximation}, to appear.

\bibitem{AW} Ahlswede, R. and Winter, A. (2002).
  \newblock Strong converse for identification via quantum channels,
  {\em IEEE Trans. Inform. Theory}, {\bf 48}, 569--579.
  
\bibitem{AGZ} Anderson, G., Guionnet, A. and Zeitouni, O. (2009). 
  {\em An Introduction to Random Matrices}. 
  Cambridge: Cambridge University Press.  

\bibitem{Bai-Silverstein}  Bai, Z. and Silverstein, J. (2010).
  {\em Spectral analysis of large dimensional random matrices.} 
  Second edition. New York: Springer.

\bibitem{Bai-Silverstein-Yin} Bai, Z., Silverstein, J. and Yin, Y. (1988).
  \newblock A note on the largest eigenvalue of a large-dimensional sample covariance matrix, 
  {\em J. Multivariate Anal.}, {\bf 26}, 166--168. 

\bibitem{Bai-Yin} Bai, Z. and Yin, Y. (1993).
  \newblock Limit of the smallest eigenvalue of a large-dimensional sample covariance matrix,
  {\em Annals of Probability}, {\bf 21}, 1275--1294.

\bibitem{Ball} Ball, K. (1997).
  \newblock An elementary introduction to modern convex geometry,
  in {\em Flavors of geometry}, pp. 1--58.
  Math. Sci. Res. Inst. Publ., 31, 
  Cambridge: Cambridge University Press.

\bibitem{BDDW} Baraniuk, R., Davenport, M., DeVore, R. and Wakin, M. (2008).
  \newblock A simple proof of the restricted isometry property for random matrices,
  {\em Constructive Approximation}, {\bf 28}, 253--263.
  
\bibitem{BBL} Boucheron, S. Bousquet, O. and Lugosi, G. (2004).
  \newblock Concentration inequalities, 
  in {\em Advanced Lectures in Machine Learning}, edited by Bousquet, O., Luxburg, U. and R\"atsch, G. 
  Springer, pp. 208--240.
  
\bibitem{Bourgain} Bourgain, J. (1999).
  \newblock Random points in isotropic convex sets,
  in: {\em Convex geometric analysis (Berkeley, CA, 1996)},  pp. 53--58. 
  Math. Sci. Res. Inst. Publ., 34.
  Cambridge: Cambridge University Press.

\bibitem{Buc01} Buchholz, A. (2001). 
  \newblock Operator Khintchine inequality in non-commutative probability,
  {\em Math. Ann.}, {\bf 319}, 1--16.
  
\bibitem{Buc05} Buchholz, A. (2005). 
  \newblock Optimal constants in Khintchine type inequalities for fermions, 
    Rademachers and $q$-Gaussian operators,
  {\em Bull. Pol. Acad. Sci. Math.}, {\bf 53}, 315--321.

\bibitem{BuKo} Buldygin, V. V. and Kozachenko, Ju. V. (1980).
  \newblock Sub-Gaussian random variables, 
  {\em Ukrainian Mathematical Journal}, {\bf 32}, 483--489.

\bibitem{Candes} Cand\`es, E. 
  \newblock The restricted isometry property and its implications for compressed sensing, 
  {\em Compte Rendus de l'Academie des Sciences, Paris, Serie I}, {\bf 346}, 589--592.

\bibitem{Candes-Tao} Cand\`es, E. and Tao, T. (2005).
  \newblock Decoding by linear programming, 
  {\em IEEE Trans. Inform. Theory}, {\bf 51}, 4203--4215.

\bibitem{Candes-Tao Fourier} Cand\`es, E. and Tao, T. (2006).
  \newblock Near-optimal signal recovery from random projections: universal encoding strategies? 
  {\em IEEE Trans. Inform. Theory}, {\bf 52}, 5406--5425. 

\bibitem{Christensen} Christensen, O. (2008).
  \newblock {\em Frames and bases. An introductory course.} 
  Applied and Numerical Harmonic Analysis. 
  Boston, MA: Birkh\"auser Boston, Inc.

\bibitem{CS website} Compressive Sensing Resources, 
  \verb=http://dsp.rice.edu/cs=

\bibitem{DS} Davidson, K. R. and Szarek, S. J. (2001).
  \newblock Local operator theory, random matrices and Banach spaces,
  in {\em Handbook of the geometry of Banach spaces}, Vol. I, pp. 317--366.
  Amsterdam: North-Holland.

\bibitem{dG} de la Pe\~na, V. and Gin\'e, E. (1999).
  \newblock {\em Decoupling. From dependence to independence. Randomly stopped processes. 
  $U$-statistics and processes. Martingales and beyond.}
  New York: Springer-Verlag.
  
\bibitem{Deift-Gioev} Deift, P. and Gioev, D. (2009).
  \newblock {\em Random matrix theory: invariant ensembles and universality.} 
  Courant Lecture Notes in Mathematics, 18. 
  Courant Institute of Mathematical Sciences, New York; 
  Providence, RI: American Mathematical Society.

\bibitem{Dembo-Zeitouni} Dembo, A. and Zeitouni, O. (1993).
  \newblock {\em Large deviations techniques and applications.} 
  Boston, MA: Jones and Bartlett Publishers.

\bibitem{DKM} Drineas, P., Kannan, R. and Mahoney, M. (2006).
  \newblock Fast Monte Carlo algorithms for matrices. I, II III,  
  {\em SIAM J. Comput.}, {\bf 36} (2006), 132--206.

\bibitem{Durrett} Durrett, R. (2005).
  \newblock {\em Probability: theory and examples.}
  Belmont: Duxbury Press.
  
\bibitem{FeSo} Feldheim, O. and Sodin, S. (2008).
  \newblock A universality result for the smal lest eigenvalues of certain sample covariance matrices, 
  {\em Geometric and Functional Analysis}, 
  to appear.
  
\bibitem{FR} Fornasier, M. and Rauhut, H. (2010).
  \newblock Compressive Sensing,
  in {\em Handbook of Mathematical Methods in Imaging}, edited by Scherzer, O.
  Springer, to appear.

\bibitem{FuKo} F\"uredi, Z.; Koml\'os, J. (1981).
  The eigenvalues of random symmetric matrices,
  {\em Combinatorica}, {\bf 1}, 233--241. 
  
\bibitem{Geman} Geman, S. (1980).
  \newblock A limit theorem for the norm of random matrices,
  {\em Annals of Probability}, {\bf 8}, 252--261.

\bibitem{Giannopoulos isotropic} Giannopoulos, A. (2003).
  \newblock {\em Notes on isotropic convex bodies},
  Warsaw.

\bibitem{GiMi concentration} Giannopoulos, A. and Milman, V. (2000).
  \newblock Concentration property on probability spaces,
  {\em Advances in Mathematics}, {\bf 156}, 77--106. 

\bibitem{GiMi} Giannopoulos, A. and Milman, V. (2001). 
  \newblock Euclidean structure in finite dimensional normed spaces,
  in {\em Handbook of the geometry of Banach spaces}, Vol. I, pp. 707--779.
  Amsterdam: North-Holland.
  
\bibitem{GMDL} Golub, G., Mahoney, M., Drineas, P. and Lim, L.-H. (2006).
  \newblock Bridging the gap between numerical linear algebra, 
    theoretical computer science, and data applications,
  {\em SIAM News}, {\bf 9}, Number 8.
    
\bibitem{Gordon 84} Gordon, Y. (1984).
  \newblock On Dvoretzky's theorem and extensions of Slepian's lemma,
  in {\em Israel seminar on geometrical aspects of functional analysis (1983/84), II.}
  Tel Aviv: Tel Aviv University.

\bibitem{Gordon 85} Gordon, Y. (1985).
  \newblock Some inequalities for Gaussian processes and applications,
  {\em Israel Journal of Mathematics}, {\bf 50}, 265--289.

\bibitem{Gordon 92} Gordon, Y. (1992).
  \newblock Majorization of Gaussian processes and geometric applications,
  {\em Probab. Theory Related Fields}, {\bf 91}, 251--267.
  
\bibitem{Johnstone} Johnstone, I. (2001).
  \newblock On the distribution of the largest eigenvalue in principal components analysis,  
  {\em Ann. Statist.}, {\bf 29}, 295--327.

\bibitem{Kahane} Kahane, J.-P. (1960).
  \newblock Propri\'et\'es locales des fonctions \`a s\'eries de {F}ourier al\'eatoires,
  {\em Studia Mathematica}, {\bf 19}, 1--25.

\bibitem{KLS} Kannan, R., Lov\'asz, L. and Simonovits, M. (1995).
  \newblock Isoperimetric problems for convex bodies and a localization lemma,
  {\em Discrete Comput. Geom.}, {\bf 13}, 541--559.

\bibitem{KC} Kova\v{c}evi\'c, J. and Chebira, A. (2008). 
  \newblock {\em An Introduction to Frames}. 
  Foundations and Trends in Signal Processing. 
  Now Publishers.
 
 \bibitem{Latala} Latala, R. (2005).
  \newblock Some estimates of norms of random matrices,
  {\em Proc. Amer. Math. Soc.}, {\bf 133}, 1273-1282.

\bibitem{Ledoux} Ledoux, M. (2005).
  \newblock {\em The concentration of measure phenomenon.}
  Mathematical Surveys and Monographs, 89. 
  Providence: American Mathematical Society.
  
\bibitem{Ledoux extremal} Ledoux, M. (2007).
  \newblock Deviation inequalities on largest eigenvalues,
  in {\em Geometric aspects of functional analysis},  pp. 167--219. 
  Lecture Notes in Math., 1910. Berlin: Springer.
  
\bibitem{Ledoux-Talagrand} Ledoux, M. and Talagrand, M. (1991).
  \newblock {\em Probability in Banach spaces.}
  Berlin: Springer-Verlag.
  
\bibitem{LPRT} Litvak, A., Pajor, A., Rudelson, M. and Tomczak-Jaegermann, N. (2005).
  \newblock Smallest singular value of random matrices and geometry of random polytopes,
  {\em Adv. Math.}, {\bf 195}, 491--523.

\bibitem{Lu-Pi} Lust-Piquard, F. and Pisier, G. (1991).
  \newblock Noncommutative Khintchine and Paley inequalities,
  {\em Ark. Mat.}, {\bf 29}, 241--260. 
  
\bibitem{Lust-Piquard} Lust-Piquard, F. (1986).
  \newblock In\'egalit\'es de Khintchine dans $C_p (1<p<\infty)$,
  {\em C. R. Acad. Sci. Paris S\'er. I Math.}, {\bf 303}, 289--292. 
  
\bibitem{Matousek} Matou\v{s}ek, J. (2002).
  \newblock {\em Lectures on discrete geometry.}
  Graduate Texts in Mathematics, 212. 
  New York: Springer-Verlag.
  
\bibitem{Meckes} Meckes, M. (2004).
  \newblock Concentration of norms and eigenvalues of random matrices.  
  {\em J. Funct. Anal.}, {\bf 211}, 508--524.

\bibitem{Mehta} Mehta, M. L. (2004). 
  \newblock {\em Random matrices.} 
  Pure and Applied Mathematics (Amsterdam), 142. 
  Amsterdam: Elsevier/Academic Press.
  
\bibitem{Mendelson} Mendelson, S. (2008).
  \newblock On weakly bounded empirical processes,
  {\em Math. Ann.}, {\bf 340}, 293--314.

\bibitem{MePaTJ reconstruction} Mendelson, S., Pajor, A. and Tomczak-Jaegermann, N. (2007).
  \newblock Reconstruction and subgaussian operators in asymptotic geometric analysis,  
  {\em Geom. Funct. Anal.}, {\bf 17}, 1248--1282.

\bibitem{MePaTJ}Mendelson, S., Pajor, A. and Tomczak-Jaegermann, N. (2008).
  \newblock Uniform uncertainty principle for Bernoulli and subgaussian ensembles,
  {\em Constr. Approx.}, {\bf 28} (2008), 277--289. 
 
\bibitem {Milman Dvoretzky} Milman, V. D. (1974).
  \newblock A new proof of A. Dvoretzky's theorem on cross-sections of convex bodies. 
  {\em Funkcional. Anal. i Prilozhen.}, {\bf 5}, 28--37. 
  
\bibitem{MS} Milman, V. and Schechtman, G. (1986). 
  \newblock {\em Asymptotic theory of finite-dimensional normed spaces.
  With an appendix by M. Gromov.} 
  Lecture Notes in Mathematics, 1200. 
  Berlin: Springer-Verlag.
  
\bibitem{Paouris} Paouris, G. (2006).
  \newblock Concentration of mass on convex bodies,  
  {\em Geom. Funct. Anal.}, {\bf 16}, 1021--1049.
  
\bibitem{PeSo} P\'ech\'e, S. and Soshnikov, A. (2008).
  \newblock On the lower bound of the spectral norm of symmetric random matrices with independent entries,
  {\em Electron. Commun. Probab.}, {\bf 13}, 280--290. 

\bibitem{Petrov} Petrov, V. V. (1975).
  \newblock {\em Sums of independent random variables.} 
   New York-Heidelberg: Springer-Verlag.
   
\bibitem{Pisier volume} Pisier, G. (1989).
  \newblock {\em The volume of convex bodies and Banach space geometry.} 
  Cambridge Tracts in Mathematics, 94. 
  Cambridge: Cambridge University Press.

\bibitem{Pisier operator} Pisier, G. (2003).
  \newblock {\em Introduction to operator space theory.}
  London Mathematical Society Lecture Note Series, 294. 
  Cambridge: Cambridge University Press.
  
\bibitem{Rauhut structured} Rauhut, H. (2010). 
  \newblock Compressive sensing and structured random matrices,
  in {\em Theoretical Foundations and Numerical Methods for Sparse Recovery},  
  edited by Fornasier,~M. 
  Radon Series Comp. Appl. Math., Volume 9, pp. 1--92. deGruyter.

\bibitem{Rudelson contact} Rudelson, M. (1997).
  \newblock Contact points of convex bodies,
  {\em Israel Journal of Mathematics}, {\bf 101}, 93--124. 
  
\bibitem{Rudelson sub-matrices} Rudelson, M. (1999).
  \newblock Almost orthogonal submatrices of an orthogonal matrix,  
  {\em Israel J. of Math.}, {\bf 111}, 143--155.

\bibitem{Rudelson isotropic} Rudelson, M. (1999).
  \newblock Random vectors in the isotropic position, 
  {\em Journal of Functional Analysis}, {\bf 164}, 60--72. 
  
\bibitem{RV sampling} Rudelson, M. and Vershynin, R. (2007).
  \newblock Sampling from large matrices: an approach through geometric functional analysis,
  {\em J. ACM}, {\bf 54}, Art. 21, 19 pp. 
  
\bibitem{RV Fourier} Rudelson, M. and Vershynin, R. (2008).
  \newblock On sparse reconstruction from Fourier and Gaussian measurements,
  {\em Comm. Pure Appl. Math.}, {\bf 61}, 1025--1045. 

\bibitem{RV rectangular} Rudelson, M. and Vershynin, R. (2009).
  \newblock Smallest singular value of a random rectangular matrix,
  {\em Comm. Pure Appl. Math.}, {\bf 62}, 1707--1739.

\bibitem{RV ICM} Rudelson, M. and Vershynin, R. (2010).
  \newblock Non-asymptotic theory of random matrices: extreme singular values, 
  {\em Proceedings of the International Congress of Mathematicians}, 
  Hyderabad, India, to appear.

\bibitem{Silverstein} Silverstein, J. (1985).
  \newblock The smallest eigenvalue of a large-dimensional Wishart matrix,
  {\em Annals of Probability}, {\bf 13}, 1364--1368.
  
\bibitem{Soshnikov} Soshnikov, A. (2002). 
  \newblock A note on universality of the distribution of the largest eigenvalues in 
    certain sample covariance matrices, 
  {\em J. Statist. Phys.}, {\bf 108}, 1033--1056.

\bibitem{Talagrand canonical} Talagrand, M. (1994).
  \newblock The supremum of some canonical processes, 
  {\em American Journal of Mathematics}, {\bf 116}, 283--325.
  
\bibitem{Tal1} Talagrand, M. (1995).
  \newblock Concentration of measure and isoperimetric inequalities in product spaces,
  {\em Inst. Hautes \'Etudes Sci. Publ. Math.}, {\bf 81}, 73--205.

\bibitem{Tal2} Talagrand, M. (1996).
  \newblock A new look at independence,  
  {\em Annals of Probability}, {\bf 24}, 1--34.

\bibitem{Ta book} Talagrand, M. (2005).
  \newblock {\em The generic chaining. Upper and lower bounds of stochastic processes.}
  Springer Monographs in Mathematics. 
  Berlin: Springer-Verlag.
  
\bibitem{Tao-Vu survey} Tao, T. and Vu, V. (2009).
  \newblock From the Littlewood-Offord problem to the circular law: 
    universality of the spectral distribution of random matrices,
  {\em Bull. Amer. Math. Soc. (N.S.)}, {\bf 46}, 377--396.

\bibitem{Tropp} Tropp, J. (2010).
  \newblock User-friendly tail bounds for sums of random matrices,
  submitted.

\bibitem{Vempala} Vempala, S. (2005).
  \newblock Geometric random walks: a survey, 
  in {\em Combinatorial and computational geometry},  pp. 577--616. 
  Math. Sci. Res. Inst. Publ., 52.
  Cambridge: Cambridge University Press.

\bibitem{Vershynin John} Vershynin, R. (2001).
  \newblock John's decompositions: selecting a large part,
  {\em Israel Journal of Mathematics}, {\bf 122}, 253--277.

\bibitem{V frames} Vershynin, R. (2005).
  \newblock Frame expansions with erasures: an approach through the non-commutative operator theory,  
  {\em Appl. Comput. Harmon. Anal.}, {\bf 18}, 167--176.
  
\bibitem{V marginals} Vershynin, R. (2010).
  \newblock Approximating the moments of marginals of high-dimensional distributions, 
  {\em Annals of Probability}, to appear.


\bibitem{V covariance} Vershynin, R. (2010).
  \newblock How close is the sample covariance matrix to the actual covariance matrix?,
  {\em Journal of Theoretical Probability}, to appear. 
  
\bibitem{Vu} Vu, V. (2007).
  \newblock Spectral norm of random matrices,
  {\em Combinatorica}, {\bf 27}, 721--736.
  
\bibitem{YBK} Yin, Y. Q., Bai, Z. D. and Krishnaiah, P. R. (1998).
  \newblock On the limit of the largest eigenvalue of the large-dimensional sample covariance matrix,
  {\em Probab. Theory Related Fields}, {\bf 78}, 509--521.

\end{thebibliography}
\end{document}